\documentclass[10pt,reqno]{amsart}

\usepackage[all]{xy}
\usepackage{float, comment}
\usepackage{mathtools}
\usepackage{amsmath}
\usepackage{amsthm}
\usepackage{amssymb}
\usepackage{amsbsy}
\usepackage{amstext}
\usepackage{amsopn}
\usepackage{mathrsfs} 
\usepackage[mathscr]{eucal}
\usepackage{enumerate}
\usepackage{enumitem}
\usepackage{xcolor}
\usepackage{graphicx} 
\usepackage{scalerel}
\usepackage{microtype} 
\usepackage[margin=1in,marginparwidth=0.8in, marginparsep=0.1in]{geometry}
\usepackage[bookmarks=true, bookmarksopen=true, bookmarksdepth=3,bookmarksopenlevel=2, colorlinks=true, linkcolor=blue, citecolor=blue, filecolor=blue, menucolor=blue, urlcolor=blue]{hyperref}
\usepackage{tikz}
\usepackage{bbm}
\usepackage[all]{xy}
\usepackage{xspace}
\usetikzlibrary{3d,arrows,calc,positioning,decorations.pathreplacing,matrix} 
\usepackage{accents}
\usepackage{extarrows}

\DeclareMathOperator{\cHom}{\mathscr{H}\text{\kern -3pt {\calligra\large om}}\,}

\let\emptyset\varnothing

\numberwithin{equation}{subsection}     
\newtheorem{Theorem}{Theorem}[subsection]
\newtheorem{Proposition}[Theorem]{Proposition}
\newtheorem{Lemma}[Theorem]{Lemma}
\newtheorem{Claim}[Theorem]{Claim}
\newtheorem{Corollary}[Theorem]{Corollary}
\newtheorem{Conjecture}[Theorem]{Conjecture}
\newtheorem*{Theorem*}{Theorem}
\theoremstyle{definition}
\newtheorem{Remark}[Theorem]{Remark}
\newtheorem{Example}[Theorem]{Example}
\newtheorem{Definition}[Theorem]{Definition}
\newtheorem*{Remark*}{Theorem}

\numberwithin{figure}{section}

\newcommand{\Coh}{\mathrm{Coh}}
\newcommand{\QCoh}{\mathrm{QCoh}}

\def\D{\mathrm{D}}

\def\F{\mathrm{F}}

\def\I{\mathrm{I}}
\def\J{\mathrm{J}}
\def\K{\mathrm{K}}

\def\O{\mathrm{O}}
\def\P{\mathrm{P}}

\def\T{\mathrm{T}}

\def\W{\mathrm{W}}
\def\X{\mathrm{X}}

\def\Z{\mathrm{Z}}
\def\Z{\mathrm{Z}}

\def\v{\mathrm{v}}
\def\w{\mathrm{w}}
\def\u{\mathrm{u}}

\def\e{\mathrm{e}}
\def\x{\mathrm{x}}

\def\m{\mathrm{m}}

\def\bbC{\mathbb{C}}

\def\bbG{\mathbb{G}}

\def\bbN{\mathbb{N}}

\def\bbQ{\mathbb{Q}}
\def\bbR{\mathbb{R}}

\def\bbZ{\mathbb{Z}}

\def\scrA{\mathscr{A}}

\def\scrI{\mathscr{I}}

\def\scrN{\mathscr{N}}
\def\scrO{\mathscr{O}}

\def\scrQ{\mathscr{Q}}

\def\scrS{\mathscr{S}}
\def\scrT{\mathscr{T}}

\def\scrX{\mathscr{X}}

\def\scrZ{\mathscr{Z}}

\def\calO{\mathfrak{O}}

\def\frakR{\mathfrak{R}}
\def\frakS{\mathfrak{S}}

\def\calA{\mathcal{A}}
\def\calB{\mathcal{B}}
\def\calC{\mathcal{C}}
\def\calD{\mathcal{D}}
\def\calE{\mathcal{E}}
\def\calF{\mathcal{F}}
\def\calG{\mathcal{G}}
\def\calH{\mathcal{H}}
\def\calI{\mathcal{I}}

\def\calM{\mathcal{M}}
\def\calN{\mathcal{N}}
\def\calO{\mathcal{O}}
\def\calP{\mathcal{P}}

\def\calR{\mathcal{R}}
\def\calS{\mathcal{S}}
\def\calT{\mathcal{T}}

\def\calX{\mathcal{X}}

\def\calZ{\mathcal{Z}}

\def\frakg{\mathfrak{g}}

\def\bfa{\mathbf{a}}

\def\bfc{\mathbf{c}}
\def\bfd{\mathbf{d}}
\def\bfe{\mathbf{e}}

\def\bfm{\mathbf{m}}

\def\bfr{\mathbf{r}}

\def\bft{\mathbf{t}}

\def\bfw{\mathbf{w}}

\def\bfz{\mathbf{z}}

\def\bfA{\mathbf{A}}

\def\bfC{\mathbf{C}}
\def\bfD{\mathbf{D}}
\def\bfE{\mathbf{E}}
\def\bfF{\mathbf{F}}
\def\bfG{\mathbf{G}}
\def\bfH{\mathbf{H}}

\def\bfL{\mathbf{L}}

\def\bfR{{\mathbf{R}}}

\def\bfT{\mathbf{T}}
\def\bfU{\mathbf{U}}

\def\bfW{\mathbf{W}}

\def\bfZ{\mathbf{Z}}

\def\cw{{{\check w}}}
\def\cz{{{\check z}}}
\def\comega{{{\check \omega}}}
\def\calpha{{{\check\alpha}}}
\def\cbfw{{{\check\bfw}}}
\def\cbfz{{{\check\bfz}}}

\def\blambda{{{\pmb\lambda}}}

\def\bmu{{{\pmb\mu}}}
\def\bomega{{{\pmb\omega}}}
\def\balpha{{{\pmb\alpha}}}
\def\bcomega{{{\pmb{\check \omega}}}}
\def\bcalpha{{{\pmb{\check\alpha}}}}

\def\Irr{\operatorname{Irr}\nolimits}

\def\dim{\operatorname{dim}\nolimits}
\def\mod{\operatorname{mod}\nolimits}
\def\nilp{\operatorname{nil}\nolimits}
\def\unip{\operatorname{uni}\nolimits}

\def\Mod{\operatorname{Mod}\nolimits}
\def\Rep{\operatorname{Rep}\nolimits}
\def\Coh{\operatorname{Coh}\nolimits}

\def\Im{\operatorname{Im}\nolimits}
\def\Ker{\operatorname{Ker}\nolimits}

\def\det{{{\operatorname{det}}}}

\def\red{{\operatorname{red}\nolimits}}
\def\ind{\operatorname{Ind}\nolimits}

\def\id{\operatorname{id}\nolimits}

\def\Spec{\operatorname{Spec}\nolimits}
\def\Max{\operatorname{Max}\nolimits}

\def\Ad{\operatorname{Ad}\nolimits}

\def\wt{\operatorname{wt}\nolimits}

\def\ev{{ev}}

\def\nil{\operatorname{nil}\nolimits}

\def\ch{\operatorname{ch}\nolimits}
\def\c1{\operatorname{c_1}\nolimits}

\def\gr{\operatorname{gr}\nolimits}
\def\GL{\operatorname{GL}}
\def\-{{{\text{-}}}}
\DeclareMathOperator{\calHom}{\mathscr{H}\text{\kern -3pt {\calligra\large om}}\,}
\def\Fl{{{\calF\ell}}}
\def\dG{{{\dot G}}}

\def\dbG{{{\dot\bfG}}}
\def\dbH{{{\dot\bfH}}}
\def\dbT{{{\dot\bfT}}}
\def\tG{{{\widetilde G}}}

\def\tbG{{{\widetilde\bfG}}}
\def\tbH{{{\widetilde\bfH}}}
\def\tbT{{{\widetilde\bfT}}}
\usepackage{scalerel}
\newcommand\Wedge{\scalerel*{\bigwedge}{j}}
\def\ind{\operatorname{ind}\nolimits}
\def\bfzeta{{{\pmb\zeta}}}
\def\bGm{{{\pmb\bbG_m}}}
\def\res{{{\operatorname{res}}}}

\def\ab{{{\operatorname{ab}}}}
\def\iw{{{\operatorname{c}}}}
\def\NH{{{\operatorname{NH}}}}

\def\cl{{\rm{cl}}}

\def\r{{\bf{r}}}

\def\al{\alpha}
\def\ttimes{\widetilde{\times}}

\def\F{{\sf{F}}}

\def\HC{{\calH\calC}}
\def\loc{{\ell oc}}

\def\CAlg{{{\mathrm{CAlg}}}}
\def\-{{{\text{-}}}}
\newcommand{\colim}{\operatornamewithlimits{colim}} 
\def\ttimes{\widetilde{\times}}

\def\domega{{{\underline\omega}}}
\def\dalpha{{{\underline\alpha}}}
\def\dmu{{{\underline\mu}}}
\def\dlambda{{{\underline\lambda}}}

\def\dV{{{\underline V}}}
\def\dW{{{\underline W}}}
\def\dv{{{\underline v}}}
\def\dw{{{\underline w}}}
\def\dvv{{{\underline\v}}}
\def\dww{{{\underline\w}}}
\def\dI{{{\underline I}}}
\def\dE{{{\underline E}}}
\def\dG{{{\underline G}}}
\def\dP{{{\underline P}}}
\def\dQ{{{\underline Q}}}
\def\dC{{{\underline C}}}
\def\dfrakg{{{\underline\frakg}}}

\def\dii{{{\underline i}}}
\def\djj{{{\underline j}}}

\def\dxx{{{\underline\x}}}

\def\duu{{{\underline\u}}}

\newcommand{\UD}[1]{{\underline{#1}}}

\def\dPhi{{{\underline\Phi}}}

\def\O{{\mathcal O}}

\def\Gr{\mathrm{Gr}}

\def\calF{{\mathcal{F}}}
\def\calG{{\mathcal{G}}}

\def\cN{{\mathcal{N}}}

\DeclareMathOperator{\Hom}{Hom}

\DeclareMathOperator{\End}{End}

\newcommand{\llb}{[\![}
\newcommand{\rrb}{]\!]}
\newcommand{\llp}{(\!(}	
\newcommand{\rrp}{)\!)}

\begin{document}

\title[Shifted affine quantum groups and Coulomb branches]{Representations of shifted affine quantum groups and Coulomb branches}

\author{M. Varagnolo} 
\address{\scriptsize{M.V. : ~CY Cergy Paris Universit\'e,  95302 Cergy-Pontoise, France,
UMR8088 (CNRS), ANR-24-CE40-3389 (GRAW).}}
\author{E. Vasserot} 
\address{\scriptsize{E.V. :~Universit\'e Paris Cit\'e, 75013 Paris, France, UMR7586 (CNRS), 
ANR-24-CE40-3389 (GRAW).
}}
\subjclass[2010]{17B37, 19E08, 81R10}

\begin{abstract}
We study the integral category $\scrO$ of shifted affine quantum groups of non symmetric types.
To do this we compute  the K-theoretic analog of the Coulomb branches with symmetrizers introduced by Nakajima and Weekes.
This yields an equivalence of the category $\scrO$ with a module category over a new type of quiver Hecke algebras.
At the decategorified level, this establishes a connection between the Grothendieck group of $\scrO$ and a finite-dimensional module 
over a simple Lie algebra of unfolded symmetric type. 
We explicitly describe this module in a few examples and provide a combinatorial rule for its crystal.\end{abstract}

\maketitle
\setcounter{tocdepth}{2}
\tableofcontents

\section{Introduction}

\subsection{Presentation}

Given an arbitrary Cartan matrix $\bfc$, a mathematical definition of the Coulomb branch of a 3D, N=4 
quiver gauge theory associated with two $I$-graded vector spaces $V$ and $W$ was given by Nakajima and Weekes in \cite{NW23}.
It was proved in \cite{NW23} that the quantization of the Coulomb branch is a truncated shifted Yangian, and the fixed point 
set of some $\bbC^\times$-action on the space of triples (i.e., the BFN space) associated with the Coulomb branch 
was computed. Some consequences for the module category of the truncated shifted Yangian were also discussed.
A related construction in physics in the context of 4D, N=2 quiver gauge theory was considered by Kimura and Pestun 
in \cite{KP18}.

In this paper we consider the quantized K-theoretic Coulomb branch with symmetrizers of the 3D, N=4 
quiver gauge theory associated with the Cartan matrix $\bfc$. Let $\calA$ denote this algebra.
We relate $\calA$ to a truncated shifted quantum loop group of type $\bfc$, generalizing the work of Finkelberg-Tsymbaliuk \cite{FT19} 
in the symmetric case, as stated in Propositions \ref{prop:Phi} and \ref{prop:onto}.

This construction allows us to link the integral category $\scrO$ of the truncated shifted quantum loop group
to a category of smooth representations of a completion $\calA^\wedge$ of the algebra $\calA$, see Proposition \ref{prop:HC}.
We then prove a version of the Segal-Thomason localization theorem which relates
$\calA^\wedge$ to the K-theory and Borel-Moore homology of the fixed point subset of the $\bbC^\times$-action on the BFN 
space, as stated in Theorem \ref{thm:Main1}. The fixed point set is computed explicitly following \cite{NW23}.
Its K-theory can be described in terms of a new version of QH algebras (=quiver Hecke algebras), 
which we call an integral $\bbZ$-weighted QH algebras.
This yields an equivalence from the integral category $\scrO$ of the truncated shifted quantum loop group
to the category of nilpotent modules of the corresponding integral $\bbZ$-weighted QH algebra. 
The integral $\bbZ$-weighted QH algebra is attached to the symmetric Cartan matrix $\dC$ obtained by unfolding $\bfc$.
It also depends on a grading given by $\bfc$ (and not by $\dC$ only).
While the presence of this unfolding was already observed in \cite{NW23}, the role of the integral 
$\bbZ$-weighted QH algebra is new and important.
For symmetric $\bfc$ the integral $\bbZ$-weighted QH algebra coincides with the parity QH algebra of type $\bfc$
considered in \cite{KTWWY19b}.
For non symmetric $\bfc$, the definition of the integral $\bbZ$-weighted QH algebra differs from the definition of 
the parity QH algebra of $\dC$.

The integral $\bbZ$-weighted QH algebra allows us to decategorify the integral category $\scrO$ 
of the truncated shifted quantum loop group
 in term of a finite dimensional module $V$ over the simple Lie algebra whose Cartan matrix is $\dC$, as stated in Theorem \ref{thm:Main2}.
 In other words, the complexified Grothendieck group of $\scrO$ is canonically identified with $V$ in such a way that 
 the classes of simple objects are taken to the dual canonical basis, Theorem \ref{thm:Main2}.
This is due to the fact, proved in Proposition \ref{prop:TT}, 
that integral $\bbZ$-weighted QH algebras are Morita equivalent to subalgebras
of tensor product algebras (in Webster's sense).
This finite-dimensional module $V$ is not generally known. It may not be simple.
We provide a few conditions it satisfies and compute it in type $B_2$.
We also give a (partly conjectural) combinatorial rule to compute this representation.
This rule uses a crystal which generalizes Nakajima's monomial crystal.

Our construction has several important consequences:
\begin{enumerate}[label=$\mathrm{(\alph*)}$,leftmargin=8mm,itemsep=1mm]
\item[$\bullet$]
It allows us to define a monoidal structure on the integral truncated shifted category $\scrO$ which extends the fusion product of
loop highest weight modules considered in \cite{H23}.
\item[$\bullet$]
The integral $\bbZ$-weighted QH algebra has a 
cohomological grading, as it is a convolution algebra in Borel-Moore homology.
Consequently, our equivalence of categories yields a grading on the category $\scrO$.
\item[$\bullet$]
The decategorification of $\scrO$ above express the $q$-character of the simple modules of the integral category $\scrO$ in terms 
of the dual canonical basis of $V$, see Corollary \ref{cor:qchar} and Theorem \ref{thm:Main2}.
\end{enumerate}
We will return to these questions in an upcoming paper \cite{VV26}.
Note also that, although we focus on finite types, many of our results extend naturally to the case of symmetrizable 
generalized Cartan matrices. 

Another motivation for this work comes from \cite{VV23}, where we provide a geometrization of (shifted) quantum loop groups of arbitrary 
types via the critical K-theory of quiver varieties, generalizing Nakajima's work on symmetric types in \cite{N01}.
Quiver varieties are the 3D mirror duals of Coulomb branches. 
We aim to better understand the relationships between these two constructions.

\subsection{Plan of the paper}
The plan of the paper is the following.
Section 2 introduces integral $\bbZ$-weighted QH algebras, 
which form a new family of QH algebras associated with Coulomb branches 
with symmetrizers. First, in \S\ref{sec:CQH}, we describe QH algebras modeled over of $\bbZ$-weighted flag manifolds, i.e.,  
spaces of nested sequences of finite-dimensional vector spaces labeled by $\bbZ$.
We compare them with the tensor product algebras introduced by Webster in \S\ref{sec:TPA}.
Next, in \S\ref{sec:IZQH}, we fix a 
non-symmetric Cartan matrix $\bfc$
whose Dynkin diagram a  folded Dynkin diagram of a symmetric Cartan matrix $\dC$.
We introduce the integral $\bbZ$-weighted QH algebras  of type $\bfc$,
which are modeled on $\bbZ$-weighted QH algebras of type $\dC$,
with an additional integrality condition that generalizes the parity QH algebras from 
\cite{KTWWY19b} in the symmetric case. 
We then prove that the module categories of 
integral $\bbZ$-weighted QH algebras are quotients of the module categories of tensor product algebras of type $\dC$.
Let $\pmb\frakg$ be complex simple Lie algebra of $\bfc$,
and $\dfrakg$ the complex simple Lie algebra of $\dC$.
Subsequently, we decategorify the integral $\bbZ$-weighted QH algebras ${}^0\scrT_\bmu^\rho$ by weight subspaces
in some $\dfrakg$-modules. 
While we do not explicitly compute these modules, we discuss their connections to \cite{KTWWY19b} in the symmetric case 
and we compute them in certain specific scenarios, such as the generic case.

Section 3 provides an introduction to Coulomb branches of 4D, N=2
quiver gauge theories with symmetrizers.
First, we introduce the BFN space with symmetrizers $\calR$, following \cite{NW23}.
To facilitate the application of K-theory later, we employ a variation of the formalism from \cite{CWb}, \cite{CWb1}, and \cite{CWb2}, 
which uses ind-tamely presented $\infty$-stacks of ind-geometric type. The theory of these stacks is briefly reviewed in Appendix B.
We then describe the fixed point locus of certain automorphisms in Proposition \ref{prop:fixedlocus}.
Finally, we introduce the Coulomb branches of 4D, N=2 
quiver gauge theories with symmetrizers $\calA_{\mu,R}^\lambda$ in Definition \ref{def:AR},
and we present a few fundamental results.

Section 4 begins with an introduction to shifted quantum groups $\bfU_{\bmu,R}$ 
and their integral category $\scrO$, following \cite{H23}.
Next, we introduce truncated shifted quantum groups and their module category 
${}^0\scrO^\rho=\bigoplus_\bmu{}^0\scrO_\bmu^\rho$ in Definitions \ref{def:O1}, \ref{def:O2}, 
along with the surjective algebra homomorphism 
$\Phi:\bfU_{\bmu,R}\otimes R_{T_W}\to\calA_{\mu,R}^\lambda$
in Proposition \ref{prop:Phi} that maps
to Coulomb branches with symmetrizers.
We then prove a localization theorem for Coulomb branches, employing techniques similar to those in \cite{VV10}, 
which identify the Coulomb branch $\calA_{\mu}^\lambda$
with the integral $\bbZ$-weighted QH algebras ${}^0\widetilde\scrT_\bmu^\rho$
after suitable completions, as stated in Theorem \ref{thm:Main1}. 
For related discussions, see also \cite{W19b}, where similar background is explored.
In the final section, leveraging the localization theorem, we establish a connection between the truncated shifted category 
$\scrO$ and integral $\bbZ$-weighted QH algebras in Theorem \ref{thm:Main2}, 
and we discuss a few implications at the decategorified level.
The main result is the following (Theorems \ref{thm:Main1} and \ref{thm:Main2}).

\begin{Theorem*}\hfill
\begin{enumerate}[label=$\mathrm{(\alph*)}$,leftmargin=8mm,itemsep=1mm]
\item
 ${}^0\scrO_\bmu^\rho$ is equivalent to a category of nilpotent modules over 
the integral $\bbZ$-weighted QH algebra ${}^0\scrT_\bmu^\rho$.
\item
There is a representation of $\dfrakg$ in $K({}^0\scrO^\rho)$ and an embedding of $K({}^0\scrO^\rho)$
into a tensor product of fundamental modules of $\dfrakg$ which takes the simple modules into  the dual canonical basis. 
\end{enumerate}
\end{Theorem*}

The representation of $\dfrakg$ in the Grothendieck group $K({}^0\scrO^\rho)$ is not known in general. 
We define in Proposition \ref{prop:crystal} a crystal of type $\dC$
which is, conjecturally, isomorphic to the crystal of the $\dfrakg$-module $K({}^0\scrO^\rho)$.
This yields a combinatorial rule to compute the $\ell$-highest weight of all simple modules 
in ${}^0\scrO^\rho$ which holds true in type $B_2$, see \S\ref{sec:B2}.
We will come back to this in  \cite{VV26}.

The BFN space is infinite dimensional. We use the formalism of Cautis-Williams \cite{CWb} to represent it by an
ind-geometric derived $\infty$-stack. To facilitate the reading we recall the terminology of loc.~cit. in an appendix in Section 5.

Kamnitzer et al. have proved in \cite[thm.~1.2]{KTWWY19b} that, for symmetric types, the truncated category $\scrO$
of shifted Yangians is equivalent to a module category of some quiver Hecke algebra, denoted ${}_-P^\bfR$ in 
\cite[\S3.4]{KTWWY19b}. The algebra ${}_-P^\bfR$ there is the same as our algebra
${}^0\T^\dww$ in \eqref{0TT}. We deduce the following.

\begin{Corollary}
For symmetric types, the truncated category $\scrO$
of shifted Yangians and quantum loop groups are equivalent. The equivalence preserves the $q$-characters.
\end{Corollary}

\subsection{Notation}

Unless specified otherwise, all vector spaces, schemes or stacks are defined over $\bbC$.
 For any ordinary algebra $A$ let $A$-mod denote the category of all finitely generated ordinary $A$-modules.
Let $\calO=\bbC\llb z\rrb$ and $K$ its fraction field.
Let $K(\calC)$ denote the complexified Grothendieck group of an Abelian category $\calC$.
Taking $\calC$ to the category of finite dimensional rational representations of an affine group $G$,
we get $K(\calC)=R_G$ the complexified representation ring of $G$.
Let $H^\bullet_G$ be the cohomology algebra of $BG$.
Given a cocharacter $\gamma$ of $G$, we write $z^\gamma$ for the element
$\gamma(z)$ in $G$ for any $z\in\bbC^\times$.
We may abbreviate $\gamma$ for $z^\gamma$ if no confusion is possible.
We will use coherent sheaves over $\infty$-stacks. 
In doing so, we will use the same terminology as in \cite[\S 2]{CWb} to which we refer for more details.
To facilitate the reading, we recall the terminology in the appendix \ref{sec:CAT}.
For each integers $d$, $\ell$, $m$ we will abbreviate $\ell\equiv_dm$ for $\ell-m\in d\bbZ$.
For any ring $R$ and any set $I$ let $RI$ denote the set of $I$-tuple of elements of $R$ with finite support.
In particular, let $\delta_i\in RI$ be the Dirac function at $i$ for each $i\in I$.
For each finite dimensional $I$-graded vector space $V$ let $\dim_IV$ be the dimension vector in $\bbN I$.

\smallskip

\textbf{Acknowledgements.} It is a pleasure to thank D. Hernandez and 
R. Maksimau for inspiring discussions concerning this paper.

\section{Quiver Hecke algebras}
\subsection{$\bbZ$-weighted quiver Hecke algebras and tensor product algebras}\label{sec:QH}

\subsubsection{Preliminaries on quivers}\label{sec:Quivers}
For any connected finite quiver $Q=(I,E)$ with sets of vertices and arrows $I$ and $E$,
let $Q_f=(I_f,E_f)$ be the corresponding framed quiver.
The set of vertices is $I_f=I\times\{\circ, $\,\scalebox{0.6}{$\Box$}$\}$,
the set of arrows  $E_f$ consists of one arrow $(i,\circ)\to (j,\circ)$ for each arrow $i\to j$ in $E$,
and one arrow $(i,$\,\scalebox{0.6}{$\Box$}$)\to (i,\circ)$ for each vertex $i\in I$.
We will assume that the quiver $Q$ is of finite type,
bipartite,
and that the arrows in $E$ point from even to odd vertices. 
We will abbreviate $i\to j$ for $i\to j\in E$.
Let $p:I\to\{0,1\}$ be the parity.
We also fix a total order on the set $I$ such that even vertices are smaller than odd ones.
We may write 
$I=\{i_1<i_2<\dots<i_n\}.$
Note that 
\begin{align}\label{order1}i\to j\Rightarrow i<j.\end{align}
Let $C$ be the Cartan matrix of the quiver $Q$.
Given finite dimensional $I$-graded vector spaces 
$$V=\bigoplus_{i\in I}V_i,\quad
W=\bigoplus_{i\in I}W_i,$$
the  representation space of $Q_f$ in $W\oplus V$ is the space of linear maps
\begin{align}\label{Rep}\Rep_Q(V,W)=\Hom_I(W,V)\oplus\End_E(V)\end{align}
where
\begin{align*}
\Hom_I(W,V)&=\bigoplus_{i\in I}\Hom(W_i,V_i),\\
\End_E(V)&=\bigoplus_{i\to j\in E}\Hom(V_i,V_j)
\end{align*}
Let
\begin{align}\label{G}
G_V=\prod_i\GL(V_i)
,\quad
G_W=\prod_i\GL(W_i).
\end{align}
Let $T_V$ and $T_W$ be the diagonal maximal tori in $G_V$ and $G_W$.
We abbreviate
\begin{align}\label{GTLL1}
G=G_V,\quad
T=T_V.
\end{align}
Set 
\begin{align}\label{lam1}
\alpha=\dim_IV
,\quad
\lambda=\dim_IW
,\quad
\alpha=\sum_{i\in I}a_i\delta_i
,\quad
\lambda=\sum_{i\in I}l_i\delta_i
,\quad
\mu=\lambda-C\cdot\alpha=\sum_{i\in I}m_i\delta_i.
\end{align}
Note that
$a_i=\dim V_i$ and
$l_i=\dim W_i$.
Let $\Lambda_V$ and $\Lambda_W$
be the sets of cocharacters of $T_V$  and $T_W$.
Set
\begin{align}\label{GTLL2}
\Lambda_\alpha=\Lambda_V,\quad
\Lambda_\lambda=\Lambda_W.
\end{align}
Let $\Lambda_\alpha^+$ denote the cone of dominant cocharacters.
The affine group $G\times G_W$ acts on the representation space $\Rep_Q(V,W)$ in the obvious way.
Let 
$ (g, x)\mapsto g\cdot x$
denote the action map.
We abbreviate
\begin{align}\label{N} N_\mu^\lambda=N=\Rep_Q(V,W).\end{align}
A point in $ N_\mu^\lambda$ is a tuple $x=(A,B)$ where 
\hfill
\begin{enumerate}[label=$\mathrm{(\alph*)}$,leftmargin=8mm,itemsep=1mm]
\item
$A=(A_i\,;\,i\in I)$ is a tuple of maps $W_i\to V_i$ for each vertex $i\in I$, 
\item 
$B=(B_h\,;\,h\in E)$ is a tuple of maps $V_i\to V_j$ for each arrow $h:i\to j$ in $E$.
\end{enumerate}
Let $|v|$ denote the sum of the entries of a dimension vector $v$. Hence, we have
\begin{align}\label{length}
|\alpha|=\sum_{i\in I}a_i
,\quad
|\lambda|=\sum_{i\in I}l_i.
\end{align}
Let $\frakg$ be the complex simple Lie algebra of type $C$.
For each $i\in I$ let $\omega_i$ be the $i$th fundamental coweight, $\alpha_i$ the simple coroot,
and $\comega_i$, $\calpha_i$ the corresponding fundamental weights and simple roots.
Let $L(\lambda)$ be the irreducible $\frakg$-module with highest weight $\sum_{i\in I}l_i\,\comega_i$.
Given a sequence $w=(w_l)$ of dimension vectors in $\bbN  I$ with finite support,  we define
\begin{align}\label{Llambda}L(w)=\bigotimes_l L(w_l).\end{align}
For any $\frakg$-module $M$ and each dimension vector $\mu=\sum_{i\in I}m_i\,\delta_i$, 
let $M_\mu$ be the set of elements of $M$ of weight $\sum_{i\in I}m_i\,\comega_i$.

\subsubsection{$\bbZ$-weighted quiver Hecke algebras}\label{sec:CQH}
We now define a proper morphism 
\begin{align}\label{XN1}\pi_\mu^w:\scrX_\mu^w\to  N_\mu^\lambda.\end{align}
To do this, let $\P(\lambda)$ denote the set of all sequences  of dimension vectors in $(\bbN I)^\bbZ$ with sum  to $\lambda$,
and set
\begin{align}\label{Ibullet}I^\bullet=\bbZ\times I.\end{align}
Then, we fix a $\bbZ$-grading $(W^k)$ on the $I$-graded vector space $W$.
Choosing $(W^k)$ is the same as choosing an $I^\bullet$-graded
refinement of $W$, i.e., an $I^\bullet$-grading on $W$ compatible with its $I$-grading.
Let $w\in\P(\lambda)$ be the dimension sequence of $(W^k)$, i.e., the sequence $(w_k)$ such that $w_k=\dim_IW^k$.
Let $\Fl_\alpha$ be the set of $\bbZ$-weighted flags of $I$-graded vector subspaces of $V$, with its obvious algebraic structure.
A point of $\Fl_\alpha$ is a sequence of $I$-graded vector subspaces 
$F=(F^{\leqslant k})$ with $F^{\leqslant k}=0$ if $k\ll 0$, $F^{\leqslant k}\subset F^{\leqslant k+1}$ for all $k\in\bbZ$,
and  $F^{\leqslant k}=V$ if $k\gg 0$.
Set  $F^k=F^{\leqslant k}/F^{<k}$. 
We define the dimension sequence of the $\bbZ$-weighted flag $F$ to be $\dim_IF=(\dim_IF^k)$.
The space $\Fl_\alpha$ is a disjoint union of infinitely many connected components labelled by the set
$\P(\alpha)$ of all sequences  of dimension vectors with sum to  $\alpha$.
For $v\in\P(\alpha)$ we set
\begin{align}\label{flag}\Fl_v=\{F\in\Fl_\alpha\,;\,\dim_IF=v\}.\end{align}
The scheme
\begin{align}\label{X}
\scrX_\mu^w=\{(F,A,B)\in \Fl_\alpha\times  N_\mu^\lambda\,;\,
A(W^{\leqslant k})\subset F^{\leqslant k}\,,\,
B(F^{\leqslant k})\subset F^{< k}\,,\,
\forall k\in\bbZ\}
\end{align}
is the disjoint union of the smooth $G$-equivariant quasi-projective varieties
\begin{align}\label{XV}
\scrX_v^w=\{(F,A,B)\in \scrX_\mu^w\,;\,F\in\Fl_v\}
,\quad
v\in\P(\alpha).
\end{align}
Forgetting the $\bbZ$-weighted flag $F$ yields the  projective morphism \eqref{XN1}.
Let $\scrZ_\mu^w$ be the fiber product 
$$\scrZ_\mu^w=\scrX_\mu^w\times_{ N_\mu^\lambda} \scrX_\mu^w.$$
We equip the complexified Grothendieck group of the category of $G$-equivariant coherent sheaves on 
$\scrZ_\mu^w$ and its $G$-equivariant Borel-Moore homology space with the convolution product
as in \cite{CG}, \cite{VV11}.
This yields two algebras over the rings $R_G$ and $H^\bullet_G$ respectively
\begin{align}\label{QH}
{}_\K\widetilde\scrT_\mu^w=K^G(\scrZ_\mu^w)
,\quad
\widetilde\scrT_\mu^w=H_\bullet^G(\scrZ_\mu^w,\bbC)
\end{align}
For each sequence $v\in\P(\alpha)$ the fundamental class of the diagonal of
$\scrX_v^w$ yields an idempotent 
\begin{align}\label{eQH1}e_{w,v}\in{}_\K\widetilde\scrT_\mu^w\,,\,\widetilde\scrT_\mu^w\end{align} 
For any module $M$ over ${}_\K\widetilde\scrT_\mu^w$ or $\widetilde\scrT_\mu^w$ we call the subspace
$e_{w,v} M$ the $\ell$-weight space of $M$ of $\ell$-weight $v$.
Set
$$\max(v)=\max\{k\in\bbZ\,;\,v_k\neq 0\}\in\bbZ\cup\{-\infty\}.$$
We define the the cyclotomic ideals ${}_K\scrI_\mu^w$ and $\scrI_\mu^w$
in ${}_\K\widetilde\scrT_\mu^w$ and $\widetilde\scrT_\mu^w$ respectively
to be the two sided ideals generated by the idempotents $e_{w,v}$ such that
\begin{align}\label{cyclo}\max(w)<\max(v)\end{align}
Taking the quotients we get the algebras
${}_\K\scrT_\mu^w={}_\K\widetilde\scrT_\mu^w\,/\,{}_K\scrI_\mu^w$
and
$\scrT_\mu^w=\widetilde\scrT_\mu^w\,/\,\scrI_\mu^w$. 
We abbreviate
$$\widetilde\scrT^w=\bigoplus_\mu\widetilde\scrT_\mu^w
,\quad
\scrT^w=\bigoplus_\mu\scrT_\mu^w
,\quad
{}_K\widetilde\scrT^w=\bigoplus_\mu{}_K\widetilde\scrT_\mu^w
,\quad
{}_K\scrT^w=\bigoplus_\mu{}_K\scrT_\mu^w
.$$
To avoid confusions with other versions of QH algebras, we call
$\widetilde\scrT^w$, $\scrT^w$, ${}_K\widetilde\scrT^w$ and ${}_K\scrT^w$
$\bbZ$-weighted QH algebras.

\subsubsection{The Chern character}\label{sec:Chern}
The $H^\bullet_G$-algebra $\widetilde\scrT_\mu^w$  is graded by the homological degree.
Its degreewise completion is the topological algebra
$$\widetilde\scrT_\mu^{w,\wedge}=\lim_k\big(\widetilde\scrT_\mu^w\,/\,(H_G^{+})^k\cdot \widetilde\scrT_\mu^w\big)$$
where $H_G^{+}\subset H_G^\bullet$ is the augmentation ideal.
 
\begin{Definition}\label{def:smooth-nil}
\hfill
\begin{enumerate}[label=$\mathrm{(\alph*)}$,leftmargin=8mm,itemsep=1mm]
\item
A smooth representation $M$ over $\widetilde\scrT_\mu^{w,\wedge}$ is a 
module with a continuous action, where $M$ is given the discrete topology.
A nilpotent module of $\widetilde\scrT_\mu^w$ is
a representation which extends to a smooth representation
of $\widetilde\scrT_\mu^{w,\wedge}$. 
Let 
$\widetilde\scrT_\mu^w\-\nilp$ be the category of finitely generated nilpotent modules
with finite dimensional $\ell$-weight spaces.
\item
Let $\scrT_\mu^w\-\nilp$  be the category of finitely generated modules
with finite dimensional $\ell$-weight spaces.
\end{enumerate}
\end{Definition}

We define the unipotent completion of 
${}_\K\widetilde\scrT_\mu^w$  similarly.
More precisely, for any affine group $H$ with a semisimple element $h$, the evaluation at $h$ yields
an algebra homomorphism $R_H\to\bbC$.
Let $\bfm_{H,h}$ be its kernel.
For any $R_H$-module $M$ and any maximal ideal $\bfm\subset R_H$,
we set
\begin{align}\label{complete}M_{\widehat\bfm}=\lim_k(M/\bfm^kM).\end{align}
We abbreviate  $\bfm_h=\bfm_{H,h}$ and $M_{\widehat h}=M_{\widehat\bfm_h}$.
We apply this construction with $H=G$ and $h=1$. Set 
\begin{align}\label{unipotent}
{}_\K{\widetilde\scrT_\mu}^{w,\wedge}=({}_\K{\widetilde\scrT_\mu}^{w})_{\widehat 1}.
\end{align}
A module over ${}_\K\widetilde\scrT_\mu^w$  is unipotent if the action extends to a smooth representation
of the unipotent completion of the $\bbZ$-weighted QH algebra.
Let 
${}_\K\widetilde\scrT_\mu^w\-\unip$
be the category of finitely generated unipotent modules
with finite dimensional $\ell$-weight spaces, and
${}_\K\scrT_\mu^w\-\unip$ be the category of finitely generated modules
with finite dimensional $\ell$-weight spaces.

\begin{Lemma}\label{lem:Chern}
There are equivalence of categories  ${}_\K\widetilde\scrT_\mu^w\-\unip\cong\widetilde\scrT_\mu^w\-\nilp$ 
and ${}_\K\scrT_\mu^w\-\unip\cong\scrT_\mu^w\-\nilp$.
\end{Lemma}

\begin{proof}
The modified Chern character in \cite{CG} yields a topological algebra isomorphism
$${}_\K\widetilde\scrT_\mu^{w,\wedge}\cong\widetilde\scrT_\mu^{w,\wedge}.$$
This isomorphism yields an equivalence
$${}_\K\widetilde\scrT_\mu^w\-\unip\cong\widetilde\scrT_\mu^w\-\nilp.$$
The equivalence descends to an equivalence
$${}_\K\scrT_\mu^w\-\unip\cong\scrT_\mu^w\-\nilp,$$ because the isomorphism
 ${}_\K\widetilde\scrT_\mu^{w,\wedge}\cong\widetilde\scrT_\mu^{w,\wedge}$ identifies the cyclotomic ideals.
\end{proof}

\subsubsection{Tensor product algebras}\label{sec:TPA}
Fix the dimension vectors $\alpha$, $\lambda$, $\mu$ as in \eqref{lam1}, and
fix a decomposition $W=\bigoplus_{s=1}^m\W^s$ of the $I$-graded vector space $W$ with non zero parts.
We define the  set $\P(\lambda,m)$ to be
$$\P(\lambda,m)=\{\w=(\w_1,\w_2,\dots,\w_m)\in(\bbN I)^m\,;\,\w_s\neq 0\,,\,\w_1+\w_2+\cdots+\w_m=\lambda\}.$$
Let $\w\in \P(\lambda,m)$ be the $m$-tuple of dimension vectors, i.e., 
we set  $\w_s=\dim_I \W^s$.
For  each $\alpha\in\bbN I$ with $h=|\alpha|$ we define 
\begin{align}\label{tuples}
\begin{split}
\I(\alpha)&=\{\v=(\v_1,\v_2,\dots,\v_h)\in I^h\,;\,\v_1+\v_2+\cdots+\v_h=\alpha\},\\
\K(\alpha,m)&=\{ \kappa\in\bbN^m\,;\,\kappa_1\leqslant\kappa_2\leqslant\cdots\leqslant\kappa_m\leqslant h\}
,\\
\D(\alpha,m)&=\I(\alpha)\times\K(\alpha,m).
\end{split}
\end{align}
For each pair $(\v,\kappa)\in\D(\alpha,m)$, 
let $\F \ell_\v$ be the set of $h$-flags of $I$-graded vector spaces 
$0=\F^{\leqslant 0}\subset \F^{\leqslant 1}\subset \cdots\subset \F^{\leqslant h}=V$ with $\dim_I \F^l=\v_l$, 
and let $\X_{\v,\kappa}^\w$ be the smooth quasi-projective $G$-variety 
\begin{align}\label{XX}
\begin{split}
\X_{\v,\kappa}^\w=\{&(\F,A,B)\in\F\ell_\v\times  N_\mu^\lambda\,;\,A(\W^{\leqslant s})\subset \F^{\leqslant \kappa_s}\,,\,
B(\F^{\leqslant l})\subset \F^{<l}\}
\end{split}
\end{align}
Let $\X_\mu^\w$ be the sum of the varieties $\X_{\v,\kappa}^\w$ as the pair $(\v,\kappa)$ runs over the set $\D(\alpha,m)$.
Forgetting the flag F yields a projective morphism 
\begin{align}\label{XN2}\pi_\mu^\w:\X_\mu^\w\to  N_\mu^\lambda.\end{align}
Let $\Z_\mu^\w$ be the fiber product
$$\Z_\mu^\w=\X_\mu^\w\times_{ N_\mu^\lambda}\X_\mu^\w.$$
We define $\widetilde\T_\mu^\w$ to be the convolution algebra in equivariant Borel-Moore homology
$$\widetilde\T_\mu^\w=H_\bullet^G(\Z_\mu^\w,\bbC).$$
The multiplication is the obvious convolution product.
For each pair $(\v,\kappa)$ as above the fundamental class of the diagonal of $\X_{\v,\kappa}^\w$
yields an idempotent $\e_{\w,\v,\kappa}$ in $\widetilde\T^\w_\mu$. 
The cyclotomic ideal of $\widetilde\T_\mu^\w$ is the two sided ideal generated by the idempotents $\e_{\w,\v,\kappa}$
such that
\begin{align}\label{cyclo1}\kappa_m<h.\end{align}
We define the algebra $\T_\mu^\w$ to be the quotient
$$\T_\mu^\w=\widetilde\T_\mu^\w\,/\,\text{cyclotomic\ ideal}.$$
We also write
$$\widetilde\T^\w=\bigoplus_\mu\widetilde\T_\mu^\w,\quad\T^\w=\bigoplus_\mu\T_\mu^\w.$$
The algebras $\widetilde\T^\w$ and $\T^\w$ were introduced in \cite{W17}, and are called tensor product algebras.
They have the following alternative description.
The Crawley-Boevey quiver of the quiver $Q$ and the dimension vector $\w$ is the quiver $Q^{(\w)}$
with the set of vertices $I\sqcup\{1,\dots,m\}$
and the set of arrows equal to the sum of $E$ and $\w_{s,i}$ arrows
$s\to i$ for each $s=1,\dots,m$ and $i\in I$, where $\w_s=\sum_{i\in I}\w_{s,i}\,\delta_i$.
The tensor product algebra $\widetilde\T_\mu^\w$ is known to be isomorphic
to a QH algebra of type $Q^{(\w)}$, see \cite{W17}.
In particular, it is graded by the homological degree, and the cyclotomic quotient
$\T_\mu^\w$ is finite dimensional. 
Let 
$$\widetilde\T^\w\-\nilp,\quad \T^\w\-\nilp$$ be the categories of finite dimensional modules
of $\widetilde\T^\w$ and $\T^\w$.
Let $L(\w)$ be the $\frakg$-module 
\begin{align}\label{LW}L(\w)=\bigotimes_{s=1}^mL(\w_s).\end{align}
By \cite[thm.~B]{W17}, there is a linear isomorphism
\begin{align}\label{decat0}
L(\w)^\vee\cong
K(\T^\w\-\nilp)
\end{align}
which identifies the weight subspace $L(\w)_\mu^\vee$ with the complexified Grothendieck group $K(\T_\mu^\w\-\nilp)$.
More precisely, there are exact endofunctors 
\begin{align}\label{EF}\calE_i,\calF_i\in\End(\T^\w\text{-}\nilp),\quad i\in I
\end{align}
which categorify the $\frakg$-action on $L(\w)^\vee$ under
the isomorphism in \eqref{decat0}.
Further, the basis of $L(w)^\vee$ formed by the image under the
isomorphism \eqref{decat0} of the simple nilpotent modules of $\T^\w$ is the dual canonical basis
considered in \cite{W12}, \cite{BW16}.

\begin{Remark}\label{rem:w-omega}
Consider the tuple $\omega\in \I(\lambda)$ given by
\begin{align}\label{omega}
\begin{split}
\omega&=\big({i_1}^{(\w_{1,i_1})}\,,\,{i_2}^{(\w_{1,i_2})}\,,\,\dots,{i_n}^{(\w_{1,i_n})}\,,\,
{i_1}^{(\w_{2,i_1})}\,,\,\dots\,,\,{i_n}^{(\w_{m,i_n})}\big)
\end{split}
\end{align}
Taking $\omega$ instead of $\w$ we get the algebras $\T_\mu^\omega$ and $\T^\omega$.
Set $m'=|\lambda|$.
For all $\v\in\I(\alpha)$ there is an inclusion 
$$\K(\alpha,m)\to\K(\alpha,m')
,\quad
\kappa=(\kappa_1,\dots,\kappa_m)\mapsto \kappa'=(\kappa_1^{(|\w_1|)},\cdots,\kappa_m^{(|\w_m|)})$$
with an obvious isomorphism of $G$-varieties
$\X_{\v,\kappa}^\w\cong \X_{\v,\kappa'}^\omega$,
where 
$|\w_s|=\sum_{i\in I}\w_{s,i}.$
Taking the homology, this isomorphism yields an algebra embedding 
$\T^\w\subset\T^\omega$.
Under the isomorphism \eqref{decat0}, 
the restriction from $\T^\omega$ to $\T^\w$ is the transpose of the obvious embedding 
$L(\w)\subset L(\omega)$ of representations of $\frakg$.
\end{Remark}

\subsubsection{$\bbZ$-weighted quiver Hecke algebras and tensor product algebras}\label{sec:TPA}
In this section we compare the $\bbZ$-weighted QH algebra with the tensor product one.
Recall that $i_1<i_2<\dots <i_n$ is the total order on the set $I$ introduced in \eqref{order1}.
Fix a sequence $w=(w_l)$ in $\P(\lambda)$ and let   $l_1<l_2<\dots< l_m$ be the increasing integers such that
\begin{align}\label{l1}\{l_1,l_2,\dots,l_m\}=\{l\in\bbZ\,;\,w_l\neq 0\}.\end{align}
We consider the map
\begin{align}\label{map-1}
\tau:\P(\lambda)\to \P(\lambda,m),\quad w\mapsto \w
\end{align}
such that
$w_{l_s}=\w_s$ for all $s=1,\dots,m.$
With the notation in \eqref{Llambda}, \eqref{LW} we have a $\frakg$-module isomorphism
\begin{align}\label{LLW}L(w)=\bigotimes_{l\in\bbZ}L(w_l)\cong L(\w)=\bigotimes_{s=1}^mL(\w_s).\end{align}
Now, fix $\alpha\in\bbN I$ with $h=|\alpha|$.
Let $\J(\alpha,m)$ the set of triples  $(\x,\v,\kappa)$ in $\bbZ^h\times\D(\alpha,m)$
such that
\begin{enumerate}[label=$\mathrm{(\alph*)}$,leftmargin=8mm,itemsep=1mm]
\item[(1)]
$\x_k=\x_{k+1}\Rightarrow \v_k\leqslant \v_{k+1}$,
\item[(2)]
$\x$ is weakly increasing,
\item[(3)]
$\x_{\kappa_s}\leqslant l_s<\x_{\kappa_s+1}$ for each $s=1,\dots,m$.
\end{enumerate}
We define the element $\e_{\w,\mu}$ in $\widetilde\T_\mu^\w$ given by
\begin{align}\label{ewmu1}
\e_{\w,\mu}=\sum_{\v,\kappa}\e_{\w,\v,\kappa}
\end{align}
where the sum runs over all pairs $(\v,\kappa)$ such that there is  tuple $\x\in\bbZ^h$ 
with $(\x,\v,\kappa)\in\J(\alpha,m)$.

\begin{Proposition}\label{prop:TT} Assume that $\w=\tau(w)$.
\hfill
\begin{enumerate}[label=$\mathrm{(\alph*)}$,leftmargin=8mm,itemsep=1mm]
\item
The module categories $\widetilde\scrT_\mu^w$-$\nil$ and 
$\e_{\w,\mu}\cdot\widetilde\T_\mu^\w\cdot\e_{\w,\mu}$-$\nil$ are equivalent.
\item
The module categories 
$\scrT_\mu^w$-$\nil$ and $\e_{\w,\mu}\cdot\T_\mu^\w\cdot\e_{\w,\mu}$-$\nil$ are equivalent.
\item
We have $K(\scrT_\mu^w\-\nilp)^\vee\subset L(w)_\mu$.
\end{enumerate}
\end{Proposition}

\begin{proof}
For each sequence $v=(v_k)$ in $\P(\alpha)$ we set
\begin{align}\label{vvk}v_k=\sum_{i\in I}v_{k,i}\,\delta_i
,\quad
|v_{k}|=\sum_{i\in I}v_{k,i}.
\end{align}
Let $k_1<k_2<\dots<k_\ell$ be the integers such that
\begin{align}\label{k1}\{k_1,k_2,\dots,k_\ell\}=\{k\in\bbZ\,;\,v_k\neq 0\}.\end{align}
Then $(v_{k_1},\dots,v_{k_\ell})\in\P(\alpha,\ell).$
Let $\v\in\I(\alpha)$ be the tuple  such that
\begin{align}\label{v1}
\v=(\,{i_1}^{(v_{k_1,{i_1}})}\,,\,{i_2}^{(v_{k_1,i_2})}\,,\,\dots,{i_n}^{(v_{k_1,i_n})}\,,\,{i_1}^{(v_{k_2,i_1})}\,,\,\dots,
{i_n}^{(v_{k_\ell,i_n})}\,)\end{align}
where $i^{(u)}=i,i,\dots,i$ with multiplicity $u$.
There is an obvious $G$-equivariant map 
\begin{align}\label{FF1}\F\ell_\v\to\Fl_v,\quad \F\mapsto F\end{align}
such that $F$ is the $\bbZ$-weighted flag given by
\begin{align}\label{FF2}F^{\leqslant k_r}=\F^{\leqslant |v_{k_1}|+\cdots+|v_{k_r}|}
,\quad
\forall r=1,\dots,\ell\end{align}
We claim that for each sequence $v\in\P(\alpha)$ there is a tuple $\kappa\in\K(\alpha,m)$ such that the map
\begin{align}\label{map0}
\pi:\P(\alpha)\to \D(\alpha,m),\quad v\mapsto(\v,\kappa)
\end{align}
gives the following diagram 
\begin{align}\label{diag0}
\begin{split}
\xymatrix{\F\ell_\v\ar[d]&\ar[l]\X_{\pi(v)}^\w\ar[rd]\ar[d]&\\
\Fl_v&\ar[l]\scrX_v^w\ar[r]&\scrN_\mu^\lambda}
\end{split}
\end{align}
Here the square is Cartesian and the right maps are as in \eqref{XN1}, \eqref{XN2}.
To construct the map \eqref{map0}, recall that
$\w_s=w_{l_s}$ for all $s$.
Fix a $\bbZ$-grading $(W^l)$ on a $I$-graded vector space $W$ with dimension sequence $w$. 
We have $W=\bigoplus_{s=1}^m\W^s$ where $\W^s=W^{l_s}$.
Let $\kappa$ be the $m$-tuple such that 
\begin{align}\label{kappa1}
\kappa_s=\sum_{k_r\leqslant l_s}|v_{k_r}|
,\quad\forall s=1,\dots,m,
\end{align}
For $B\in\End_E(V)$ we have
\begin{align}\label{isomb}
\begin{split}
B(\F^{\leqslant l})\subset \F^{<l},\ \forall l=1,\dots,h
&\iff B(F^{\leqslant k_r})\subset F^{<k_r},\ \forall r=1,\dots,\ell\\
&\iff B(F^{\leqslant k})\subset F^{<k},\ \forall k\in\bbZ
\end{split}
\end{align}
In the first equivalence the implication $\Leftarrow$ follows from  \eqref{FF2}, while 
$\Rightarrow$ follows from \eqref{order1}  and \eqref{v1} which imply that the endomorphism of $F^{k_r}$ induced by $B$
 is lower triangular in a basis adapted to the complete flag $\F$.
The second equivalence follows from 
$$k_r\leqslant k<k_{r+1}\Rightarrow
F^{<k_r}\subset F^{<k}\ \text{and}\ F^{\leqslant k}=F^{\leqslant k_r}$$
Similarly,  set $[s]=\max\{r\,;\,k_r\leqslant l_s\}$. Then, we have
\begin{align}\label{isomc}
\begin{split}
A(\W^{\leqslant s})\subset\F^{\leqslant\kappa_s},\ \forall s=1,\dots,m
&\iff A(W^{\leqslant l_s})\subset F^{\leqslant k_{[s]}},\ \forall s=1,\dots,m,\\
&\iff A(W^{\leqslant l_s})\subset F^{\leqslant l_s},\ \forall s=1,\dots,m,\\
&\iff A(W^k)\subset F^{\leqslant k},\ \forall k\in\bbZ
\end{split}
\end{align}
The first equivalence follows from the equalities $\W^s=W^{l_s}$ and $\F^{\leqslant \kappa_s}=F^{\leqslant k_{[s]}}$,
which are consequences of \eqref{l1}, \eqref{FF2},  \eqref{kappa1}.
The second equivalence follows from the condition
$k_{[s]}\leqslant l_s<k_{[s]+1}$, which implies that $F^{\leqslant k_{[s]}}=F^{\leqslant l_s}$.
The third equivalence follows from the relations $W^{\leqslant k}=W^{\leqslant l_s}$ and $F^{\leqslant l_s}\subset F^{\leqslant k}$
for each $l_s\leqslant k<l_{s+1}$, which are consequences of  \eqref{l1}.
From \eqref{X}, \eqref{XX}  and \eqref{isomb}, \eqref{isomc} we deduce that the map \eqref{FF1} extends to a diagram as in \eqref{diag0}, 
proving the claim.

Let $v,v'\in\P(\alpha)$. 
We have
\begin{align*}
\e_{\w,\pi(v)}\cdot\widetilde\T_\mu^\w\cdot \e_{\w,\pi(v')}&=
H_\bullet^G(\X_{\pi(v)}^\w\times_{ N_\mu^\lambda}\X_{\pi(v')}^\w,\bbC)
\\
e_{w,v}\cdot\widetilde\scrT_\mu^w\cdot e_{w,v'}&=
H_\bullet^G(\scrX_v^w\times_{ N_\mu^\lambda}\scrX_{v'}^w,\bbC)
\end{align*}
By \eqref{diag0}, there are finite dimensional vector spaces
$M_v$ and an algebra isomorphism
\begin{align}\label{TcalT}
\bigoplus_{v,v'\in\P(\alpha)}\e_{\w,\pi(v)}\cdot\widetilde\T_\mu^\w\cdot \e_{\w,\pi(v')}
\cong
\bigoplus_{v,v'\in\P(\alpha)}e_{w,v}\cdot\widetilde\scrT_\mu^w\cdot e_{w,v'}\otimes\Hom(M_{v'},M_v)
\end{align}
The multiplication in the right side is the tensor product of the convolution and the composition of matrices.
Let us consider the following idempotent
\begin{align}\label{ewmu2}
\e'_{\w,\mu}=\sum_{(\v,\kappa)\in\Im(\pi)}\e_{\w,\v,\kappa}
\end{align}
Recall the following fact.

\begin{Lemma}\label{lem:Morita}
Let $A$, $B$  be algebras.
Assume that $\{e_i\,;\,i\in I\}$ is a complete and finite family of orthogonal idempotents of $A$,
and  $\{f_j\,;\,j\in J\}$ is a family of orthogonal idempotents of $B$ such that $B=\bigoplus_{j_2,j_1}f_{j_2}Bf_{j_1}$.
Let $\pi:J\to I$ be a map with a section $\sigma:\pi(J)\to J$.
Set $e=\sum_{i\in\pi(J)}e_i$.
Assume further that there are finite dimensional vector spaces $M_j$ for all $j\in J$, and linear isomorphisms
$$e_{i_2}Ae_{i_1}\cong f_{j_2}Bf_{j_1}\otimes\Hom(M_{j_1},M_{j_2})
,\quad
i_1=\pi(j_1)
,\quad
i_2=\pi(j_2),$$ 
such that the multiplication in $B$ is given by
$(a\otimes\phi)\cdot(b\otimes\psi)=ab\otimes(\phi\circ\psi)$.
Then, the category of finite dimensional $eAe$-modules
is equivalent to the category of finitely generated $B$-modules $M$ such that
$\dim(f_jM)<\infty$ for all $j\in J$.
\qed
\end{Lemma}

Note that in the lemma above we do not assume the set $J$ to be finite.
From \eqref{TcalT} and Lemma \ref{lem:Morita}, we deduce that there 
is an equivalence  of categories
$$\widetilde\scrT_\mu^w\-\nil\cong\e'_{\w,\mu}\cdot\widetilde\T_\mu^\w\cdot\e'_{\w,\mu}\-\nil.$$
To finish the proof of Part (a) we must check that the idempotents
 $\e_{\w,\mu}$ and $\e'_{\w,\mu}$ in \eqref{ewmu1} and \eqref{ewmu2} coincide. To do so,
 for each sequence $v\in\P(\alpha)$ let $k_1<k_2<\dots<k_\ell$ be as in \eqref{k1} and consider the tuple of integers $\x$ given by
\begin{align}\label{xx}
\x=\big(k_1^{(|v_{k_1}|)}\,,\,\dots,k_\ell^{(|v_{k_\ell}|)}\big)
\end{align}
where the integer $|v_k|$ is as in \eqref{vvk}.
Let $\J(\alpha)$ be the set of pairs  in $\bbZ^h\times\I(\alpha)$ satisfying the conditions (1) and (2) above.
We have the bijections
\begin{align}\label{map1}
\P(\alpha)\cong\J(\alpha,m)\cong\J(\alpha)
,\quad
v\mapsto(\x,\v,\kappa)\mapsto(\x,\v)
\end{align}
given by \eqref{v1},  \eqref{kappa1} and \eqref{xx}.
Under these bijections the map $\pi$ in \eqref{map0} is identified with the map
$$\J(\alpha,m)\to\D(\alpha,m)
,\quad
(\x,\v,\kappa)\mapsto(\v,\kappa).$$

To prove (b), note that the cyclotomic conditions \eqref{cyclo} and \eqref{cyclo1} are compatible because
\begin{align*}
\max(w)<\max(v)\iff l_m<k_\ell\iff [m]<\ell\iff \kappa_m<h
\end{align*}
Hence the algebras $\scrT_\mu^w$ and $\e_{\w,\mu}\cdot\T_\mu^\w\cdot\e_{\w,\mu}$ are also Morita equivalent.

If $W=0$ then (c) follows from \cite{KL09}.
If $W\neq 0$, then the category 
$\e_{\w,\mu}\cdot\T_\mu^\w\cdot\e_{\w,\mu}\text{-}\nilp$
is the quotient of the category $\T_\mu^\w\text{-}\nilp$
by the Serre subcategory generated by all
simple modules killed by the idempotent $\e_{\w,\mu}$.
Hence, the exact functor $M\mapsto \e_{\w,\mu} \cdot M$ yields a surjective morphism of Grothendieck groups
$$K(\T_\mu^\w\text{-}\nilp)\to K(\e_{\w,\mu}\cdot\T_\mu^\w\cdot\e_{\w,\mu}\text{-}\nilp).$$
Then (c) follows from (a), (b)  and \eqref{decat0}.
\end{proof}

From \S\ref{sec:TPA} we deduce the following.

\begin{Corollary}\label{cor:DCB}
The images by the embedding $K(\scrT_\mu^w\-\nilp)\subset L(w)_\mu^\vee$
in Lemma $\ref{prop:TT}$ of the simple nilpotent modules of $\scrT_\mu^w$ belong to the dual canonical basis.
\qed
\end{Corollary}

\begin{Remark}\label{rem:presentation}
\hfill
\begin{enumerate}[label=$\mathrm{(\alph*)}$,leftmargin=8mm,itemsep=1mm]
\item
A presentation of the algebra $\widetilde\T^\w$ is computed 
in  \cite[\S 4]{W17} and \cite[\S 3.2-4.2]{W19} with the same method as in \cite{VV11}.
We deduce that  $\widetilde\T^\w$ is isomorphic to the algebra
$\widetilde T^\w$ used in \cite[\S 3.1]{KTWWY19b}.

\item
The isomorphism \eqref{TcalT} is compatible with the cyclotomic quotients.
Since  $\T_\mu^\w$ is finite dimensional, we deduce that
$e_{w,v}\cdot\scrT_\mu^w\cdot e_{w,v'}$ is finite dimensional for each sequences $v,$ $v'$.
In particular, the $H_G^\bullet$-action on $\scrT_\mu^w$
is nilpotent.
Hence we do not need to complete $\scrT_\mu^w$.

\item
If $W=0$ then the same proof as in Proposition \ref{prop:TT} implies that
the restricted dual of $K(\widetilde\scrT_\mu^0\-\nilp)$ is isomorphic to $U_\alpha^+$.
\end{enumerate}
\end{Remark}

\subsection{Integral $\bbZ$-weighted quiver Hecke algebras for non symmetric Cartan matrices}
\label{sec:IZQH}

\subsubsection{Non symmetric Cartan matrices}\label{sec:NSCD}
By a Cartan datum we mean a pair ($\bfc,\bfd)$ where
$\bfc=(c_{ij}\,;\,i,j\in I)$ is a Cartan matrix and $\bfd=(c_i\,;\,i\in I)$ a symmetrizer of $\bfc$.
Recall that the $c_i$'s are positive integers such that $c_ic_{ij}=c_jc_{ji}$.
Let $c=\gcd(c_i\,;\,i\in I)$, $d_i=c_i/c$ and $b_{ij}=d_ic_{ij}$. 
Let $\pmb\frakg$ be the complex simple Lie algebra with Cartan matrix $\bfc$.
For each $i\in I$ let $\bomega_i$ be the $i$th fundamental coweight and $\balpha_i$ the simple coroot.
Let $\bcomega_i$, $\bcalpha_i$ be the corresponding fundamental weights and simple roots.
It is known that the reflexive, symmetric and transitive closure of the relation 
$$(i,r)\sim(j,s)\iff (b_{ij}\neq 0 \ \text{and}\ s=r+b_{ij})$$
has two equivalences classes in $I^\bullet$.
Let
\begin{align}\label{0I}
{}^0\!I\,,\,\dI\subset I^\bullet
\end{align} 
be one of these equivalences classes and
the subset given by
\begin{align}\label{dI}
\dI=\{(i,r)\in{}^0\!I\,;\,i\in I\,,\,r\in[1,2c_i]\}.
\end{align}
Let $Q=(I,E)$ and $\dQ=(\dI,\dE)$ be simply laced quivers such that
\hfill
\begin{enumerate}[label=$\mathrm{(\alph*)}$,leftmargin=8mm,itemsep=1mm]
\item[$\bullet$]
$i\to j\in E$ or $j\to i\in E\iff c_{ij}<0$,
\item[$\bullet$]
$\dE=\{(i,r)\to(j,s)\,;\,i\to j\in E\,,\,s-r+b_{ij}\in 2\gcd(c_i,c_j)\cdot\bbZ\}$ 
\end{enumerate}
Assume that the quiver $Q$ satisfies the conditions in \S\ref{sec:Quivers}.
Hence $Q$ is bipartite with parity $p$ and the set $I$ is given a total order such that \eqref{order1} holds.
Let $C$ be the Cartan matrix of $Q$ and $\dC$ the Cartan matrix of $\dQ$.
The quiver $\dQ$ is bipartite with parity given by $p(i,r)=p(i)$, and
the arrows in $\dE$ point from even to odd vertices. 
We fix the total order on the set $\dI$ such that 
$(i,r)<(j,s)\Rightarrow i<j$. Hence, even vertices are smaller than odd ones.
Let $\dfrakg$ be the complex simple Lie algebra of type $\dC$. 
From now on, unless specified otherwise we will assume that $c=1$.
A case-by-case checking yields the following choices.
Set $I=\{1,2,\dots,n\}$ and $\bfd=(d_1,\dots,d_{n-1},d_n)$. \\[-4mm]

\begin{enumerate}[label=$\mathrm{(\alph*)}$,leftmargin=8mm,itemsep=2mm]
\item[$\bullet$]
If $\bfc=A_n$, $D_n$, $E_n$ then $C=\dC=\bfc$ and ${}^0\!I=\{(i,p(i)+r)\,;\,r\in 2\bbZ\,,\,i\in I\}$.
\item[$\bullet$]
If $\bfc=B_n$ and $\bfd=(2,2,\dots,2,1)$, then $C=A_n$, $\dC=A_{2n-1}$ and ${}^0\!I=\{(i,r)\,;\,r\in 2\bbZ\,,\,i\in I\}$.
We choose the parity on $Q$ such that $p(n)=0$. The picture of $\dQ$ is\\[-2mm]

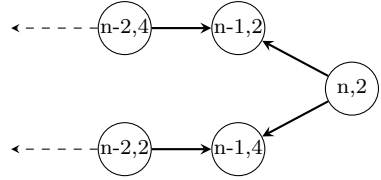
\begin{center}
\begin{tikzpicture}[scale=1, 
every node/.style={circle, draw, minimum size=0.7cm, inner sep=0pt, font=\footnotesize}, >=stealth 
 ]

\node (1) at (3,0) {n,2};
\node (2) at (1.5,0.8) {n-1,2};
\node (3) at (1.5,-0.8) {n-1,4 };
\node (4) at (0,0.8) {n-2,4 };
\node (5) at (0,-0.8) {n-2,2 };

\draw[->, line width=0.8pt, scale=1.5] (4) -- (2);  
\draw[->, line width=0.8pt, scale=1.5] (5) -- (3);  
\draw[->, line width=0.8pt, scale=1.5] (1) -- (2);  
\draw[->, line width=0.8pt, scale=1.5] (1) -- (3);  
\draw[dashed,-stealth] (4) -- ++(-1.5,0);
\draw[dashed,-stealth] (5) -- ++(-1.5,0);

\end{tikzpicture}
\end{center}

\item[$\bullet$]
If $\bfc=C_n$ and $\bfd=(1,1,\dots,1,2)$, then $C=A_n$, $\dC=D_{n+1}$ and ${}^0\!I=\{(n,r),(i,p(i)+r)\,;\,r\in 2\bbZ\,,\,i\neq n\}$.
We choose the parity on $Q$ such that $p(n)=1$. The picture of $\dQ$  is

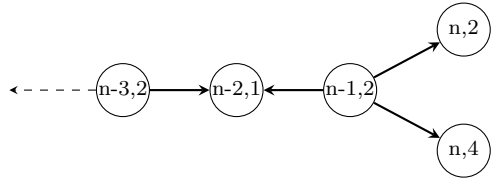
\begin{center}
\begin{tikzpicture}[scale=1, every node/.style={circle, draw, minimum size=0.7cm, inner sep=0pt, font=\footnotesize}, 
>=stealth 
]

\node (1) at (0,0) {n-3,2};
\node (2) at (1.5,0) {n-2,1};
\node (3) at (3,0) {n-1,2};
\node (4) at (4.5,0.8) {n,2};
\node (5) at (4.5,-0.8) {n,4};

\draw[->, line width=0.8pt, scale=1.5] (1) -- (2);  
\draw[->, line width=0.8pt, scale=1.5] (3) -- (2);  
\draw[->, line width=0.8pt, scale=1.5] (3) -- (4);  
\draw[->, line width=0.8pt, scale=1.5] (3) -- (5);  
\draw[dashed,-stealth] (1) -- ++(-1.5,0);

\end{tikzpicture}
\end{center}

\item[$\bullet$]
If $\bfc=F_4$ and $\bfd=(2,2,1,1)$, then $C=A_4$, $\dC=E_6$ and ${}^0\!I=\{(4,1+r),(i,r)\,;\,r\in 2\bbZ\,,\,i\neq 4\}$.
We choose the parity on $Q$ such that $p(4)=1$. The picture of $\dQ$ is

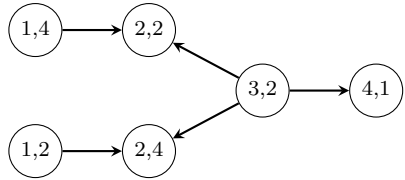
\begin{center}
\begin{tikzpicture}[scale=1, 
every node/.style={circle, draw, minimum size=0.7cm, inner sep=0pt, font=\footnotesize}, >=stealth 
 ]

\node (0) at (4.5,0) {4,1};
\node (1) at (3,0) {3,2};
\node (2) at (1.5,0.8) {2,2};
\node (3) at (1.5,-0.8) {2,4 };
\node (4) at (0,0.8) {1,4 };
\node (5) at (0,-0.8) {1,2 };

\draw[->, line width=0.8pt, scale=1.5] (4) -- (2);  
\draw[->, line width=0.8pt, scale=1.5] (5) -- (3);  
\draw[->, line width=0.8pt, scale=1.5] (1) -- (2);  
\draw[->, line width=0.8pt, scale=1.5] (1) -- (3);  
\draw[->, line width=0.8pt, scale=1.5] (1) -- (0);  

\end{tikzpicture}
\end{center}

\item[$\bullet$]
If $\bfc=G_2$ and $\bfd=(3,1)$, then $C=A_2$, $\dC=D_4$ and ${}^0\!I=\{(2,1+r),(1,r)\,;\,r\in 2\bbZ\}$.
We choose the parity on $Q$ such that $p(2)=0$. The picture of $\dQ$ is
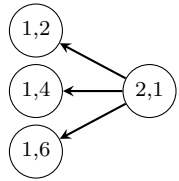
\begin{center}
\begin{tabular}{c@{\hspace{2cm}}c}  

\begin{tikzpicture}[scale=1, 
    every node/.style={circle, draw, minimum size=0.7cm, inner sep=0pt, font=\footnotesize}, 
    >=stealth 
]

\node (1) at (0,0) {1,4};
\node (2) at (1.5,0) {2,1};
\node (3) at (0,0.8) {1,2};
\node (4) at (0,-0.8) {1,6};

\draw[->, line width=0.8pt, scale=1.5] (2) -- (1);  
\draw[->, line width=0.8pt, scale=1.5] (2) -- (3);  
\draw[->, line width=0.8pt, scale=1.5] (2) -- (4);  

\end{tikzpicture}

\end{tabular}
\end{center}

\end{enumerate}

\subsubsection{Integral $\bbZ$-weighted quiver Hecke algebras}\label{sec:IQHA}
Let the quiver $\dQ$ be as in \S\ref{sec:NSCD}.
Let $\dI=\{\dii_1<\dii_2<\dots<\dii_n\}$ be the set of vertices of $\dQ$.
For any finite dimensional $\dI$-graded vector spaces $\dV$ and $\dW$, we define
\begin{align}\label{dlam1}
\dalpha=\dim_{\dI}\dV,
\quad
\dlambda=\dim_{\dI}\dW
,\quad
\dmu=\dlambda-\dC\cdot\dalpha\end{align}
We apply the construction in \S\ref{sec:CQH} to the quiver $\dQ$.
By \eqref{Rep}, \eqref{N}, the representation variety of $\dQ_f$ is
\begin{align}\label{dN}N_\dmu^\dlambda=\Rep_\dQ(\dV,\dW)
\end{align} 
We define the $\dG$-equivariant quasi-projective varieties 
$\scrX_\dmu^\dw$ and $\scrX_\dv^\dw$ as in \eqref{X} and \eqref{XV}.
Taking the equivariant Borel-Moore homology, we define the convolution algebra $\widetilde\scrT_\dmu^\dw$ as in \eqref{QH}.
Let $\scrT_\dmu^\dw$ be the cyclotomic quotient of $\widetilde\scrT_\dmu^\dw$.
The goal of this section is to introduce a refinement of this construction, to define subalgebras
${}^0\widetilde\scrT_\dmu^\dw$
and
${}^0\scrT_\dmu^\dw$
of the algebras
$ \widetilde\scrT_\dmu^\dw$
and
$\scrT_\dmu^\dw$
respectively.
To do this, we say that a pair $(\dii,k)\in\dI^\bullet$ with $\dii=(i,r)\in\dI$ and $k\in\bbZ$ is even if and only if
\begin{align}\label{even1}
k\equiv_{2d_i}r
\end{align}  Further, we say that a
sequence of dimension vectors $\dw_k=\sum_\dii\dw_{k,\dii}\,\delta_\dii$ in $\bbN\dI$
is even if $\dw_{k,\dii}=0$ for each odd pair $(\dii,k)$.
A $\bbZ$-grading $(\dW^k)$ on  $\dW$ is even if its dimension sequence 
$\dw=(\dw_k)$  is even. 
Let 
$${}^0\P(\dlambda)\subset\P(\dlambda)$$ be the subset of all even sequences.
We define similarly an even $\bbZ$-grading $(\dV^k)$ on $\dV$, its dimension sequence $\dv=(\dv_k)$,
and the subset of even sequences
\begin{align}\label{0Palpha}{}^0\P(\dalpha)\subset\P(\dalpha).\end{align}
Let ${}^0\!e_{\dw,\dmu}$ denote the idempotent in $\widetilde\scrT_\dmu^\dw$ or $\scrT_\dmu^\dw$ such that
\begin{align}\label{idempotent0}
{}^0\!e_{\dw,\dmu}=\sum_{\dv\in{}^0\P(\dalpha)}e_{\dw,\dv}.
\end{align}

\begin{Definition}\label{def:0T}
Let $\dw$ be an even sequence.
\begin{enumerate}[label=$\mathrm{(\alph*)}$,leftmargin=8mm,itemsep=2mm]
\item
The integral $\bbZ$-weighted QH algebra of type $(\bfc,\bfd)$ and weight $(\dw,\dmu)$ is the subalgebra of $\widetilde\scrT_\dmu^\dw$ given by
\begin{align}\label{0tT}
{}^0\widetilde\scrT_\dmu^\dw={}^0\!e_{\dw,\dmu}\cdot\widetilde\scrT_\dmu^\dw\cdot {}^0\!e_{\dw,\dmu}
\end{align}
We abbreviate
$${}^0\widetilde\scrT^\dw=\bigoplus_\dmu{}^0\widetilde\scrT_\dmu^\dw.$$
\item
The integral cyclotomic $\bbZ$-weighted QH algebra is 
\begin{align}\label{0T1}
{}^0\scrT_\dmu^\dw={}^0\!e_{\dw,\dmu}\cdot\scrT_\dmu^\dw\cdot {}^0\!e_{\dw,\dmu}.
\end{align}
We abbreviate
$${}^0\scrT^\dw=\bigoplus_\dmu{}^0\scrT_\dmu^\dw.$$
\item
A $\scrT_\dmu^\dw$-module is said to be odd if it is killed by the idempotent ${}^0\!e_{\dw,\dmu}$. 
The category of nilpotent ${}^0\scrT_\dmu^\dw$-modules is
the quotient of the category of nilpotent $\scrT_\dmu^\dw$-modules by the Serre subcategory generated by all
odd nilpotent modules. 
Let ${}^0\scrT_\dmu^\dw\-\nilp$ denote this category.
The quotient functor is 
$$\scrT_\dmu^\dw\-\nilp\to {}^0\scrT_\dmu^\dw\-\nilp,\quad M\mapsto {}^0\!e_{\dw,\dmu}\cdot M$$ 
where the category $\scrT_\dmu^\dw\-\nilp$ is as in Definition \ref{def:smooth-nil}.
\end{enumerate}
\end{Definition}

The geometric realization of the algebra ${}^0\widetilde\scrT_\dmu^\dw$ is the following.
Let ${}^0\!\Fl_\dalpha\subset\Fl_\dalpha$ be the subset of $\bbZ$-weighted flags  with even
dimension sequence, i.e., we have
$${}^0\!\Fl_\dalpha=\bigsqcup_{\dv\in{}^0\P(\dalpha)}\Fl_\dv.$$
Set
\begin{align}\label{0X}
\begin{split}
{}^0\scrX_\dmu^\dw=\{(F,A,B)\in \scrX_\dmu^\dw\,;\,F\in{}^0\!\Fl_\dalpha\}
,\quad
{}^0\scrX_\dv^\dw=\{(F,A,B)\in \scrX_\dv^\dw\,;\,F\in{}^0\!\Fl_\dalpha\}
\end{split}
\end{align}
Forgetting the flag yields the projective morphism
\begin{align}\label{0XN}
{}^0\scrX_\dv^\dw\to N_\dmu^\dlambda
\end{align}
The integral $\bbZ$-weighted QH algebra ${}^0\widetilde\scrT_\dmu^\dw$ 
 is isomorphic to the convolution algebras of the fiber product
 $${}^0\scrX_\dmu^\dw\times_{N_\dmu^\dlambda} {}^0\scrX_\dmu^\dw$$ 
 in Borel-Moore homology.
Let ${}_\K^0\widetilde\scrT_\dmu^\dw$ denote
the convolution algebra in K-theory.
The quotient functor $\scrT_\dmu^\dw\-\nilp\to {}^0\scrT_\dmu^\dw\-\nilp$ yields a surjective morphism of complexified Grothendieck groups
\begin{align}\label{decat1}
K(\scrT_\dmu^\dw\-\nilp)\to
K({}^0\scrT_\dmu^\dw\-\nilp)
\end{align}
Assume that $W\neq 0$. 
Then Proposition \ref{prop:TT} and \eqref{decat1} yield a surjective morphism 
\begin{align}\label{decat2}
L(\dw)_\dmu^\vee\to
K({}^0\scrT_\dmu^\dw\-\nilp)
\end{align}
where $L(\dw)^\vee$ is the dual of the $\dfrakg$-module $L(\dw)$ as in \eqref{LLW}.
Let $m$ and $l_1<l_2<\dots<l_m$ be the integers such that
\begin{align}\label{l}\{l_1,l_2,\dots,l_m\}=\{l\in\bbZ\,;\,\dw_l\neq 0\}.\end{align}
Let  $\dww=\tau(\dw)$ in $\P(\dlambda,m)$, where $\pi$ is as in \eqref{map-1}.
Hence, we have $L(\dw)\cong L(\dww)$.
We define 
\begin{align*}{}^0\!L(\dww)_\dmu\subset L(\dww)_\dmu
\end{align*}
to be the image of the transpose of the map \eqref{decat2}.
We define
\begin{align}\label{0L}{}^0\!L(\dww)=\bigoplus_\dmu{}^0\!L(\dww)_\dmu.\end{align}
We will also need the tuple
\begin{align}\label{omega}
\begin{split}
\domega&=\big({\dii_1}^{(\dw_{l_1,\dii_1})}\,,\,{\dii_2}^{(\dw_{l_1,\dii_2})}\,,\,\dots,{\dii_n}^{(\dw_{l_1,\dii_n})}\,,\,
{\dii_1}^{(\dw_{l_2,\dii_1})}\,,\,\dots\,,\,{\dii_n}^{(\dw_{l_m,\dii_n})}\big)\in\I(\dlambda)
\end{split}
\end{align}
where all zero multiplicities are omitted.
There is an obvious embedding of representations of $\dfrakg$
\begin{align}\label{womega}L(\dww)\subset L(\domega).\end{align}
Let $\calM$ be the set of Laurent monomials over the set of variables 
$$\{\bfz_{\dii,k}\,;\,\dii\in\dI\,,\,k\in\bbZ\,,\,k\equiv_2p(i)\}$$
We equip the set $\calM$ with the Nakajima crystal structure in \cite{N03}, see also \cite[\S 2.3]{KTWWY19a}.
Let $\calM(\bfz_{\dii,k})$ be the subcrystal generated by the monomial $\bfz_{\dii,k}$.
Let
$\calM(\dw)$ be the product of the monomials in $\calM(\bfz_{\dii,k})$ where
$k$ is any integer and  $\calM(\bfz_{\dii,k})$ has the multiplicity $\dw_{k,\dii}$.

\begin{Proposition}\label{prop:Parity} \hfill
\begin{enumerate}[label=$\mathrm{(\alph*)}$,leftmargin=8mm,itemsep=2mm]

\item
The map \eqref{decat2} identifies the vector space $K({}^0\scrT^\dw\-\nilp)^\vee$ with ${}^0\!L(\dww)$, 
and $K({}^0\scrT_\dmu^\dw\-\nilp)^\vee$ with ${}^0\!L(\dww)_\dmu$. 
\item
${}^0\!L(\dww)$ is a $\dfrakg$-submodule of $L(\dww)$ which is
spanned by a subset of the dual canonical basis.
\item
If $\bfc=ADE$ then the crystal of ${}^0\!L(\dww)$ is isomorphic
to $\calM(\dw)$. 
\end{enumerate}
\end{Proposition}

\begin{proof}
The proof of (a) was sketched in the discussion above. Let us give more details.
Let $\T_\dmu^\dww$ and $\T^\dww$ be the tensor product algebras of type $\dQ$,
see \S\ref{sec:TPA} for details.
We will identify the module category of the algebra ${}^0\scrT^\dw$ with a module category of
a subalgebra of $\T^\dww$.
To do this set $h=|\dalpha|$ and let $\J(\dalpha,m)$ be the set of triples 
$(\dxx,\dvv,\kappa)$ in $\bbZ^h\times\D(\dalpha,m)$ 
such that
\begin{enumerate}[label=$\mathrm{(\alph*)}$,leftmargin=8mm,itemsep=1mm]
\item[(1)]
$\dxx_k=\dxx_{k+1}\Rightarrow \dvv_k\leqslant \dvv_{k+1}$,
\item[(2)]
$\dxx$ is weakly increasing,
\item[(3)]
$\dxx_{\kappa_s}\leqslant l_s<\dxx_{\kappa_s+1}$.
\end{enumerate}
Let ${}^0\J(\dalpha,m)$ be the subset of triples such that
\begin{enumerate}[label=$\mathrm{(\alph*)}$,leftmargin=8mm,itemsep=1mm]
\item[(4)]
$(\dvv_k,\dxx_k)$ is even for all $k$.
\end{enumerate}
By \eqref{map1}, there is a bijection
$\P(\dalpha)\cong\J(\dalpha,m)$ taking $v$ to the triple $(\dxx,\dvv,\kappa)$ such that
\begin{align*}
\dvv&
=(\dots,{\dii_1}^{(\dv_{k,\dii_1})}\,,\,{\dii_2}^{(\dv_{k,\dii_2})}\,,\,\dots,{\dii_n}^{(\dv_{k,\dii_n})}\,,\,{\dii_1}^{(\dv_{k+1,\dii_1})},\dots),\\
\dxx&
=(\dots ,k^{(|\dv_k|)}\,,\,(k+1)^{(|\dv_{k+1}|)},\dots),\\
\kappa_s&=\max\{k\,;\,\dxx_k\leqslant l_s\}
\end{align*}
It restricts to a bijection 
${}^0\P(\dalpha)\cong {}^0\J(\dalpha,m).$
An idempotent $\e_{\dww,\dvv,\kappa}$ in $\T_\dmu^\dww$ is called even
if there is a tuple $\dxx$ such that  $(\dxx,\dvv,\kappa)\in {}^0\J(\dalpha,m)$. 
We define ${}^0\e_{\dww,\dmu}$ to be the sum of all even idempotents.
We consider the algebras
\begin{align}\label{0T2}
{}^0\T_\dmu^\dww={}^0\e_{\dww,\dmu}\cdot\T_\dmu^\dww\cdot{}^0\e_{\dww,\dmu}
,\quad
{}^0\T^\dww=\bigoplus_\dmu{}^0\T_\dmu^\dww.
\end{align}
Proposition \ref{prop:TT}, applied to the quiver $\dQ$, yields an equivalence of categories
\begin{align}\label{0TT}
{}^0\scrT^\dw\-\nilp\cong {}^0\T^\dww\-\nilp
\end{align} 
Let ${}^0\e_\dww=\sum_\dmu{}^0\e_{\dww,\dmu}$
be the sum over all $\dmu$'s. The category ${}^0\T^\dww\-\nilp$ is the quotient of $\T^\dww\-\nilp$
by the Serre subcategory generated by all simple odd nilpotent modules, i.e., 
the simple nilpotent modules killed by the idempotent ${}^0\e_\dww$.
This proves Part (a).

For (b), let us first prove that ${}^0\!L(\dww)$ is a $\dfrakg$-submodule of $L(\dww)$.
To do this, we must check that the functors $\calE_\dii$ and $\calF_\dii$ in \eqref{EF} descend to the category 
${}^0\T^\dww$-nil.
We will use the same notation as in \cite[\S3.2]{KTWWY19b} to which we refer for more details.
For each pair $(\v,\kappa)$ we write
\begin{align}\label{v'k'}
\check\dvv=(\dii,\dvv_1,\dots,\dvv_{|\dalpha|})
,\quad
\check\kappa=(1+\kappa_1,1+\kappa_2,\dots,1+\kappa_m)
\end{align}
Let $\e_\dii$
be the sum of all idempotents $\e_{\dww,\check\dvv,\check\kappa}$ 
as $(\dvv,\kappa)$ runs over the set $\D(\alpha,m)$.
There is an algebra embedding
$$\phi_\dii :\T_{\dmu}^\dww\to\e_\dii\cdot\T_{\dmu-\delta_\dii}^\dww\cdot\e_\dii$$
which takes the idempotent
$\e_{\dww,\dvv,\kappa}$ to the idempotent  $\e_{\dww,\check\dvv,\check\kappa}$.
The functor
\begin{align}\label{EF1}\calF_\dii:\T_\dmu^\dww\-\nilp\to\T_{\dmu-\delta_\dii}^\dww\-\nilp
\,,\,M\mapsto \T_{\dmu-\delta_\dii}^\dww\otimes_{\T_\dmu^\dww} M\end{align}
is the induction along $\phi_\dii$.
The functor $\calF_\dii$ is right adjoint to $\calE_\dii$. It is given by
\begin{align}\label{EF2}\calE_\dii:\T_{\dmu-\delta_\dii}^\dww\-\nilp\to\T_\dmu^\dww\-\nilp
\,,\,M\mapsto\e_\dii\cdot M\end{align}

To prove that $\calE_\dii$ descends it is enough to check that
if the module $M$ is odd then
$\calE_\dii(M)$ is also odd. 
Assume that  $M$ is odd but
$\calE_\dii(M)$ is not.
Since $\calE_\dii(M)$ is not odd there is an even idempotent $\e_{\dww,\dvv,\kappa}$ such that
$\e_{\dww,\dvv,\kappa}\cdot\calE_\dii(M)\neq 0$
and $\e_{\dww,\check\dvv,\check\kappa}\cdot M\neq 0$.
From the relations (1)-(4) we deduce that if $\e_{\dww,\dvv,\kappa}$ is even then $\e_{\dww,\check\dvv,\check\kappa}$ is also even.
This contradicts the fact that $M$ is odd.

To prove that $\calF_\dii$ descends  it is enough to check that if $M$ is odd then
$\calF_\dii(M)$ is also odd. 
By \cite{KL09} the character of induced modules of QH algebras is given by shuffling the characters of the original
modules. We deduce that if $\e_{\dww,\dvv',\kappa'}\cdot \calF_\dii (M)\neq 0$ then there is an idempotent 
$\e_{\dww,\dvv,\kappa}\in\T_\dmu^\dww$ such that
$\e_{\dww,\dvv,\kappa}\cdot M\neq 0$ and the tuple $\dvv'$ in $\I(\dalpha+\delta_\dii)$ is a shuffle of $\delta_\dii$ and
$\dvv$.
Then, using the relations (1)-(4) above, it is not difficult to see that if $\e_{\dww,\dvv',\kappa'}$ is even then $\e_{\dww,\dvv,\kappa}$ is also even.
Since the module $M$ is odd and $\e_{\dww,\dvv,\kappa}\cdot M\neq 0$, we deduce that the idempotent $\e_{\dww,\dvv,\kappa}$ is odd.
Hence $\e_{\dww,\dvv',\kappa'}$ is also odd.
This implies that the module $ \calF_\dii(M)$ is odd.
The compatibility with canonical bases follows from Corollary \ref{cor:DCB}.

To prove Part (c) of the proposition we assume that $\bfc=ADE$. Then,  if $\dxx_k=\dxx_{k+1}$
there are no arrows between $\dvv_k$ and $\dvv_{k+1}$ because the quiver $\dQ$ is bipartite, hence switching
$\dvv_k$ and $\dvv_{k+1}$ does not change the idempotent $\e_{\dww,\dvv,\kappa}$.
We deduce that an idempotent $\e_{\dww,\dvv,\kappa}\in\T^\dww$ is even if and only if there is a tuple of integers $\dxx$
for which the triple $(\dxx,\dvv,\kappa)$ satisfies the following conditions
\begin{enumerate}[label=$\mathrm{(\alph*)}$,leftmargin=8mm,itemsep=1mm]
\item[(2)]
$\dxx$ is weakly increasing,
\item[(3)]
$\dxx_{\kappa_s}\leqslant l_s<\dxx_{\kappa_s+1}$ for each $s=1,\dots,m$.
\item[(4)]
$\dxx_k$ and $\dvv_k$ have the same parity for all $k$.
\end{enumerate}
Hence, the proposition follows from \cite[def.~3.21-3.23, cor.~3.27]{KTWWY19b},
because $\T^\dww$ is isomorphic to the algebra $T^\dww$ in \cite[\S 3.1]{KTWWY19b} by Remark \ref{rem:presentation}.
Note that \cite{KTWWY19b} uses the algebra $\T^\domega$ instead of $\T^\dww$, and that the convention there
for the condition (3) is different from ours, see Remark \ref{rem:xvk}.
\end{proof}

For each integer $d$
we will say that the sequence $\dw$ is $d$-generic if we have $\dww=\domega$ and if $l_{s+1}-l_s\gg d$ for each $s$.
The following is an immediate consequence of the definitions.

\begin{Lemma}\hfill
\begin{enumerate}[label=$\mathrm{(\alph*)}$,leftmargin=8mm,itemsep=2mm]
\item
$L(\dlambda)\subset{}^0\!L(\dww)$.
\item
If $m=1$ then ${}^0\!L(\dlambda)=L(\dlambda)$.
\item
If $\dw$ is $|\alpha|$-generic then ${}^0\!L(\dww)_\dmu=L(\dww)_\dmu$.
\end{enumerate}
\qed
\end{Lemma}

\begin{Remark}\label{rem:xvk}\hfill
\begin{enumerate}[label=$\mathrm{(\alph*)}$,leftmargin=8mm,itemsep=2mm]
\item
If $\bfc=ADE$ then $\dQ\cong Q$, $\dI\cong I$ and  $(\dw_k)$ is even 
if and only $\dw_{k,\dii}=0$ for each $k\not\equiv_{2}p(\dii)$.
\item
We could replace the ${}^0\!I$-grading in the definition of the integral QH algebra by a grading by the 
larger set $I^\bullet$.
In Proposition \ref{prop:Parity} the quiver $\dQ$ would be replaced by  a non connected quiver with 2 isomorphic connected 
components. We will see later that the ${}^0\!I$-grading is the condition required for the integral
category $\scrO$ of the (truncated) shifted quantum loop group considered in \S\ref{sec:O} and \S\ref{sec:HCO}.
\item
Taking $\domega$ instead of $\dww$ we get the algebras $\T_\dmu^\domega$ and $\T^\domega$.
Set $m'=|\dlambda|$.
For all $\dvv\in\I(\dalpha)$ there is an inclusion 
$$\K(\dalpha,m)\to\K(\dalpha,m')
,\quad
\kappa=(\kappa_1,\dots,\kappa_m)\mapsto \kappa'=(\kappa_1^{(|\dww_1|)},\cdots,\kappa_m^{(|\dww_m|)})$$
By Remark \ref{rem:w-omega}, under the isomorphism \eqref{decat0}
the restriction from $\T^\domega$ to $\T^\dww$ is the transpose of the embedding \eqref{womega}.
\item
The integral $\bbZ$-weighted QH algebra ${}^0\widetilde\scrT^\dw$ and its cyclotomic quotient ${}^0\scrT^\dw$ 
associated with the Cartan datum $(\bfc,\bfd)$ still make sense without the coprime condition $c=1$.
Theorem \ref{thm:Main1} below remains valid in the non coprime case.
\end{enumerate}
\end{Remark}

\section{Coulomb branches of 4D, N=2 quiver gauge theories with symmetrizers}

Let $Q=(I,E)$ be the quiver defined in \S\ref{sec:NSCD}.
Fix $I$-graded vector spaces $V$ and $W$.
Let $\alpha$, $\lambda$, $\mu$ be as in \eqref{lam1}.

\subsection{BFN spaces with symmetrizers}

\subsubsection{Definition of the BFN spaces}\label{sec:BFN1}
For each vertex $i\in I$  we abbreviate 
$V_{i,\calO}=V_i\llb z_i\rrb$, $V_{i,K}=V_i\llp z_i\rrp$
and
\begin{align*}
\GL(V_i)_\calO&=\GL(V_i)(\bbC\llb z_i\rrb)
,\quad
G_\calO=\prod_{i\in I}\GL(V_i)_\calO,\\
\GL(V_i)_K&=\GL(V_i)(\bbC\llp z_i\rrp)
,\quad
G_K=\prod_{i\in I}\GL(V_i)_K.
\end{align*}
The affine Grassmannian of $\GL(V_i)$ is the ind-scheme 
$\Gr_{\GL(V_i)}$ representing the functor
$\GL(V_i)_K/\GL(V_i)_\calO$.
Set $\Gr=\prod_i\Gr_{\GL(V_i)}$ and
\begin{align}\label{GNOK}
\begin{split}
\Hom_I(W_\calO,V_\calO)&=\bigoplus_{i\in I}\Hom_{\bbC\llb z_i\rrb}(W_{i,\calO},V_{i,\calO}),\\
\End_E(V_\calO)&=\bigoplus_{i\to j}\Hom_{\bbC\llb z\rrb}(V_{i,\calO},V_{j,\calO}),\\
 N_\calO&=\Hom_I(W_\calO,V_\calO)\oplus\End_E(V_\calO),\\[2mm]
 N_K&=\Hom_I(W_K,V_K)\oplus\End_E(V_K).
\end{split}
\end{align}
Here, for any arrow $i\to j$ in $E$ we write 
$$z=(z_i)^{-c_{ij}}=(z_j)^{-c_{ji}}.$$
We equip $ N_\calO$ and $ N_K$ with their obvious classical scheme and ind-scheme structure.
To avoid confusions we may write
$ N_\calO^\lambda= N_\calO$
or
$ N_K^\lambda= N_K$.
The $G\times G_W$-action on $\calN$ in \S\ref{sec:Quivers} equips $ N_\calO$ and $ N_K$
with obvious actions of the groups $G_\calO\times G_W$ and $G_K\times G_W$.
Let $\calT$ be the classical $G_K\times G_W$-equivariant ind-scheme representing the twisted product
\begin{align}\label{T}\calT=G_K\times_{G_\calO} N_\calO.\end{align}
The first projection yields a $G_K$-equivariant morphism of ind-schemes
\begin{align}\label{p1}
p:\calT\to\Gr.
\end{align}
The $G_K\times G_W$-action on $ N_K$ and the inclusion 
$ N_\calO\subset  N_K$ yield a $G_K\times G_W$-equivariant morphism of ind-schemes
\begin{align}\label{pi}\pi:\calT\to  N_K.\end{align}
Following \cite{CWb}, the derived BFN space with symmetrizers is the 
fiber product in the category of dg ind-schemes
\begin{align}\label{R1}
\begin{split}
\xymatrix{\frakR\ar[r]^-{i_1}\ar[d]_-{i_2}&\Gr\times  N_\calO\ar[d]^-{j_1}\\
\calT\ar[r]^-{j_2}&\Gr\times  N_K}
\end{split}
\end{align}
where $j_1$ is the obvious inclusion and $j_2=p\times\pi$.
The $\infty$-stack $\frakR$ is ind-tamely presented and all maps in \eqref{R1} are almost ind-finitely presented ind-closed 
immersions by \cite[prop.~4.7]{CWb}.
Note that loc.~cit.~only deals with a symmetric Cartan matrix $\bfc$.
The classical truncation of $\frakR$ is the BFN space in \cite{NW23}.
It is the classical fiber product
$$\calR=\calT\times_{ N_K} N_\calO.$$ 
The $G_K\times G_W$-action on $\calT$ restricts to a $G_\calO\times G_W$-action on $\calR$ and $\frakR$.
There is an obvious closed embedding
\begin{align}\label{cl}\cl:\calR\to\frakR\end{align}
The group $\bbG_m$ acts on $\bbC\llp z_i\rrp$ such that $\tau\in\bbC^\times$ acts as follows
\begin{align}\label{loop1} \tau\cdot f(z_i)=f(z_i\tau^{d_i})\end{align}
We will use the convention in \cite[rem.~5.7]{NW23} for the loop rotation action on $ N_\calO$.
More precisely, choose $\sigma_i\in\frac12\bbZ$ for each $i\in I$ such that for all $i\to j$ we have
\begin{align}\label{sigma}\frac12\,b_{ij}=\sigma_i-\sigma_j-d_i+d_j.\end{align}
We modify the $\bbG_m$-action on $V_{i,K}$ and $W_{i,K}$
such that $\tau$ acts on 
$V_i\otimes z_i^n$ and $W_i\otimes z_i^n$ through 
\begin{align}\label{loop2}\tau^{\sigma_i+nd_i}\id.\end{align}
Taking the product over all $i\in I$, we get actions of $\bbG_m$  on  the group ind-scheme $G_K$ 
and the ind-schemes $\Gr$,  $ N_K$, etc.
We consider the groups 
\begin{align}\label{tGOK}
\begin{split}
&\widetilde G=G\times\bbG_m
,\quad
\tG_\calO=G_\calO\rtimes\bbG_m
,\quad
\tG_K=G_K\rtimes\bbG_m
,\quad
\widetilde T=T\times\bbG_m,\quad
\widetilde T_\calO=T_\calO\rtimes\bbG_m
,\\
&
\dot G=\widetilde G\times T_W,\quad
\dot G_\calO=\widetilde G_\calO\times T_W,\quad
\dot G_K=\widetilde G_K\times T_W,\quad
\dot T=\widetilde T\times T_W,\quad
\dot T_\calO=\widetilde T_\calO\times T_W.
\end{split}
\end{align}
The group scheme $\dot G_\calO$ acts on the scheme $ N_\calO$ and the ind-schemes $\calR$ and $\frakR$.
The group ind-scheme $\dot G_K$ acts on the ind-schemes  $ N_K$ and $\calT$ so that the
map $\pi$ in \eqref{pi} is $\dot G_K$-equivariant.
The group ind-scheme $\widetilde G_K$ acts on the ind-scheme $\Gr$.
We have an obvious isomorphism of stacks $$\calT/\dot G_K\cong N_\calO/\dot G_\calO.$$
By \cite[prop.~4.11]{CWb}, the stacks $ N_\calO/\dot G_\calO$ and $\frakR/\dot G_\calO$ are ind-tamely presented ind-geometric stacks,
see the terminology in \S\ref{sec:CAT}.
This allows us to use the formalism in loc.~cit. for coherent sheaves.
Further, by \cite[prop.~4.8]{CWb} there is
a Cartesian square of ind-geometric stacks
\begin{align}\label{R2}
\begin{split}
\xymatrix{\frakR/\dot G_\calO\ar[r]^-{\pi_1}\ar[d]_-{\pi_2}& N_\calO/\dot G_\calO\ar[d]\\
 N_\calO/\dot G_\calO\ar[r]& N_K/\dot G_K}
\end{split}
\end{align}
Here the map $\pi_1$ is the composition of $i_1$ with the second projection, 
and the map $\pi_2$ is the composition of $i_2$ with the quotient map 
$$\calT/\dot G_\calO\to\calT/\dot G_K= N_\calO/\dot G_\calO.$$
Further, the  maps $\pi_1$ 
and $\pi_2$ are ind-proper and almost ind-finitely presented.

\begin{Remark}\label{rem:H1} 
In the definition of $\Gr$, $\calT$, $\frakR$ or $\calR$ we may allow the group $G$ to vary.
Given a Levi subgroup $H\subset G$ with $H=\prod_{i\in I}H_i$ and $H_i \subset \GL(V_i)$, we define as above
\begin{align}\label{RH}
\Gr_H=H_K/H_\calO
,\quad
\calT_H=H_K\times_{H_\calO} N_\calO
,\quad
\frakR_H=\calT_H\times_{ N_K} N_\calO.
\end{align}
The obvious inclusion $\Gr_H\subset\Gr_G=\Gr$ yields ind-finitely presented closed immersions
\begin{align}\label{iota}
\calT_H\to\calT
,\quad
\frakR_H\to\frakR
,\quad
\calT_H/\dot H_\calO\to\calT/\dot H_\calO
,\quad
\frakR_H/\dot H_\calO\to\frakR/\dot H_\calO
\end{align}
Let $\iota$ denote any of these maps.
\end{Remark}

\subsubsection{Approximations of $\frakR$}
We will need various approximations of the ind-geometric stack $\frakR$. 
Fix $\tG_\O$-invariant presentations as ind-schemes
\begin{align}\label{colim}
\Gr = \colim_n \Gr_n
,\quad
\calT = \colim_n \calT_n
,\quad
\frakR =\colim_n \frakR_n
\end{align}
such that
$$\calT_n= \Gr_n\times_\Gr\calT
,\quad
\frakR_n = \Gr_n\times_\Gr\frakR.$$
The map $\frakR_n\to\frakR$ is an almost ind-finitely presented ind-closed immersion, 
since it is a base change of the immersion $\Gr_n\to\Gr$.
 Each $\calT_n$ is a classical scheme, because the map $\calT \to \Gr$ is flat.
 Thus, the map $\calT_n \to  N_K $ factors through the subscheme $t^{-k_n}  N_\calO\subset N_K$ for some $k_n$.
We deduce that 
$$\frakR_n = \colim_{k \geqslant k_n} \frakR_n^k,$$ where
the derived scheme $\frakR_n^k$ is defined by the 
Cartesian square below
\begin{align}\label{R3}
\begin{split}
\xymatrix{\frakR_n^k\ar[r]^-{i_1}\ar[d]_-{i_2}&\Gr_n\times  N_\calO\ar[d]^-{j_1}\\
\calT_n\ar[r]^-{j_2}&\Gr_n\times t^{-k} N_\calO}
\end{split}
\end{align}
Note that, although $\frakR_n$ is only an ind-scheme, the following are schemes
$$\calR_n=(\frakR_n)^\cl=(\frakR_n^k)^\cl$$ 
We can further write each $\frakR_n^k$ as a limit of finite type schemes. To do this we consider the quotient
$\calT_n^\ell \to \Gr_n$ of the bundle $\calT_n\to\Gr_n$ 
whose fiber over the point $g\cdot G_\calO$ in $\Gr_n$ is $ N_\calO/(g^{-1}t^\ell  N_\calO).$
We define $\frakR_n^{k,\ell}$ by the Cartesian diagram 
\begin{align}\label{R4}
\begin{split}
\xymatrix{\frakR_n^{k,\ell}\ar[r]^-{i_1}\ar[d]_-{i_2}&\Gr_n\times  N_\calO/t^\ell  N_\calO\ar[d]^-{j_1}\\
\calT_n^\ell\ar[r]^-{j_2}&\Gr_n\times (t^{-k} N_\calO/t^\ell  N_\calO)}
\end{split}
\end{align}
Then, we have the following limits along faithfully flat affine maps
\begin{align}\label{lim}
\calT_n=\lim_\ell\calT_n^\ell
,\quad
\frakR_n^k=\lim_\ell\frakR_n^{k,\ell}.
\end{align}

\subsubsection{The affine cell decomposition of the BFN space}
\label{sec:BFN2}
A cocharacter $\gamma\in\Lambda_V$  is an $I$-tuple 
of cocharacters $\gamma_i$ of the tori $T_{V_i}$.
It is dominant if each $\gamma_i$ is dominant.
Let $z^\gamma\in G_K$ be  the point given by
\begin{align}\label{zgamma}
z^\gamma=(z_i^{\gamma_i})
\end{align}
For each point $g\in G_K$ let $[g]=g\cdot G_\calO$ be the corresponding point in $\Gr$.
We abbreviate $[\gamma]=[z^\gamma].$
Let $\Gr_{\leqslant\gamma}$ be the closure in $\Gr$ of the $G_\calO$-orbit 
\begin{align}\label{Grg}
\Gr_\gamma=G_\calO\cdot [\gamma].
\end{align}
We say that the cocharacter $\gamma$ is minuscule if the orbit $\Gr_\gamma$ is closed in $\Gr$.
For every ind-scheme $X$ there is a unique ind-scheme $X_\red$ together with a monomorphism
$X_\red\subset X$ such that for any reduced affine scheme $S$ we have
$X_\red(S)=X(S)$. The injection $\Gr_\red\subset\Gr$ is  a closed immersion.
We have
\begin{align}\label{Grred}\Gr_\red=\bigsqcup_{\gamma\in\Lambda^+}\Gr_\gamma.\end{align}
For each dominant cocharacter $\gamma$ the map $p:\calR\to\Gr$ allows us to define the following schemes
\begin{align}\label{Rgamma}
\calR_{\gamma}=\Gr_{\gamma}\times_{\Gr}\calR
,\quad
\calR_{\leqslant\gamma}=\Gr_{\leqslant\gamma}\times_{\Gr}\calR
\end{align}
We define $\calT_{\gamma}$ and $\calT_{\leqslant\gamma}$ similarly.
From \eqref{Grred} we deduce that
\begin{align}\label{RTred}
\calR_\red=\bigsqcup_{\gamma\in\Lambda^+}\calR_\gamma
,\quad
\calT_\red=\bigsqcup_{\gamma\in\Lambda^+}\calT_\gamma.
\end{align}
The obvious inclusion $\Gr_\gamma\subset\Gr_{\leqslant\gamma}$ is an open immersion.
By base change it yields an open immersion
$\calR_\gamma\subset\calR_{\leqslant\gamma}$.
The map $p$ in \eqref{p1} restricts to a $G_\calO$-equivariant vector bundle
\begin{align}\label{p2}
p:\calR_\gamma\to\Gr_\gamma
\end{align}
whose fiber at $[\gamma]$ is the vector space
$ N_\calO\cap z^{-\gamma} N_\calO$ given by 
$$ N_\calO\cap z^{-\gamma} N_\calO= N_\calO\cap\Ad\!\big(z^\gamma\big)^{-1}( N_\calO).$$
The cocharacter lattice of the torus $T$ is 
$$X_*(T)=\bigoplus_{i\in I}\bigoplus_{r=1}^{a_i}\bbZ w_{i,r},$$
where the $w_{i,r}$'s are the basic cocharacters.
The fundamental coweights of  $G$ are
 \begin{align}\label{fundamental1}
 \omega_{i,r}=w_{i,1}+\cdots+w_{i,r}
 ,\quad
 \omega_{i,-r}=-w_0 (\omega_{i,r}).
 \end{align}
 where $i\in I$ and $r=1,\dots,a_i$.
 Set also $\omega_{i,0}=0$.
 Let $ \stackrel{r}{\subset}$ be the inclusion of a codimension $r$ subspace.
We have
\begin{align*}
\Gr_{\omega_{i,r}} &= \{L=(L_j)\in\Gr\,;\,L_j=V_{j,\calO}\,,\,z_iV_{i,\calO} \subset L_i \stackrel{r}{\subset} V_{i,\calO}\,,\,j\neq i\},\\
\Gr_{\omega_{i,-r}}&= \{ L=(L_j)\in\Gr\,;\,L_j=V_{j,\calO}\,,\,V_{i,\calO} \stackrel{r}{\subset} L_i\subset z_i^{-1}V_{i,\calO}\,,\,j\neq i\}.
\end{align*}
Similarly, we have the following sets of pairs 
\begin{align}\label{RT}
\begin{split}
\calT_{\omega_{i,\pm r}}&=\{(L,A,B)\,;\,L\in\Gr_{\omega_{i,\pm r}} \,,\,A(W_{\calO})\subset L\,,\,B(L)\subset L\},\\
\calR_{\omega_{i,r}} &= \{(L,A,B)\in\calT_{\omega_{i,r}}\,;\,B_{ji}(V_{i,\calO})\subset V_{j,\calO}\},\\
\calR_{\omega_{i,-r}} &= \{(L,A,B)\in\calT_{\omega_{i,-r}}\,;\, A_i(W_{i,\calO})\subset V_{i,\calO}\,,\,B_{ij}(V_{j,\calO})\subset V_{i,\calO}\},
\end{split}
\end{align}
where $L=(L_j)$ is a tuple of lattices, $A=(A_j)$ with $A_j$ a $\bbC\llb z_j\rrb$-linear map, and $B=(B_{ji})$ with $B_{j,i}$ a $\bbC\llb z\rrb$-linear map.
For any dimension vector $\nu = \sum_i r_i\delta_i$ with $-a_i \leqslant r_i \leqslant a_i$, set
$$\omega_\nu=\sum_{i\in I} \omega_{i,r_i}.$$
Hence, we have
\begin{align}\label{RGr}
\Gr_{\omega_\nu}=\prod_{i\in I}\Gr_{\omega_{i,r_i}}
,\quad
\calR_{\omega_\nu}=p^{-1}(\Gr_{\omega_\nu})
\end{align}
Note that the orbit $\Gr_{\omega_\nu}$ is closed.

\subsubsection{Fixed point locus in the affine Grassmannian}\label{sec:FP1}
For each pairs of coweights 
\begin{align}\label{gamma+rho}
\gamma\in\Lambda_\alpha\otimes\bbQ
,\quad
\rho\in\Lambda_\lambda\otimes\bbQ
\end{align} 
we view $\gamma$ and $\rho$ as $I$-tuples of coweights
$\gamma=(\gamma_i)$ and $\rho=(\rho_i)$
with $\gamma_i$ a coweight of $T_{V_i}$ and $\rho_i$ a coweight of $T_{W_i}$.
Thus $\gamma$ and $\rho$ are identified with tuples of rational numbers
$(\gamma_{i,r})$ and $(\rho_{i,s})$ in the obvious way.
Recall the set ${}^0\!I$ introduced in \eqref{0I}.
We define
\begin{align}\label{evenrhogamma}
\begin{split}
{}^0\!\Lambda_\alpha&=\{\gamma\in\Lambda_\alpha\otimes\bbQ\,;\,(i,2d_i\gamma_{i,r})\in{}^0\!I\,,\,\forall i,r\},\\
{}^0\!\Lambda_\lambda&=\{\rho\in\Lambda_\lambda\otimes\bbQ\,;\,(i,2d_i\rho_{i,s})\in{}^0\!I\,,\,\forall i,s\}.
\end{split}
\end{align}
We say that a coweight in ${}^0\!\Lambda_\alpha$ or ${}^0\!\Lambda_\lambda$ is even.
From now on we assume that
\begin{align}\label{0gamma+rho}
\gamma\in{}^0\!\Lambda_\alpha
,\quad
\rho\in{}^0\!\Lambda_\lambda
\end{align} 
In particular, this implies that
$(2d_i\gamma_i)$ and $(2d_i\rho_i)$ are cocharacters in $\Lambda_\alpha$ and $\Lambda_\lambda$.
Let the set $\dI$ be as in \S\ref{sec:IQHA}.
For each vertex $\dii=(i,r)$ in $\dI$ and each $k\in\bbZ$ we define
\begin{align}\label{VW'}
\begin{split}
&\dV_{\gamma,\dii}=\Ker(\varepsilon^{2d_i\gamma_i}-\varepsilon^r)
,\quad
\dW_{\rho,\dii}=\Ker(\varepsilon^{2d_i\rho_i}-\varepsilon^r)\end{split}
\end{align}
and
\begin{align}\label{dbfVW}
\begin{split}
&\dV^k_{\gamma,\dii}=\Ker(\tau^{2d_i\gamma_i}-\tau^k)
,\quad
\dW^k_{\rho,\dii}=\Ker(\tau^{2d_i\rho_i}-\tau^k)
\end{split}
\end{align}
where $\varepsilon=\exp(i\pi/d_i)$ and $\tau\in\bbC^\times$ is generic. 
Taking the sum over all vertices we get the $\dI$-graded vector spaces $\dV_\gamma$, $\dW_\rho$
and the sequences of $\dI$-graded vector subspaces $(\dV_\gamma^k)$, $(\dW_\rho^k)$.
We also consider the flag of $\dI$-graded vector spaces $(\dV_\gamma^{\leqslant k})$ such that
$\dV_\gamma^{\leqslant k}=\sum_{l\leqslant k}\dV_\gamma^l.$
If the coweights $\gamma$, $\rho$ are clear from the context we abbreviate 
\begin{align}\label{nogamma}
\dV=\dV_\gamma
,\quad
\dW=\dW_\rho
\quad
\dV^k=\dV_\gamma^k
,\quad
\dW^k=\dW_\rho^k
,\quad
\dV^{\leqslant k}=\dV_\gamma^{\leqslant k}.
\end{align}
Following \eqref{dlam1}, we denote the dimension vectors 
of the $\dI$-graded vectors space $\dV$, $\dW$ by
\begin{align}\label{dlam2}
\dalpha,\dlambda,\dmu\in\bbN\dI
\end{align} 
We will say that $\dalpha$, $\dlambda$, $\dmu$ refine $\alpha$, $\lambda$, $\mu$.
Let $\dG\subset G$ be the Levi subgroup  given by
\begin{align}\label{G'}
\dG=\prod_{\dii\in\dI}\GL(\dV_\dii)
\end{align}
We may view the torus $T$ as a maximal torus either in $G$ or in $\dG$.
So any coweight $\eta$ of $T$ decomposes as tuples of coweights
$(\eta_i)$ and $(\eta_\dii)$ of the tori $T_{V_i}$ and  $T_{\dV_\dii}$ where $i\in I$ and $\dii\in\dI$.
We set
\begin{align}\label{Lambda'}
\begin{split}
{}^0\!\Lambda_\dV
&=\{\eta\in{}^0\!\Lambda_\alpha\,;\,\dV_\eta=\dV\}.
\end{split}
\end{align}
Hence
${}^0\!\Lambda_\dV$ is the set of coweights 
$\eta\in{}^0\!\Lambda_\alpha$ such that $\eta_\dii\in \frac{r}{2d_i}+\Lambda_\alpha$ for all vertex $\dii=(i,r)$ in $\dI$.
Using this identification, for each coweight $\eta$  we define a cocharacter  $\hat\eta$ such that
\begin{align}\label{hat}
\hat\eta_\dii=\eta_\dii-\frac{r}{2d_i}.
\end{align} 
We say that $\eta$ is $\dG$-dominant if $\hat\eta_\dii$ is a dominant cocharacter of $\GL(\dV_\dii)$ for each vertex $\dii\in\dI$.
Let ${}^0\!\Lambda_\dV^+$ be the set of all $\dG$-dominant even coweights in ${}^0\!\Lambda_\dV$.
For each $\tau\in\bbC^\times$ and $i\in I$ we set
\begin{align}\label{zetagamma}
\begin{split}
\tau^{2d_i\gamma_i}&=\text{diag}(\tau^{2d_i\gamma_{i,1}},\tau^{2d_i\gamma_{i,2}},\dots,\tau^{2d_i\gamma_{i,a_i}}),\\
\tau^{2d_i\rho_i}&=\text{diag}(\tau^{2d_i\rho_{i,1}},\tau^{2d_i\rho_{i,2}},\dots,\tau^{2d_i\rho_{i,l_i}}).
\end{split}
\end{align}
Taking the product over all vertices in $I$, we define
\begin{align}\label{tau-gamma}
\tau^{2\gamma}=(\tau^{2d_i\gamma_i})\in T,\quad
\tau^{2\tilde\gamma}=(\tau^{2\gamma}\,,\,\tau^2)\in\widetilde T
,\quad
\tau^{2\rho}=(\tau^{2d_i\rho_i})\in T_W
,\quad
\tau^{2\dot\gamma}=(\tau^{2\tilde\gamma}\,,\,\tau^{2\rho})\in\dot T
\end{align}
Following \eqref{tau-gamma}, let $\varepsilon^{2\gamma}$ denote the $I$-tuple of matrices $(\varepsilon^{2d_i\gamma_i})$ in $G$.
By \eqref{VW'}, the group $\dG$ is the centralizer of $\varepsilon^{2\gamma}$ in $G$.
Let $\dP_\eta\subset\dG$ be the parabolic subgroup  such that
\begin{align}\label{P}\dP_\eta=\{g=(g_\dii)\in \dG\,;\,
 \lim_{\tau\to 0}\Ad(\tau^{-\hat\eta_\dii})(g_\dii) \text{\ exists\ in\ }\dG\ \text{for\ all}\ \dii\}
\end{align}
We consider the semidirect products  
\begin{align}\label{GP}
\widetilde\dG=\dG\rtimes\bbG_m
,\quad 
\widetilde\dP_\eta=\dP_\eta\rtimes\bbG_m
\end{align}
such that $\tau^2\in\bbG_m$
acts by conjugation by $\tau^{-2\gamma}\in\dG$.
There is an obvious group isomorphism 
\begin{align}\label{GPISOM}\widetilde\dG\to\dG\times\bbG_m
,\quad (g\zeta^{2\gamma},\zeta^2)\mapsto(g,\zeta^2)\end{align}
However, using the semidirect presentation is useful. Let 
\begin{align}\label{fpdg1}
\Gr^{\tilde\gamma}\subset\Gr
\end{align}
be the fixed point set of $\tilde\gamma$ in $\Gr$.
We define similarly the fixed point subgroups 
\begin{align}\label{GFP1}(\widetilde G_K)^{\tilde\gamma}\subset \widetilde G_K
,\quad
(G_K)^{\tilde\gamma}\subset  G_K
,\quad
(\widetilde G_\calO)^{\tilde\gamma}\subset\widetilde G_\calO
,\quad
(G_\calO)^{\tilde\gamma}\subset  G_\calO.\end{align}
For each  $\eta\in{}^0\!\Lambda_\dV$ we define the subset $\Gr^{\tilde\gamma,\eta}\subset \Gr^{\tilde\gamma}$ to be
\begin{align}\label{Greta}\Gr^{\tilde\gamma,\eta}=(G_K)^{\tilde\gamma}\cdot[\hat\eta-\hat\gamma]\end{align}
with its reduced scheme structure.

\begin{Lemma}[\cite{NW23}]\label{lem:GP}
Let $\gamma,\eta\in{}^0\!\Lambda_\dV$.
Let $\dv\in{}^0\P(\dalpha)$ be the dimension sequence of  $(\dV_\eta^k)$.
\hfill
\begin{enumerate}[label=$\mathrm{(\alph*)}$,leftmargin=8mm]
\item
The subgroup $\dP_\eta$ of $\dG$ is the stabilizer of the $\bbZ$-weighted flag $(\dV_\eta^{\leqslant k})$ in $\Fl_{\dv}$.
We have $\Fl_{\dv}\cong\dG/\dP_\eta$ as $\dG$-varieties.
\item
There are group isomorphisms
\begin{align}\label{FPG}
(\widetilde G_K)^{\tilde\gamma}\cong \widetilde\dG
,\quad
(G_K)^{\tilde\gamma}\cong\dG
, \quad
(\widetilde G_\calO)^{\tilde\gamma}\cong \widetilde\dP_\gamma
,\quad
(G_\calO)^{\tilde\gamma}\cong \dP_\gamma.
\end{align}
\item
The isomorphism \eqref{FPG}
identifies  the stabilizers in $(\widetilde G_K)^{\tilde\gamma}$ and $(G_K)^{\tilde\gamma}$
of the point $[\hat\eta-\hat\gamma]$ with the groups $\widetilde\dP_\eta$ and $\dP_\eta$,
yielding an isomorphism $\Gr^{\tilde\gamma,\eta}\cong\dG\,/\,\dP_\eta$ which intertwines the 
$(\widetilde G_K)^{\tilde\gamma}$-action with the $\widetilde\dG$-action.
\item
We have  $\Gr_\red^{\widetilde\gamma}=\bigsqcup_{\eta\in{}^0\!\Lambda_\dV^+}\Gr^{\tilde\gamma,\eta}$.
\end{enumerate}
\end{Lemma}

\begin{proof}
Note that, since $\gamma,\eta\in{}^0\!\Lambda_\dV$, we have $\dV_\gamma=\dV_\eta=\dV$. Part (a) is obvious.
To prove (b) we  observe that
\begin{align}\label{Ggamma}
(G_K)^{\tilde\gamma}&=\{g=(g_i(z_i\rrp\in G_K\,;\,
g_i(z_i)=\Ad(\tau^{2d_i\gamma_i})(g_i(z_i\tau^{2d_i}\rrp\,,\,\forall i\in I\,,\,\forall\tau\in\bbC^\times\}
 \end{align}
Fix $\tau\in\bbC^\times$ generic and $\varepsilon=\exp(i\pi/d_i)$.
For each $g\in (G_K)^{\tilde\gamma}$ we have 
\begin{align}\label{ggamma}
g_i(1)=\Ad(\varepsilon^{-2d_i\gamma_i})(g_i(1\rrp
,\quad
g_i(\tau)=\Ad(\tau^{-\hat\gamma_i})(g_i(1\rrp.
\end{align}
Since $\dG$ is the centralizer of $\varepsilon^{2\gamma}$ in $G$,
we deduce that the element $(g_i(1\rrp$  belongs to $\dG$.
Thus the assignment 
$g=(g_i(z_i\rrp\mapsto(g_i(1\rrp$ is a group homomorphism $(G_K)^{\tilde\gamma}\to\dG$.
It is easily seen to be invertible.
Let $\ev$ denote this isomorphism.
Since $\widetilde G_K=G_K\rtimes\bbG_m$, we have $(\widetilde G_K)^{\tilde\gamma}=(G_K)^{\tilde\gamma}\rtimes\bbG_m$.
Hence the isomorphism $ev:(G_K)^{\tilde\gamma}\to\dG$ lifts to a group isomorphism 
$\ev:(\widetilde G_K)^{\tilde\gamma}\to \dG\rtimes\bbG_m$
such that $\ev(g,\tau)=( \ev(g),\tau)$.
To identify the cocycle of $\dG\rtimes\bbG_m$, recall that
$\tau\cdot g=(g_i(\tau^{d_i}z_i\rrp$ by \eqref{loop1}, and that \eqref{tau-gamma} and \eqref{ggamma} yield
$$\tau^2\cdot\ev(g)=(g_i(\tau^{2d_i}\rrp=(\Ad(\tau^{-2d_i\hat\gamma_i})(g_i(1\rrp)=(\Ad(\tau^{-2d_i\hat\gamma_i})(g_i(1\rrp)
=\Ad(\tau^{-2\gamma})\ev(g).$$
We deduce that  $\dG\rtimes\bbG_m$ is the semidirect product $\widetilde\dG$ considered in \eqref{GP}.

To prove (c), note that for each $g\in (G_K)^{\tilde\gamma}$ we have
$$g\cdot[\hat\eta-\hat\gamma]=[\hat\eta-\hat\gamma]
\iff
\Ad\llp z_i)^{\hat\gamma_i-\hat\eta_i})(g_i(z_i\rrp\in\GL(V_i)_\calO
,\quad
\forall i\in I$$
Further, from \eqref{ggamma} we deduce that
$\Ad\llp z_i)^{\hat\gamma_i})( g_i(z_i\rrp=g_i(1)$.
Thus, we have
\begin{align*}
g\cdot[\hat\eta-\hat\gamma]=[\hat\eta-\hat\gamma]
&\iff \Ad(z_i^{-\hat\eta_i})(g_i(1\rrp\in GL(V_i)_\calO,\quad\forall i\in I\\
&\iff \ev(g)\in\dP_\eta
\end{align*}
where  the last equivalence is \eqref{P}.
For (d) we first observe that the fixed point locus $ \Gr^{\tilde\gamma}$ is given by
\begin{align*}
 \Gr^{\tilde\gamma}&
 =\{[g]\in\Gr\,;\,
 \tau^{2\widetilde\gamma}\cdot [g]=[g]\,,\,\forall\tau\in\bbC^\times\}.
 \end{align*}
In particular, if $[g]\in \Gr^{\tilde\gamma}$ then
$\varepsilon^{2\gamma}\cdot[g]=[g]$ where $\varepsilon$ is as above.
For any parabolic subgroup $\dP\subset G$ with Levi factor $\dG$
there is a morphism of ind-schemes $\dP_K/\dP_O\to\Gr$ which is a bijection on $\bbC$-points. 
This implies  that any element $[g]$ in $\Gr^{\tilde\gamma}$ belongs to the image of $\dG_K/\dG_O$ in $\Gr$. 
Thus, we may assume that $g\in\dG_K$ and that $(\tau^{\hat\gamma},\tau)\cdot [g]=[g]$ for each $\tau\in\bbC^\times$.
Then, the image in  $\Gr$
of the element $\phi\in\dG_K$ such that $\phi=z^{\hat\gamma}\cdot g$
is fixed by the loop $\bbC^\times$-action, i.e., we have $\tau\cdot[\phi]=[\phi]$ for all $\tau\in\bbC^\times$.
This implies that $[\phi]=[\Ad(h)(z^{\hat\eta})]$ for some elements $h\in\dG$ and $\hat\eta\in\Lambda^+$, 
where $z^{\hat\gamma}$ is as in \eqref{zgamma}.
We deduce that
$$[g]=z^{-\hat\gamma}\cdot h\cdot[\hat\eta]=\Ad(z^{-\hat\gamma})(h)\cdot[\hat\eta-\hat\gamma]$$
and it is easy to see that $\Ad(z^{-\hat\gamma})(h)\in (G_K)^{\tilde\gamma}$.
Therefore, we have
$\Gr_\red^{\widetilde\gamma}=\bigsqcup_\eta\Gr^{\tilde\gamma,\eta}$.
\end{proof}

\subsubsection{Fixed point locus of the BFN space}\label{sec:FP2}
Now, we consider the fixed point sets of $\dot\gamma$ given by
\begin{align}\label{fpdg}
( N_K)^{\dot\gamma}\subset N_K
,\quad
\calT^{\dot\gamma}\subset\calT
,\quad
\calR^{\dot\gamma}\subset\calR
\end{align}
By Lemma \ref{lem:GP}, the fixed point sets of the action
of $\dot\gamma$ on  $\calT_\red$ and $\calR_\red$ decompose in the following way
\begin{align}\label{TT}
\begin{split}
\calT_\red^{\dot\gamma}=\bigsqcup_{\eta\in{}^0\!\Lambda_\dV^+}\calT^{\dot\gamma,\eta}
,\quad
\calT^{\dot\gamma,\eta}\cong\Gr^{\tilde\gamma,\eta}\times_\Gr\calT^{\dot\gamma},\\
\calR_\red^{\dot\gamma}=\bigsqcup_{\eta\in{}^0\!\Lambda_\dV^+}\calR^{\dot\gamma,\eta}
,\quad
\calR^{\dot\gamma,\eta}\cong\Gr^{\tilde\gamma,\eta}\times_\Gr\calR^{\dot\gamma}
\end{split}
\end{align}
Our next goal is to describe $\calT^{\dot\gamma}$ and $\calR^{\dot\gamma}$.
Before to do this, recall that ${}^0\P(\dalpha)$ is the set of even sequences  of dimension vectors in $\bbN\dI$ with sum  $\dalpha$, 
see \eqref{0Palpha}.
Let $\dV=\dV_\gamma$ and $\dW=\dW_\rho$ as in \eqref{nogamma}, and
fix $\dalpha,\dlambda,\dmu\in\bbN\dI$ as in \eqref{dlam2}.
If it does not create any confusion, we abbreviate 
\begin{align}\label{0Lalpha}{}^0\!\Lambda_\dalpha={}^0\!\Lambda_\dV
,\quad
{}^0\!\Lambda_\dalpha^+={}^0\!\Lambda_\dV^+\end{align}
There is a bijection 
\begin{align}\label{bijection}
{}^0\!\Lambda_\dalpha^+\cong{}^0\P(\dalpha)
,\quad
\gamma\mapsto\dv_\gamma
\end{align} 
such that $\dv_\gamma$ is the sequence $(\dv_{\gamma,k})$
where $\dv_{\gamma,k}$ is the dimension vector of the $\dI$-graded vector space
$\dV_\gamma^k$ in \eqref{dbfVW}.
Similarly, let 
\begin{align}\label{bijection-rho}
\dw_\rho\in {}^0\P(\dlambda)
\end{align}
be the sequence $\dw_\rho=(\dw_{\rho,k})$ where $\dw_{\rho,k}$ is the dimension vector of the $\dI$-graded vector space
$\dW_\rho^k$ in \eqref{dbfVW}.
Let $\dQ$ be the quiver considered in \S\ref{sec:NSCD}, and 
$N_\dmu^\dlambda$ be the representation space in \eqref{dN}.
We abbreviate
\begin{align}\label{Xrhogamma}
\scrX_\gamma^\rho=\scrX_{\dv_\gamma}^{\dw_\rho},
\end{align}
where the right hand side is as in \eqref{XV}.
The map \eqref{XN1} yields a projective morphism $\scrX_\gamma^\rho\to N_\dmu^\dlambda$.
If $\bfc\neq G_2$, then for each coweight $\eta\in{}^0\!\Lambda_\alpha$ we consider 
the $\dP_\eta$-subvariety $ N_\eta^\rho\subset N_\dmu^\dlambda$ given by
\begin{align}\label{dNeta}
 N_\eta^\rho=\{(A,B)\in N_\dmu^\dlambda\,;\,A(\dW_\rho^k)\subset\dV_\eta^{\leqslant k}\,,\,
B(\dV_\eta^{\leqslant k})\subset\dV_\eta^{< k}\}
\end{align}
If $\bfc=G_2$, we define $ N_\eta^\rho$ similarly, modulo a grading shift given in the proof below.
Given an affine group $G$, a closed subgroup $H$ and a finite dimensional representation $N$ of $H$, the induction from
$H$ to $G$ associates to any finite dimensional representation $V$ of $H$ a $G$-equivariant vector bundle $\ind(V)$ on the variety 
$G\times_HN$ such that
$$\ind(V)=\big(G\times_H(V\otimes\scrO_N)\big).$$
Taking the Grothendieck groups, we get a linear map
\begin{align}\label{induction}\ind:R_H\to K^G(G\times_HN).\end{align}
Taking $G=\dG$, $H=\dP_\eta$, $N= N_\eta^\rho$ and composing with the restriction, this yields a map
\begin{align}\label{inductionb}
\ind:R_{\dP_\eta}\to K^{\dP_\gamma}(\dG\times_{\dP_\eta} N_\eta^\rho).
\end{align}
Similarly, the induction yields a map
\begin{align}\label{inductiona}
\ind:R_{\widetilde G_\calO}\to K^{\widetilde G_\calO}(\calT)
\end{align}
hence, composing with the restriction, a map
\begin{align}\label{inductiona1}
\ind:R_{G_\calO}\to K^{(G_\calO)^{\tilde\gamma}}(\calT^{\dot\gamma,\eta})
\end{align}
We need the following result, which strengthens the computations in \cite[App.~B]{NW23} (where $W=0$).

\begin{Proposition}\label{prop:fixedlocus}
Let $\gamma,\eta\in{}^0\!\Lambda_\dalpha^+$. 
\hfill
\begin{enumerate}[label=$\mathrm{(\alph*)}$,leftmargin=8mm]
\item
There is an isomorphism of $\dG\ltimes\bbG_m$-varieties
$\scrX_\eta^\rho\cong\dG\times_{\dP_\eta} N_\eta^\rho$.
\item
There is an isomorphism of varieties
$\calT^{\dot\gamma,\eta}\cong\scrX_\eta^\rho$
which intertwines the $(\widetilde G_K)^{\tilde\gamma}$-action on $\calT^{\dot\gamma,\eta}$ and the $\widetilde\dG$-action 
on $\scrX_\eta^\rho$. 
For each $V\in R_{G_\calO}$, this isomorphism identifies the class $\ind(V)\in K^{(G_\calO)^{\tilde\gamma}}(\calT^{\dot\gamma,\eta})$ 
with the class $\ind(V_\eta)\in K^{\dP_\gamma}(\scrX_\eta^\rho)$, where $V_\eta\in R_{\dP_\eta}$ is the restriction
of $V$ to the group $(G_\calO)^{\tilde\eta}=\dP_\eta$ under the isomorphism \eqref{FPG}.
\item
There is a commutative diagram 
$$\xymatrix{
\calR^{\dot\gamma,\eta}\ar[r]^-{i_2}\ar[d]_-{\pi_1}&\calT^{\dot\gamma,\eta}\ar@{=}[r]\ar[d]_-\pi&\scrX_\eta^\rho\ar[d]\\
 N_\gamma^\rho\ar[r]&( N_K)^{\dot\gamma}\ar@{=}[r]&N_\dmu^\dlambda}$$
The right square is equivariant respectively to the group isomorphism $(\widetilde G_K)^{\tilde\gamma}\cong\widetilde\dG$,
and the left square  respectively to the group isomorphism  $(\widetilde G_\calO)^{\tilde\gamma}\cong\widetilde\dP_\gamma$.
\end{enumerate}
\end{Proposition}

\begin{proof}
We extend the $\dG$-action on $\scrX_\eta^\rho$ and $\dG\times N_\eta^\rho$
to an action of the semidirect product $\widetilde\dG$ in \eqref{Ggamma}
such that $\tau^2\in\bbG_m$ acts by the diagonal action of the element $\tau^{-2\gamma}\in\dG$.
Part (a) follows from the definitions \eqref{X} and \eqref{XV}.
Let us concentrate on (b).
Recall that $\dV=\dV_\gamma$ and $\dW=\dW_\rho$.
We first concentrate on the right isomorphism.
To prove it, we  consider a tuple $A=(A_i)$ in $\Hom_I(W_K,V_K)$, and we set
\begin{align}\label{A}
A_i=\sum_{n\in\bbZ}A_i^{(n)}\otimes z_i^n
,\quad
\ev(A_i)=\sum_{n\in\bbZ}A_i^{(n)}
,\quad
A_i^{(n)}\in \Hom(W_i,V_i).
\end{align}
Hence $A$ is fixed by the cocharacter $\dot\gamma$ if and only if
for each $i\in I$ and $k,n\in\bbZ$ we have
\begin{align}\label{An}
A_i^{(n)}(\dW^k)\subset  \dV^{k-2d_in}.
\end{align}
This yields a map which is easily seen to be an isomorphism using \eqref{dbfVW}
\begin{align}\label{evA}
\ev:\Hom_I(W_K,V_K)^{\dot\gamma}\to\Hom_{\dI}(\dW,\dV)
,\quad
A_i\mapsto\ev(A_i)
\end{align}
Next, fix $B=(B_{ji})$ in $\End_E(V_K)$.
Note that
\begin{align}\label{EndV}
\End_E(V_K)=\bigoplus_{\substack{i\to j\\c_{ij}=-1}}\Hom(V_i,V_j)\llp z_j\rrp
\oplus\bigoplus_{\substack{i\to j\\c_{ij}<-1}}\Hom(V_i,V_j)\llp z_i\rrp
\end{align}
We write 
$\ev(B_{ji})=\sum_nB_{ji}^{(n)}$ where
$B_{ji}\in \Hom(V_i,V_j)$ is given by
\begin{align}\label{B}
B_{ji}=\begin{cases}
\sum_{n\in\bbZ}B_{ji}^{(n)}\otimes z_j^n&\text{if}\ c_{ij}=-1\\
\sum_{n\in\bbZ}B_{ji}^{(n)}\otimes z_i^n&\text{if}\  c_{ij}<-1
\end{cases}
\end{align}  
By \eqref{sigma} and \eqref{loop2}, the tuple $B$ is fixed by $\dot\gamma$ if and only if 
for each $i, j\in I$ and $k,n\in\bbZ$ we have 
\begin{align}\label{Bn}
\begin{cases}
B_{ji}^{(n)}( \dV^k)\subset  \dV^{k-2(n+1)d_j+d_i}&\text{if}\ c_{ij}=-1,\\[2mm]
B_{ji}^{(n)}( \dV^k)\subset  \dV^{k-2(n-1)d_i-3d_j}&\text{if}\ c_{ij}<-1.
\end{cases}
\end{align}
Thus,  we have the isomorphism
\begin{align}\label{evB}
\ev:\End_E(V_K)^{\dot\gamma}\to\End_{\dE}(\dV)
,\quad
B_{ji}\mapsto\ev(B_{ji})
\end{align}
From \eqref{N}, \eqref{evA} and \eqref{evB} we deduce that the map $ev$ yields an isomorphism
\begin{align}\label{Ngamma}(N_K)^{\dot\gamma}\cong N_\dmu^\dlambda.\end{align}
Now, we consider the upper right isomorphism in (b). Recall that
$$\calT=\{([g],A,B)\in\Gr\times  N_K\,;\,(A,B)\in g\cdot N_\calO\}.$$
By Lemma \ref{lem:GP} and \eqref{Ngamma}  we have an isomorphism
$$\ev:(\Gr\times N_K)^{\dot\gamma}\to\bigsqcup_{\eta\in{}^0\!\Lambda_\dV^+}\Gr^{\tilde\gamma,\eta}\times N_\dmu^\dlambda$$
Fix a tuple $([\hat\eta-\hat\gamma],A,B)$ in the right hand side.
We claim that 
\begin{align}\label{etagammaNO}
\ev\big\llp z^{\hat\eta-\hat\gamma}\cdot N_\calO)\cap( N_K)^{\dot\gamma}\big)= N_\eta^\rho.
\end{align}
This yields an isomorphism
\begin{align*}
\calT^{\dot\gamma,\eta}&\cong\dG\times_{\dP_\eta} N_\eta^\rho\cong\scrX_\eta^\rho
\end{align*}
from which part (b) follows. 
Let us concentrate on (c).
Setting $\eta=\gamma$, the map $ev$ yields an isomorphism
$$( N_\calO)^{\dot\gamma}\cong N_\gamma^\rho.$$
Hence the map $\pi_1:\calR\to N_\calO$ factorizes through a map
$$\pi_1:\calR^{\dot\gamma,\eta}\to N_\gamma^\rho$$

Now, we concentrate on the proof of \eqref{etagammaNO}.
Fix $(A,B)$ in $( N_K)^{\dot\gamma}$.
First, let $A$, $A_i^{(n)}$ and $\ev(A_i)$ be as in \eqref{A}.
Then
$z^{\hat\gamma-\hat\eta}\cdot  A$ is the sum over all integers $l$ of the operators
$$A_i^{(n)}\otimes z_i^{n+\hat\gamma_i-\hat\eta_i}:
\dW^k\to \Big( \dV^{k-2d_in}\cap\dV_\eta^l\Big)\otimes z_i^{\frac{1}{2d_i}(k-l)}$$
We have
\begin{align}\label{NO1}
\Hom_I(W_\calO,V_\calO)=\bigoplus_{i\in I}\Hom(W_i,V_i)\llb z_i\rrb.
\end{align}
We deduce that
\begin{align}\label{AA}
z^{\hat\gamma-\hat\eta}\cdot A\in  N_\calO\iff\ev(A)(\dW^k)\subset\dV_\eta^{\leqslant k}
,\quad
\forall k\in\bbZ.\end{align}
Next, let $B$, $B_{ji}^{(n)}$ and $\ev(B_{ji})$ be as in \eqref{B}.
Then $z^{\hat\gamma-\hat\eta}B$ is the sum over all integers $l$, $l'$ of the operators 
\begin{align*}
B_{ji}^{(n)}\otimes z_j^{n+\hat\gamma_j-\hat\eta_j-\frac{d_i}{d_j}(\hat\gamma_i-\hat\eta_i)}:
 \dV^k\cap\dV_\eta^l\to\Big( \dV^{k-2(n+1)d_j+d_i}\cap\dV_\eta^{l'}\Big)\otimes (z_j)^{-1+\frac{1}{2d_j}(l-l'+d_i)}
\ \text{if}\  c_{ij}=-1,\\
B_{ji}^{(n)}\otimes z_i^{n+\frac{d_j}{d_i}(\hat\gamma_j-\hat\eta_j)-\hat\gamma_i+\hat\eta_i}:
 \dV^k\cap\dV_\eta^{l}\to\Big( \dV^{k-2(n-1)d_i-3d_j}\cap\dV_\eta^{l'}\Big)\otimes (z_i)^{1+\frac{1}{2d_i}(l-l'-3d_j)}
\ \text{if}\  c_{ij}<-1,
\end{align*} 
Since
\begin{align}\label{NO2}
\End_E(V_\calO)=\bigoplus_{\substack{i\to j\\c_{ij}=-1}}\Hom(V_i,V_j)\llb z_j\rrb
\oplus\bigoplus_{\substack{i\to j\\c_{ij}<-1}}\Hom(V_i,V_j)\llb z_i\rrb z_i^{1+c_{ij}},
\end{align}
we deduce that $z^{\hat\gamma-\hat\eta}\cdot B\in  N_\calO$ if and only if the expression above vanishes for 
\hfill
\begin{enumerate}[label=$\mathrm{(\alph*)}$,leftmargin=8mm,itemsep=1mm]
\item[$\bullet$]
$l'\notin l+d_i-2d_j-2d_j\bbN$ if $c_{ij}=-1$,
\item[$\bullet$]
$l'\notin l-d_j-2d_i\bbN$ if $c_{ij}<-1$.
\end{enumerate}
First, assume that $\bfc\neq G_2$.
A case-by-case checking using the definition of the quiver $\dQ$ in \S\ref{sec:NSCD} yields
\begin{align}\label{BB1}
z^{\hat\gamma-\hat\eta}\cdot B\in  N_\calO\iff\ev(B)(\dV_\eta^{\leqslant l})\subset\dV_\eta^{< l}
,\quad
\forall l\in\bbZ
\end{align}
Let us check the case $\bfc=B_n$, the other ones are similar.
We set $I=\{1,\dots,n-1,n\}$ with $\bfd=(d_1,\dots,d_{n-1},d_n)=(2,\dots,2,1)$.
If $i\neq n$ then $d_i=d_j=2$ and $c_{ij}=-1$, hence
\begin{align*}
z^{\hat\gamma-\hat\eta}\cdot B_{ji}\in  N_\calO
&\iff\ev(B_{ji})(V_i\cap\dV_\eta^{l})\subset V_j\cap\dV_\eta^{l-2-4\bbN},\quad\forall l\in\bbZ\\
&\iff\ev(B_{ji})(V_i\cap\dV_\eta^{\leqslant l})\subset V_j\cap\dV_\eta^{<l}
,\quad
\forall l\in\bbZ
\end{align*}
Else, we have $i=n$, $j=n-1$, thus $d_i=1$, $d_j=2$, $c_{ij}<-1$, hence
\begin{align*}
z^{\hat\gamma-\hat\eta}\cdot B_{ji}\in  N_\calO
&\iff\ev(B_{ji})(V_i\cap\dV_\eta^{l})\subset V_j\cap\dV_\eta^{l-2-2\bbN},\quad\forall l\in\bbZ\\
&\iff\ev(B_{ji})(V_i\cap\dV_\eta^{\leqslant k})\subset V_j\cap\dV_\eta^{<l}
,\quad
\forall l\in\bbZ
\end{align*}
If $\bfc=G_2$ we get instead
\begin{align}\label{BB2}
z^{\hat\gamma-\hat\eta}\cdot B\in  N_\calO\iff\ev(B)(\dV_\eta^{\leqslant l})\subset\dV_\eta^{<l-1}
,\quad
\forall l\in\bbZ
\end{align}
In this case we have $I=\{1,2\}$ with $\bfd=(d_1,d_2)=(3,1)$.
We shift the $\bbZ$-gradings $(\dV_\eta^l)$, $(\dW^l)$ such that 
$\dV_{\eta,\dii}^l$, $\dW_\dii^l$ are replaced by $\dV_{\eta,\dii}^{l+2}$, $\dW_\dii^{l+2}$ 
if $\dii$ is the vertex $(2,1)\in\dI$ and are unchanged otherwise.
Then, the condition \eqref{AA} is unchanged and 
\eqref{BB2} is changed to \eqref{BB1}.
The relation \eqref{etagammaNO} follows from \eqref{dNeta}, \eqref{AA} and \eqref{BB1}.
Finally, the left square in (b) follows from \eqref{R1} and the computations above.
\end{proof}

\subsection{Coulomb branches with symmetrizers}\label{sec:CB}
\subsubsection{The monoidal product on the BFN space}\label{sec:product}
Replacing $\GL(V_i)$ by the 2-fold cover $\bfG\bfL(V_i)$ 
obtained by adding a square root of the determinant, and $\bbG_m$ by its 2-fold cover $\bGm$
we define the groups
$\bfG$, $\bfG_\calO$ and $\bfG_K$ as in \eqref{G} and \eqref{GNOK}, and the groups
$\tbG$, $\tbG_\calO$, $\tbG_K$, $\dot\bfG$, $\dot\bfG_\calO$ and $\dot\bfG_K$ as in \eqref{tGOK}.
We define similarly the tori $\bfT$,  $\tbT$ and $\dbT$.
Composing the $\dot G_\calO$-action on $\calR$, $\frakR$ and the $\dot G_K$-action on $\calT$ in
\S\ref{sec:BFN1} with the coverings $\dbG_\calO\to\dot G_\calO$ and $\dbG_K\to\dot G_K$, we get 
actions of $\dbG_\calO$ and $\dbG_K$.
We have an obvious isomorphism of stacks 
$$\calT/\dbG_K\cong N_\calO/\dbG_\calO.$$
The $\infty$-stack $ N_\calO/\dbG_\calO$ and $\frakR/\dbG_\calO$ are ind-tamely presented ind-geometric stack,
see \S\ref{sec:BFN1}.
Thus the stable $\infty$-category $\Coh^{\dot\bfG_\calO}(\frakR)$ is well-defined.
We refer to Appendix \ref{sec:CAT} for the terminology.
By \cite[prop.~5.7]{CWb} there is a monoidal product 
$$
\Coh^{\dbG_\calO}(\frakR)\times\Coh^{\dbG_\calO}(\frakR)\to\Coh^{\dbG_\calO}(\frakR)
,\quad
(\calF,\calG)\mapsto\calF\star\calG.
$$ 
Let us briefly recall the definition of this product.
There is an obvious map
\begin{align}\label{delta}\delta:\frakR/\dbG_\calO\times_{ N_\calO/\dbG_\calO}\frakR/\dbG_\calO\to
\frakR/\dbG_\calO\times\frakR/\dbG_\calO.\end{align}
The map $\delta$ has coherent pullback.
By \eqref{R2} there is an isomorphism of ind-geometric stacks
\begin{align}\label{R}
\frakR/\dbG_\calO\cong  N_\calO/\dbG_\calO\times_{ N_K/\dbG_K} N_\calO/\dbG_\calO.
\end{align}
Hence, the projection to the first and third factors along the second one yields a map
\begin{align}\label{m}
\begin{split}
m:\frakR/\dbG_\calO\times_{ N_\calO/\dbG_\calO}\frakR/\dbG_\calO
&=\, N_\calO/\dbG_\calO\times_{ N_K/\dbG_K} N_\calO/\dbG_\calO
\times_{ N_K/\dbG_K} N_\calO/\dbG_\calO\\
&\to N_\calO/\dbG_\calO\times_{ N_K/\tbG_K} N_\calO/\dbG_\calO\\
&=\,\frakR/\dbG_\calO
\end{split}
\end{align}
By \cite[\S 5.1]{CWb}, the map $m$ is ind-proper and almost ind-finitely presented.
Hence, the pushforward by $m$ preserves coherence by \S\ref{sec:B.k}.
The monoidal product $\star$ in $\Coh^{\dot\bfG_\calO}(\frakR)$ is given by
\begin{align}\label{*}
\star=m_*\circ \delta^*.
\end{align}
Let $\calR_0=\calR_\gamma$ in
\eqref{Rgamma} with $\gamma=0$.
By \cite[prop.~5.7]{CWb},
the monoidal unit is the object 
\begin{align}\label{unit}{\bf1}=\cl_*(\scrO_{\calR_0/\dbG_\calO}).
\end{align}

\subsubsection{BFN spaces for Levi subgroups}\label{sec:H2}
Our definition of $\star$ is taken from \cite{CWb}.
It is the same as the product in \cite{BFNa} by \cite[prop.~5.13]{CWb}.
For a future use, let us briefly recall a few properties of $\star$ following \cite{BFNa}.
Let $H\subset G$ be a Levi subgroup. 
As in \eqref{R2}, there are isomorphisms of ind-geometric stacks
\begin{align*}
\frakR_H/\dbH_\calO&\cong  N_\calO/\dbH_\calO\times_{ N_K/\dbH_K} N_\calO/\dbH_\calO
\\
\frakR/\dbH_\calO&\cong  N_\calO/\dbH_\calO\times_{ N_K/\dbG_K} N_\calO/\dbG_\calO
\end{align*}
They yield the following diagrams
\begin{gather*}
\xymatrix{
\frakR_H/\dbH_\calO\times\frakR_H/\dbH_\calO
&\ar[l]
\frakR_H/\dbH_\calO\times_{ N_\calO/\dbH_\calO}\frakR_H/\dbH_\calO
\ar[r]&
\frakR_H/\dbH_\calO}\\
\xymatrix{
\frakR/\dbH_\calO\times\frakR/\tbG_\calO
&\ar[l]
\frakR/\dbH_\calO\times_{ N_\calO/\dbG_\calO}\frakR/\dbG_\calO
\ar[r]&
\frakR/\dbH_\calO}\\
\xymatrix{
\frakR_H/\dbH_\calO\times\frakR/\dbH_\calO
&\ar[l]
\frakR_H/\dbH_\calO\times_{ N_\calO/\dbH_\calO}\frakR/\dbH_\calO
\ar[r]&
\frakR/\dbH_\calO
}
\end{gather*}
We define as in \eqref{*} the monoidal products 
\begin{align*}
&\star:\Coh^{\dbH_\calO}(\frakR_H)\times\Coh^{\dbH_\calO}(\frakR_H)\to\Coh^{\dbH_\calO}(\frakR_H),\\
&\star:\Coh^{\dbG_\calO}(\frakR)\times\Coh^{\dbG_\calO}(\frakR)\to\Coh^{\dbG_\calO}(\frakR)
\end{align*} 
and the functor 
\begin{align*}
\star:\Coh^{\widetilde\bfH_\calO}(\frakR_H)\otimes 
\Coh^{\dbH_\calO}(\frakR)\otimes 
\Coh^{\dbG_\calO}(\frakR)\to
\Coh^{\dbH_\calO}(\frakR)
\end{align*}
They satisfy the obvious associativity axioms.
Let $\bf1$ denote the objects
$${\bf1}=\cl_*(\scrO_{\calR_{H,0}/\dbH_\calO}),\quad
\bf1=\cl_*(\scrO_{\calR_0/\dbH_\calO}),\quad
{\bf1}=\cl_*(\scrO_{\calR_0/\dbG_\calO})$$
in
$\Coh^{\dbH_\calO}(\frakR_H)$, $\Coh^{\dbH_\calO}(\frakR)$ and 
$\Coh^{\dbG_\calO}(\frakR)$.
Taking the K-theory yields a map
\begin{align*}
\star:K^{\dbH_\calO}(\calR_H)\otimes 
K^{\dbH_\calO}(\calR)\otimes 
K^{\dbG_\calO}(\calR)\to
K^{\dbH_\calO}(\calR)
\end{align*}
The pushforward by the map $\iota$ in \eqref{iota} coincides with the map
\begin{align}\label{iota*}
\iota_*:K^{\dbH_\calO}(\calR_H)\to K^{\dbH_\calO}(\calR)
,\quad
c\mapsto c\star\bf1
\end{align}
The Weyl group $\W$ of $G$ acts by algebra automorphisms of the left hand side.
The map $\iota_*$ restricts to an algebra homomorphism
$$K^{\dbH_\calO}(\calR_H)^{\W}\to K^{\dbH_\calO}(\calR)^{\W}=K^{\dbG_\calO}(\calR)$$
The  map $\iota_2$ in \eqref{R1} and the zero section of the bundle $\calT\to\Gr$ yield the diagram
$$\xymatrix{\frakR/\dbG_\calO\ar[r]^-{\iota_2}&\calT/\dbG_\calO&\ar[l]_-{\sigma_2}\Gr/\dbG_\calO}.$$ 
We equip the complexified Grothendieck group $K^{\dbG_\calO}(\Gr)$ with an algebra structure as above.
The map 
$$z^*=\sigma_2^*\circ (i_2)_*:K^{\dbG_\calO}(\calR)\to K^{\dbG_\calO}(\Gr),$$
is an algebra homomorphism, see \cite[\S 5.3]{CWb}, or \cite[lem.~5.11]{BFNa} for a proof in homology.
Setting $G=H$, we get in the same way an algebra homomorphism
\begin{align}\label{z*}
z^*:K^{\dbH_\calO}(\calR_H)\to K^{\dbH_\calO}(\Gr_H).
\end{align}
All results above remains true if we replace everywhere the groups $\dbH_\calO$, $\dbG_\calO$ by 
$\tbH_\calO$, $\tbG_\calO$.

\subsubsection{The algebra $\calA_{\mu,R}^\lambda$}
Let $K^{\dot\bfG_\calO}(\frakR)$ and $K^{\dot\bfG_\calO}(\calR)$ be the complexified 
Grothendieck groups of the stable $\infty$-categories $\Coh^{\dot\bfG_\calO}(\frakR)$ and
$\Coh^{\dot\bfG_\calO}(\calR)$.
We will identify the groups 
$K^{\dot\bfG_\calO}(\calR)\cong K^{\dot\bfG_\calO}(\frakR)$
by the pushforward by the classical embedding cl in \eqref{cl}.

\begin{Definition}\label{def:AR}
Set $\calA_{\mu,R}^\lambda=K^{\dot\bfG_\calO}(\calR)$.
We equip  $\calA_{\mu,R}^\lambda$ with the associative $R_{\bGm\times T_W}$-algebra structure given by the monoidal product $\star$.
\end{Definition}

There is an $R_{\bGm\times T_W}$-algebra isomorphism
\begin{align*}
\calA_{\mu,R}^\lambda
=K^{\dot G_\calO}(\calR)\otimes_{R_{\dot G_\calO}}R_{\dot\bfG_\calO}
=K^{\dot G_\calO}(\calR)[\bfzeta,\det(V_i)^{\frac12}]
\end{align*}
where
\begin{align}\label{zeta}R_{\bbG_m}=R[q,q^{-1}],\quad R=R_{\bGm}=R[\bfzeta,\bfzeta^{-1}],\quad \bfzeta=q^{\frac12}\end{align}
In the following we will omit the 2-fold covers $\tbG$, $\tbT$ and $\bGm$ to simplify the notation.
Let 
$$\otimes:R_{\dot G}\otimes\calA_{\mu,R}^\lambda\to\calA_{\mu,R}^\lambda$$
denote the $R_{\dot G}$-action given by the tensor product of 
$\dot G_\calO$-equivariant coherent sheaves with representations of the group $\dot G_\calO$.
The monoidal product is $R_{\bbG_m\times T_W}$-linear. It is not $R_G$-linear.
A finite dimensional representation $V$ of $G$ extends trivially to a representation of $G_\calO$.
The induction \eqref{inductiona} yields map
\begin{align}\label{inductiona2}\ind:R_{\dot G}\to \calA_{\mu,R}^\lambda.\end{align}
We fix $\zeta\in\bbC^\times$ which is not a root of unity.
Let $(-)|_\zeta$ be the base change along the homomorphism 
\begin{align}\label{bczeta}
(-)|_\zeta:R\to\bbC,\  \bfzeta\mapsto\zeta.
\end{align}
The specialization at the element $\zeta^{2\rho}$ in \eqref{zetagamma}
yields an algebra homomorphism
\begin{align}\label{bcrho}(-)|_\rho:R_{T_W}\to\bbC.\end{align}
By base change along the maps \eqref{bczeta} and \eqref{bcrho}, we get the following algebras
\begin{align}\label{zetarho}
\calA_\mu^\lambda=\calA_{\mu,R}^\lambda\big|_\zeta
,\quad
\calA_{\mu,R}^\rho=\calA_{\mu,R}^\lambda|_\rho
,\quad
\calA_\mu^\rho=\calA_\mu^\lambda\big|_\rho.
\end{align}
We will need the following basic facts, which are similar to results in \cite{BFNa}, \cite{BFNb}.
We abbreviate
\begin{align}\label{AR}
\begin{split}
\calA_{\gamma,R}^\lambda=K^{\dot\bfG_\calO}(\calR_\gamma),\quad
\calA_{\leqslant\!\gamma,R}^\lambda=K^{\dot\bfG_\calO}(\calR_{\leqslant\gamma})
\end{split}
\end{align}

\begin{Proposition}\label{prop:basic}
Let $f \in R_{\dot G}$ and $x\in\calA_{\mu,R}^\lambda$.
\hfill
\begin{enumerate}[label=$\mathrm{(\alph*)}$,leftmargin=8mm,itemsep=1mm]
\item
$\calA_{\mu,R}^\lambda$ is free as an $R_{\bbG_m\times T_W}$-module.
\item
$\calA_{\mu,R}^\lambda|_{\zeta=1}$ is a commutative algebra.
\item
$\calA_{\mu,R}^\lambda=\bigcup_\gamma \calA_{\leqslant\!\gamma,R}^\lambda$
 is a filtered algebra.
\item
The monoidal product on $\calA_{\mu,R}^\lambda$ is $R_{\dot G}$-linear
in the first variable. We have
$(f\otimes{\bf1})\star x=f\otimes x$.
\item
$x\star (f\otimes{\bf1})=x\otimes \ind(f)$,  where the right hand side is the tensor product by the vector bundle $\ind(f)$.
\end{enumerate}
\end{Proposition}

\begin{proof}
Arguing as in \cite[\S 5.5]{CG}, to prove (a) it is enough to observe that $\calR$ has a decomposition into 
$\dot T$-invariant affine algebraic cells and that 
$K^{\dot G_\calO}(\calR)=K^{\dot T_\calO}(\calR)^\W$.
Part (b) is proved as in \cite[prop.~5.15]{BFNa}, \cite[\S B]{BFNc} and Part (c) as in as in \cite[prop.~6.1]{BFNa}.
The first claim of (d) is as \cite[thm.~3.10]{BFNa},
the second one follows using the associativity of the monoidal product and the fact that $\bf1$ is a monoidal unit.
Part (e) is similar.
\end{proof}

Since $G$ and $\widetilde G$ are the Levi factors of the pro-groups $G_\calO$ and $\widetilde G_\calO$, we have obvious isomorphisms
\begin{align}\label{RGG}R_{G_\calO}\cong R_{G}\cong R_{\widetilde G}|_\zeta\cong R_{\widetilde G_\calO}|_\zeta.\end{align}
Therefore, by Proposition \ref{prop:basic}(a), we can identify 
\begin{align}\label{formal}
\calA_\mu^\rho=K^{G_\calO}(\calR).\end{align}
Since the product on $\calA_{\mu,R}^\lambda$ is $R_{\dot G}$-linear
in the first variable, we have an $R_{\bbG_m\times T_W}$-algebra embedding 
\begin{align*}
R_{\dot G}\cong R_{\dot G}\otimes{\bf1}\subset\calA_{\mu,R}^\lambda.
\end{align*}
We abbreviate 
\begin{align}\label{CL}
\calA_{\mu,R}^{\lambda,0}=R_{\dot G}\otimes{\bf1}.
\end{align}
and
\begin{align*}
\calA_\mu^{\rho,0}=\calA_{\mu,R}^{\lambda,0}|_{\zeta,\rho}
\end{align*}
Hence $\calA_\mu^{\rho,0}$ is a commutative subalgebra of $\calA_\mu^\rho$ which is isomorphic to $R_G$.

\subsubsection{The algebra $\gr\calA_{\mu,R}^\lambda$}\label{sec:gr}

Let $\gr\calA_{\mu,R}^\lambda$ denote the associated graded algebra of the filtration 
of $\calA_{\mu,R}^\lambda$ in Proposition \ref{prop:basic}.
We have an isomorphism of $R_{\dot G}$-modules
\begin{align*}
\gr\calA_{\mu,R}^\lambda\cong\bigoplus_{\gamma\in\Lambda^+} \calA_{\gamma,R}^\lambda
\end{align*}
Let $\W_\gamma$ be the stabilizer  in the  Weyl group $\W$ of the cocharacter $\gamma$, and 
$P_\gamma$ the stabilizer in $G$ of the point $[\gamma]\in\Gr$.
Since $\calR_\gamma$ is a vector bundle over the $G_\calO$-orbit $\Gr_\gamma$ and
$\Gr_\gamma$ is a vector bundle over $G/P_\gamma$,
we have
\begin{align}\label{grA}
\gr\calA_{\mu,R}^\lambda\cong\bigoplus_{\gamma\in\Lambda^+}(R_{\dot T})^{\W_\gamma}\otimes [\scrO_{\calR_\gamma}]
\end{align}
where $[\scrO_{\calR_\gamma}]$ is the fundamental class of $\calR_\gamma$.
Here, the tensor product is a tensor product over $\bbC$. 
To avoid confusions we write 
$(R_{\dot T})^{\W_\gamma}\otimes [\scrO_{\calR_\gamma}]=(R_{\dot T})^{\W_\gamma}\diamond[\scrO_{\calR_\gamma}]$.
Hence, for each  $f\in (R_{\dot T})^{\W_\gamma}$ the element
\begin{align}\label{MO}f\diamond [\scrO_{\calR_\gamma}]\in\gr\calA_{\mu,R}^\lambda\end{align}
 is the image in the associated graded of an element in $K^{\dot G_\calO}(\calR_{\leqslant\gamma})$
whose restriction to $\calR_\gamma$ is identified with the image $\ind(f)$
of $f$ by the induction isomorphism 
$$\ind:(R_T)^{\W_\gamma}\to K^{\dot G_\calO}(G/P_\gamma)\cong K^{\dot G_\calO}(\calR_\gamma)$$ 
in \eqref{inductiona}.
If $\gamma$ is minuscule this element has an obvious lift in $\calA_{\mu,R}^\lambda$.
Let $f\diamond [\scrO_{\calR_\gamma}]$ denote also this lift.
We will say that two integers $n$ and $m$ have the same sign if $nm\geqslant 0$.
Note that for each $\gamma,\eta\in\Lambda^+$ we have $\W_{\gamma+\eta}=\W_\gamma\cap \W_\eta$,
see \cite[rem.~6.4]{BFNa}.
Let $D_\gamma=\prod_{i\in I}\prod_{r=1}^{a_i}(D_{i,r})^{\gamma_{i,r}}$ and 
$D_{i,r}:R_{\dot T}\to R_{\dot T}$ is the difference operator such that
$D_{i,r}\cw_{j,s}=\zeta_i^{2\delta_{i,j}\delta_{r,s}}\cw_{j,s}D_{i,r}$ as in \eqref{DW}.
The following is similar to results in \cite{BFNa}, \cite{BFNb}.

\begin{Lemma}\label{lem:multgrA}
Let $\gamma,\eta\in\Lambda^+$. 
\hfill
\begin{enumerate}[label=$\mathrm{(\alph*)}$,leftmargin=8mm,itemsep=1mm]
\item
There is an element $a_{\gamma,\eta}\in R_{\dot G}$ such that
the following formula holds in $\gr\calA_{\mu,R}^\lambda$
\begin{align}\label{multgrA}
(f\diamond [\scrO_{\calR_\gamma}]) \star (g\diamond [\scrO_{\calR_\eta}])
=a_{\gamma,\eta}fD_\gamma(g)\diamond [\scrO_{\calR_{\gamma+\eta}}]
,\quad
f\in (R_{\dot T})^{\W_\gamma}
,\quad
g\in (R_{\dot T})^{\W_\eta}
\end{align}
\item
If $W=0$, then $a_{\gamma,\eta}=1$ if
$d_j\gamma_{j,s}-d_i\gamma_{i,r}$ and $d_j\eta_{j,s}-d_i\eta_{i,r}$ have the same sign if $c_{ij}<0$.
\item
If $W\neq 0$, then $a_{\gamma,\eta}=1$ if
$\gamma_{i,r}$ and $\eta_{i,r}$ have the same sign for each $(i,r)$, and
$d_j\gamma_{j,s}-d_i\gamma_{i,r}$ and $d_j\eta_{j,s}-d_i\eta_{i,r}$ have the same sign if $c_{i,j}<0$.
\end{enumerate}
\end{Lemma}

\begin{proof}
The proof of the proposition is similar to the proof of \cite[prop.~6.2]{BFNa}.
We consider the $\infty$-stack $\frakR_T$ and its classical truncation $\calR_T$  by setting $H=T$ in Remark \ref{rem:H1}.
Note that there is an obvious inclusion 
$\Lambda=(\Gr_T)_\red\subset\Gr_G$ such that
$(\calR_T)_\red\cong  \Lambda\times_{(\Gr_G)_\red}(\calR_G)_\red$.
Set 
$$
\calA_T=K^{\dot T_\calO}(\calR_T)
,\quad
\calA_G=K^{\dot G_\calO}(\calR_G)
$$
We equip $\calA_T$ with the product in \S\ref{sec:H2}.
For each $\gamma\in\Lambda$ let $[\scrO_{(\calR_T)_\gamma}]$ be the class in $\calA_T$ of the structural sheaf of the 
fiber $(\calR_T)_\gamma$ of the map $\calR_T\to\Gr_T$ at $\gamma$.
Recall that if  $\gamma\in\Lambda^+$ then $[\scrO_{\calR_\gamma}]$ 
is the class of $\scrO_{\calR_\gamma}$ in $\calA_G$,
see \eqref{Rgamma}.
We have
\begin{align}\label{multgrAT}
(f\diamond [\scrO_{(\calR_T)_\gamma}]) \star (g\diamond [\scrO_{(\calR_T)_\eta}])=
a_{\gamma,\eta}fD_\gamma(g)\diamond [\scrO_{(\calR_T)_{\gamma+\eta}}]
,\quad
f,g\in R_{\dot T}
,\quad
\gamma,\eta\in\Lambda
\end{align}
for some element $a_{\gamma,\eta}\in R_{\dot T}$.
The definition of the monoidal product in \eqref{*} yields
\begin{align}\label{a-gamma-eta2}
a_{\gamma,\eta}=
\Wedge\Big(\frac{z^\gamma  N_\calO}{ N_\calO\cap z^\gamma  N_\calO}\Big)\cdot
\Wedge\Big(\frac{ N_\calO\cap z^{\gamma+\eta} N_\calO}{ N_\calO\cap z^\gamma  N_\calO\cap z^{\gamma+\eta}  N_\calO}\Big)\cdot
\Wedge\Big(\frac{z^\gamma  N_\calO\cap z^{\gamma+\eta} N_\calO}{ N_\calO\cap z^\gamma  N_\calO\cap z^{\gamma+\eta}  N_\calO}\Big)^{-1},
\end{align}
and the quotients in \eqref{a-gamma-eta2} are all finite dimensional representations of the tous $T$.
Now, the pushforward by the map $\iota$ in \eqref{iota} with $H=T$ yields an injective map
\begin{align*}\iota_*:\calA_T\to R_T\otimes_{R_G}\calA_G.
\end{align*}
We equip $\calA_T$ with the filtration induced by the filtration of $\calA_G$ in Proposition \ref{prop:basic}.
Taking the associated graded, we get a generically invertible $R_T$-linear embedding
$$\gr\iota_*:\gr\calA_T\to R_T\otimes_{R_G}\gr\calA_G$$
such that
\begin{align}\label{multgrAT2}
(\gr\iota_*)^{-1}(f\diamond [\scrO_{\calR_\gamma}])=\sum_{w\in \W/\W_\gamma}
\frac{wf}{\Wedge(T_{w\gamma}\Gr_\gamma)}\diamond [\scrO_{(\calR_T)_{w\gamma}}]
,\quad
f\in (R_{\dot T})^{\W_\gamma}
\end{align}
Here $\Wedge(V)$ is the equivariant Euler class given by
\begin{align}\label{Wedge}\Wedge(V)=\sum_{i\geqslant 0}(-1)^i\Wedge^i(V^\vee).\end{align}
The formula \eqref{multgrA} follows as in \cite[prop.~6.2]{BFNa}
from \eqref{multgrAT}, \eqref{multgrAT2} and \eqref{iota*},
 which implies that the map $\gr \iota_*$ is an algebra homomorphism.
 
Finally, we compute the constant $a_{\gamma,\eta}$.
The vector space $z^\gamma N_\calO$ is a representation of the torus $\dot T$.
It decomposes as an infinite sum of characters, each of them being a triple of characters of $T$, $\bbG_m$ and $T_W$.
We consider separately the contributions of the summands 
$ N_\calO=N^{(1)}\oplus N^{(2)}$ in
\eqref{NO1} and \eqref{NO2} given by
\begin{align*}
N^{(1)}=\bigoplus_{\substack{i\to j\\c_{ij}=-1}}\Hom(V_i,V_j)\llb z_j\rrb\oplus
\bigoplus_{\substack{i\to j\\c_{ij}<-1}}\Hom(V_i,V_j)\llb z_i\rrb z_i^{1+c_{ij}},\quad
N^{(2)}=\bigoplus_{i\in I}\Hom(W_i,V_i)\llb z_i\rrb.
\end{align*}
Using \eqref{GNOK}, \eqref{sigma}, \eqref{loop2} we deduce that $z^\gamma N_\calO$
is the sum of the following characters 
\begin{equation}\label{weightsN}
\begin{minipage}{0.9\linewidth}
\begin{enumerate}[label=$\mathrm{(\alph*)}$,leftmargin=4mm,itemsep=2.5mm]
\item[$\bullet$]
$\Big(\cw_{j,s}-\cw_{i,r}\,,\,d_j(\gamma_{j,s}+n+1)-d_i(\gamma_{i,r}+\frac12)\,,\,0\Big)$ for all 
$i\to j\,,\,c_{ij}=-1,\,r\in[1,a_i], s\in[1,a_j], n\in\bbN,$
\item[$\bullet$]
$\Big(\cw_{j,s}-\cw_{i,r}\,,\,d_j(\gamma_{j,s}+\frac12)-d_i(\gamma_{i,r}-n)\,,\,0\Big)$ for all
$i\to j\,,\,c_{ij}<-1,\,r\in[1,a_i], s\in[1,a_j], n\in\bbN$,
\item[$\bullet$]
$\Big(\cw_{i,r}\,,\,d_i(\gamma_{i,r}+n)\,,\,-\cz_{i,t}\Big)$
for all $i\in I$, $r\in[1,a_i],$  $ t\in [1,l_i]$, $n\in\bbN$.
\end{enumerate}
\end{minipage}
\end{equation}
A computation using \eqref{a-gamma-eta2} and \eqref{weightsN} yields
\begin{align}\label{a-gamma-eta2}
a_{\gamma,\eta}=
\prod_A(1-\zeta_j^{-c_{ji}-2-2m}\cbfw_{j,s}^{-1}\cbfw_{i,r})\cdot
\prod_B(1-\zeta_i^{c_{ij}-2m}\cbfw_{j,s}^{-1}\cbfw_{i,r})\cdot
\prod_C(1-\zeta_i^{-2m}\cbfw_{i,r}^{-1}\cbfz_{i,t})
\end{align}
where the products run over the elements
$i,j\in I$, $r\in[1,a_i],$, $s\in[1,a_j]$, $t\in [1,l_i]$ and $m,m',m''\in\bbZ$ such that $m,m''\geqslant 0>m'$ or $m'\geqslant 0>m,m''$ and
\begin{align*}
A&=\{d_j\gamma_{j,s}-d_i\gamma_{i,r}=d_j(m-m')\,,\,
d_j\eta_{j,s}-d_i\eta_{i,r}=d_j(m'-m'')\,,\,i\to j\,,\,c_{ij}=-1\},\\
B&=\{d_j\gamma_{j,s}-d_i\gamma_{i,r}=d_i(m-m')\,,\,
d_j\eta_{j,s}-d_i\eta_{i,r}=d_i(m'-m'')\,,\,i\to j\,,\,c_{ij}<-1\},\\
C&=\{\gamma_{i,r}=m-m'\,,\,\eta_{i,r}=m'-m''\}.
\end{align*}
The lemma is a direct consequence of the formula \eqref{a-gamma-eta2}.

\end{proof}

\section{Shifted truncated affine quantum groups and QH algebras}

This section is a reminder on shifted affine quantum groups and truncated ones.
We follow  \cite{FT19} and \cite{H23}.

\subsection{Shifted affine quantum groups}\label{sec:SQG}

\subsubsection{Definition of the shifted affine quantum groups}\label{sec:SQG1}
For each $m\in\bbN$ and each $i,j\in I$ we define
\begin{align}\label{bmu}
\begin{split}
[m]_\bfzeta&=(\bfzeta^m-\bfzeta^{-m})/(\bfzeta-\bfzeta^{-1})
,\quad
[m]_\bfzeta!=[m]_\bfzeta[m-1]_\bfzeta\cdots[1]_\bfzeta,\\
\bfzeta_i&=\bfzeta^{d_i},\quad
\delta(u)=\sum_{n\in\bbZ}u^n
,\quad
g_{ij}(u)=\frac{\bfzeta_i^{c_{ij}}u-1}{u-\bfzeta_i^{c_{ij}}}\\
\bmu^\pm&=\sum_{i\in I}m_i^\pm\bomega_i
,\quad
\bmu=\bmu^++\bmu^-
,\quad
m_i=m_i^++m_i^-
,\quad
m_i^\pm\in\bbZ.
\end{split}
\end{align}
Consider the formal series 
\begin{align*}
x^\pm_i(u)=\sum_{n\in\bbZ}x^\pm_{i,n}\,u^{-n}
 ,\quad
\psi^\pm_i(u)=\sum_{n\geqslant-m_i^\pm}\psi^\pm_{i,\pm n}\,u^{\mp n}.
\end{align*}
Let $F$ be the fraction field the ring $R$  in \eqref{zeta}.
Let $\bfU_{\bmu^+,\bmu^-,F}$ be the $(\bmu^+,\bmu^-)$-shifted adjoint quantum loop group over $F$ 
with quantum parameter $\bfzeta$ of type $\bfc$.
It is the $F$-algebra generated by the elements
$$x^\pm_{i,m}
 ,\quad
\psi^\pm_{i,\pm n}
,\quad
(\phi^\pm_{i})^{\pm1}
,\quad
i\in I
 ,\quad
m,n\in\bbZ
 ,\quad
n\geqslant-m^\pm_i$$
with the following defining relations  
where  $i,j\in I$
\hfill
\begin{enumerate}[label=$\mathrm{(\alph*)}$,leftmargin=8mm,itemsep=2.5mm]
\item $\phi^\pm_i$ is invertible with inverse $(\phi^\pm_i)^{-1}$ and
$\psi_{i,\mp m_i^\pm}^\pm=\prod_j(\phi_j^\pm)^{\pm c_{ji}}$,
\item
$\phi^\pm_j\psi_i^+(u)=\psi_i^+(u)\phi^\pm_j$ 
and
$\phi^\pm_j\psi_i^-(u)=\psi_i^-(u)\phi^\pm_j$, 
\item
$\phi^\pm_j\,x^+_i(u)=\bfzeta_j^{\pm \delta_{ij}}x^+_i(u)\,\phi_j^\pm$
and
$\phi^\pm_j\,x^-_i(u)=\bfzeta_j^{\mp\delta_{ij}}x^-_i(u)\,\phi_j^\pm$,
\item 
$\psi^+_i(u)\,\psi^\pm_j(v)=\psi^\pm_j(v)\,\psi^+_i(u)$
and
$\psi^-_i(u)\,\psi^\pm_j(v)=\psi^\pm_j(v)\,\psi^-_i(u)$,
\item 
$\psi^+_i(u)\,x^\pm_j(v)=x^\pm_j(v)\,\psi^+_i(u)\,g_{ij}(u/v)^{\pm 1}$
and
$\psi^-_i(u)\,x^\pm_j(v)=x^\pm_j(v)\,\psi^-_i(u)\,g_{ij}(u/v)^{\pm 1}$,
\item $x^\pm_i(u)\,x^\pm_j(v)=x^\pm_j(v)\,x^\pm_i(u)\,g_{ij}(u/v)^{\pm 1}$,
\item $(\bfzeta_i-\bfzeta_i^{-1})[x^+_i(u)\,,\,x^-_j(v)]=\delta_{ij}\,\delta(u/v)\,(\psi^+_i(u)-\psi^-_j(u\rrp$,
\item the quantum Serre relations 
$$\sum_{\sigma\in \frakS_s}\sum_{r=0}^{s}(-1)^r\Big[\begin{matrix}s\\r\end{matrix}\Big]_{\bfzeta_i}
x^\pm_i(u_{\sigma(1)})\cdots x^\pm_i(u_{\sigma(r)})x^\pm_j(v)x^\pm_i(u_{\sigma(r+1)})\cdots
x^\pm_i(u_{\sigma(s)})=0
,\quad
s=1-c_{ij}
,\quad
i\neq j$$
\end{enumerate}
\setcounter{equation}{7}
The rational function $g_{ij}(u)$ is expanded as a power series of $u$.
Let $\bfU_{\bmu^+,\bmu^-,F}^\pm$ 
be the subalgebra generated by the $x_{i,n}^\pm$'s
and $\bfU_{\bmu^+,\bmu^-,F}^0$ the subalgebra generated by 
the $\psi_{i,\pm n}^\pm$'s and the $(\phi^\pm_{i})^{\pm1}$'s.
We define $(x^\pm_{i,n})^{[m]}$ and $h_{i,\pm m}$ such that
\begin{align*}
(x^\pm_{i,n})^{[m]}&=\frac{1}{[m]_{\bfzeta_i}!}(x^\pm_{i,n})^{m},\\[1mm]
\psi_i^\pm(u)&= u^{\pm m^\pm_i}\psi_{i,\mp m^\pm_i}^\pm
\exp\Big(\pm(\bfzeta_i-\bfzeta_i^{-1})\sum_{m>0}h_{i,\pm m}u^{\mp m}\Big).
\end{align*}
We define $\bfU_{\bmu^+,\bmu^-,R}$ 
to be the $R$-subalgebra of $\bfU_{\bmu^+,\bmu^-,F}$ generated by 
\begin{align}\label{SLGR}
(\phi^\pm_{i})^{\pm 1}
 ,\quad
h_{i,\pm m}/[m]_{\bfzeta_i}
 ,\quad
(x^\pm_{i,n})^{[m]}
,\quad
i\in I
,\quad
n\in\bbZ
,\quad
m\in\bbN^\times.
\end{align}
Specializing $\bfzeta$ to $\zeta$ as in \eqref{bcrho}, we define
$$\bfU_{\bmu^+,\bmu^-}=\bfU_{\bmu^+,\bmu^-,R}|_\zeta.$$
Up to isomorphisms, the algebra $\bfU_{\bmu^+,\bmu^-}$ only depends on the cocharacter $\bmu$.
From now on, we will assume that $\bmu^+=0$.
We define 
$$\bfU_\bmu=\bfU_{0,\bmu}
,\quad
\bfU_{\bmu,F}=\bfU_{0,\bmu,F}
,\quad
\bfU_{\bmu,R}=\bfU_{0,\bmu,R}.$$
We define
$\bfU_{\bmu}^+$, $\bfU_{\bmu}^0$ and $\bfU_{\bmu}^-$ similarly.
We have a triangular decomposition
$$\bfU_{\bmu}=\bfU_{\bmu}^+\otimes\bfU_{\bmu}^0\otimes\bfU_{\bmu}^-.$$
We abbreviate 
$$
\zeta_i=\zeta^{d_i}
,\quad
\bfE_{i,n}=(\zeta_i-\zeta_i^{-1})\,x^+_{i,n}
,\quad
\bfF_{i,n}=(\zeta_i-\zeta_i^{-1})\,x^-_{i,n}
.$$

\begin{Remark}\label{rem:QLG}
\hfill
\begin{enumerate}[label=$\mathrm{(\alph*)}$,leftmargin=8mm,itemsep=2.5mm]
\item
The element $\phi_i^+\in\bfU_\bmu$ is an analogue of $\zeta_i^{\bomega_i}$ where $\bomega_i$ is the $i$th fundamental coweight.
\item 
The element $\phi^+_{i}\,\phi^-_{i}$ in $\bfU_\bmu$ is central for each $i\in I$.
The (non shifted) quantum loop group $\bfU$ is the quotient
of  $\bfU_0$ by the relations 
\begin{align}\label{lastrelation}
\phi^+_{i}\,\phi^-_{i}=1,\quad i\in I.
\end{align}
\end{enumerate}
\end{Remark}

\subsubsection{Shifted affine quantum groups and the integral category $\scrO$}\label{sec:O}
\begin{Definition}\label{def:O1}
We define $\scrO_\bmu$ to be the category of all finitely generated $\bfU_\bmu$-modules $M$ such that
\begin{enumerate}[label=$\mathrm{(\alph*)}$,leftmargin=8mm,itemsep=1mm]
\item the central element $\phi^+_{i}\,\phi^-_{i}-(-\zeta_i)^{-a_i}$ acts by 0 for each $i\in I$ for some integers $a_i\in\bbZ$,
\item the action of $\phi_i^+$ on $M$ is semisimple with eigenvalues in $(\zeta_i)^\bbZ$ for each $i\in I$,
\item the eigenspaces of $\{\phi_i^+\,;\,i\in I\}$ in $M$ are finite dimensional, 
\item the weights of $\{\phi_i^+\,;\,i\in I\}$ are bounded above, i.e., for each $i\in I$  the spectrum of $\phi_i^+$ belongs to the set
$\{\zeta_i^k,;\,k\in\bbZ_{\leqslant n_i}\}$ for some integer $n_i$.

\end{enumerate}
\end{Definition}

Let $\Psi$ be an $I$-tuple of rational functions over $\bbC$ 
such that $\Psi_i(u)$ is regular and non zero at $u=\infty$ for all $i\in I$.
Assume that the expansions of the rational function $\Psi_i(u)$
in powers of $u$ are of the following form
\begin{align*}
\Psi^+_i(u)=\sum_{n\in\bbN}\Psi_{i,n}^+\,u^{-n}
 ,\quad
\Psi^-_i(u)=\sum_{n\geqslant-m_i}\Psi_{i,-n}^-\,u^{n}
,\quad
\Psi_{i,0}^+,\,\Psi_{i,m_i}^-\neq 0
\end{align*}
where the integers $m_i$ are as in \eqref{bmu}.
Then, we will say that $\Psi$ is an $\ell$-weight of degree $\bmu$. 
Let $Z_{i,r}$ be the $\ell$-weight given by
$$(Z_{i,r})_j(u)=(1-\zeta^ru^{-1})\ \text{if}\ j=i,\ \text{and}\ (Z_{i,r})_j(u)=1\ \text{else}.$$
For any $\bfU_\bmu$-module $M$ 
the generalized $\ell$-weight space of $M$ of $\ell$-weight $\Psi$ is
\begin{align}\label{MPsi}
M_\Psi=\{v\in M\,;\,(\psi^+_{i,n}-\Psi^+_{i,n})^\infty\cdot v=0\,,\,i\in I\,,\,n\in\bbN\}.
\end{align}
By \cite[\S 4]{H23} any object in $\scrO_\bmu$ is the direct sum of its generalized $\ell$-weight spaces 
and the latter are finite dimensional. Further, 
the generalized $\ell$-weight space of $\ell$-weight $\Psi$ 
is zero unless $\Psi$ is of degree $\bmu$.
Finally, the simple objects in $\scrO_\bmu$ are labelled by their highest
$\ell$-weights, which run over the set of all $\ell$-weights of degree $\bmu$.
More precisely, the module $L(\Psi)$ is generated by a vector $v$ such that
$$x_{i,n}^+\cdot v=0
 ,\quad
\psi^+_{i,n}\cdot v=\Psi^+_{i,n}v
 ,\quad
i\in I
,\quad
n\in\bbZ.$$
The $q$-character of a module $M\in\scrO_\bmu$ is the formal series
$$q\text{-}\!\ch(M)=\sum_\Psi\dim(M_\Psi)\,\Psi$$
Let ${}^0\scrO_\bmu\subset\scrO_\bmu$ be the full subcategory of modules $M$ such that
$M_\Psi=0$ unless
\begin{align}\label{even2}\Psi=\prod_{(i,r)\in{}^0\!I}(Z_{i,r})^{v_{i,r}},\quad v=(v_{i,r})\in\bbZ {}^0\!I\end{align}
where the product is the componentwise multiplication of $I$-tuples.
We call ${}^0\scrO_\bmu$ the integral category $\scrO$ of $\bfU_\bmu$, see \cite[Def.~9.13]{GHL24} for details.

\begin{Remark}\label{rem:catO}\hfill
\begin{enumerate}[label=$\mathrm{(\alph*)}$,leftmargin=8mm,itemsep=1mm]
\item 
Definition \ref{def:O1} differs from the definition of $\scrO_\bmu$ in \cite{H23}.
Condition (a) is not mentioned there. 
We imposed it in order to fit with Coulomb branches via the morphism $\Phi$ in \eqref{Phi}.
See \eqref{A0} and \eqref{formula} for more details.
Condition (d) is written differently in loc.~cit.
By \eqref{equiv}, the category $\scrO_\bmu$ is unchanged if we replace (c) by the condition (c') that
the $\ell$-weight spaces of $M$ are finite dimensional.

\item
For any tuple $(i,r)\in I^\bullet$ let $Y_{i,r}$ be the $\ell$-weight 
whose $j$th component is $\zeta_i(1-\zeta^{r-d_i} u^{-1})\,(1-\zeta^{r+d_i}u^{-1})^{-1}$ if $j=i$ and 1 else.
For any $v\in\bbZ I^\bullet$, we consider the $\ell$-weight $Y^v$ of degree 0 given by 
\begin{align}\label{Y}
Y^v=\prod_{(i,r)\in I^\bullet}(Y_{i,r})^{v_{i,r}}
\end{align}
Let $\scrO\subset\scrO_0$ be the subcategory of all $\bfU_0$-modules which descend to $\bfU$-modules.
A simple module $L(\Psi)\in\scrO$ is finite dimensional if and only if 
$\Psi=Y^v$ for some $v\in\bbN I^\bullet$.
\end{enumerate}
\end{Remark}

\subsection{Shifted truncated affine quantum groups}\label{sec:STQG}

Fix a cocharacter $\bmu=\sum_{i\in I}m_i\,\bomega_i$ as above (the shift).
Choose a dominant cocharacter $\blambda=\sum_{i\in I}l_i\,\bomega_i$ and 
an element $\balpha=\sum_{i\in I}a_i\,\balpha_i$ of the coroot semigroup such that
\begin{align}\label{blam}
\blambda=\bmu+\balpha.
\end{align}
Note that
$m_i=l_i-\sum_{j}a_jc_{ji}$ for each $i\in I$.
Next, we define $\alpha$, $\lambda$, $\mu$ by
\begin{align}\label{lam2}
\alpha=\sum_{i\in I}a_i\delta_i
,\quad
\lambda=\sum_{i\in I}l_i\delta_i
,\quad
\mu=\lambda-C\cdot\alpha.
\end{align}
From now on, we fix $\zeta\in\bbC^\times$ which  is not a root of unity,
and we fix the even coweight $\rho$ in ${}^0\!\Lambda_\lambda$.
We fix the tuple $\dlambda$ as in  \eqref{VW'} and \eqref{dlam2}.

\subsubsection{Definition of the map $\Phi$}
We want to define an $R_{\bbG_m\times T_W}$-algebra homomorphism
\begin{align}\label{Phi}
\Phi:\bfU_{\bmu,R}\otimes R_{T_W}\to\calA_{\mu,R}^\lambda
\end{align}
To do this, we identify in the obvious way 
\begin{align}\label{RT1}
R_{T_W}=\bbC[\,\cbfz_{i,s}^{\pm 1}\,;\,i\in I\,,\,s\in [1,l_i\rrb
\end{align}
Set $M_i^+=0$ and $M_i^-=l_i-m_i$.
There are formal series 
$$\bfA^\pm_i(u)=\sum_{n\geqslant 0}\bfA^\pm_{i,\pm n}u^{\mp n}\in\bfU_{\bmu,R}^0\otimes R_{T_W}\llb u^{\mp 1}\rrb$$  such that
\begin{align}\label{AZ}
\begin{split}
\psi^\pm_i(u)&=u^{M_i^\pm}\cdot
\bfZ_i(u)\cdot \bfA^\pm_i(u)^{-1}\cdot \bfA^\pm_i(\zeta_i^{-2}u)^{-1}\cdot
\prod_{c_{ji}<0}\prod_{n=1}^{-c_{ji}}\bfA^\pm_j(\zeta_j^{-c_{ji}-2n}u),\\
\bfZ_i(u)&=\prod_{s=1}^{l_i}\big(1-\zeta_i^2\cbfz_{i,s}u^{-1}\big).
\end{split}
\end{align}
See \cite[\S 7]{FT19} for details. 
In particular, we have
\begin{align}\label{A0}
\bfA_{i,0}^\pm=(\phi_i^\pm)^{-1}.
\end{align}
For any dimension vector $\nu = \sum_i r_i\delta_i$ with $0\leqslant r_i \leqslant a_i$, let
\begin{align*}
\scrQ_{i,\nu}=(V_{i,\calO}/L)\otimes\scrO_{\calR_{\omega_\nu}}
,\quad
\scrS_{i,\nu}=(L/V_{i,\calO})\otimes\scrO_{\calR_{\omega_{-\nu}}}
\end{align*}
denote the tautological rank $r_i$ bundles over  the space $\calR_{\omega_\nu}$ and  $\calR_{\omega_{-\nu}}$ in
\eqref{RGr}.
We abbreviate 
$$\scrQ_i=\scrQ_{i,\delta_i},\quad\scrS_i=\scrS_{i,\delta_i}.$$
The algebra $R_\tbG$ is generated by $\bfzeta^{\pm 1}$, the square root of the determinant
$\det(V_i)^{\pm \frac12}$ and the exterior powers $\Lambda_{i,r}=\Wedge^r(V_i)$ and $\Lambda_{i,-r}=\Wedge^r(V_i^\vee)$
where $\Lambda_{i,\pm r}=0$ if $r>a_i$.
We abbreviate 
\begin{align}\label{L}D_i=\det(V_i)^{\frac12}
,\quad
D_i^+=\bigotimes_{i\to j}D_j
,\quad
D_i^-=\bigotimes_{j\to i}D_j.
\end{align}
We consider the elements of $\calA_{\mu,R}^\lambda$ given by
\begin{align}\label{formula}
\begin{split}
E_{i,r}&=(-1)^{a_i}(D_i^-)^{-1}\otimes(\bfzeta_i^2[\scrS_i])^{\otimes (r+a_i)},\\
F_{i,r}&=(-1)^{a_i^+}(\bfzeta_i)^{a_i^--1}D_i^+\otimes
(\bfzeta_i^2[\scrQ_i])^{\otimes (r-a_i^+)},\\
A_{i,r}^+&=(-1)^rD_i^{-1}\otimes\Lambda_{i,r}\otimes {\bf1},\\
A_{i,-r}^-&=(-1)^r(-\bfzeta_i)^{a_i}D_i\otimes\Lambda_{i,-r}\otimes {\bf1}
\end{split}
\end{align}
Here the tensor product is the $R_{\tbG}$-action on $\calA_{\mu,R}^\lambda$ and we have set
\begin{align}\label{apm}
a_i^+=-\sum_{i\to j}a_jc_{ji},\quad a_i^-=-\sum_{i\to j}a_jc_{ij}.
\end{align}
The following is similar to \cite[thm.~8.1]{FT19}, \cite[thm.~5.6]{NW23}. A proof is given in the appendix.

\begin{Proposition} \label{prop:Phi}
There is a unique $R_{\bGm\times T_W}$-algebra homomorphism 
$$\Phi:\bfU_{\bmu,R}\otimes R_{T_W}\to\calA_{\mu,R}^\lambda$$ 
such that
$\bfE_{i,r}\mapsto E_{i,r}$, $\bfF_{i,r}\mapsto F_{i,r}$ and $\bfA_{i,r}^\pm\mapsto A_{i,r}^\pm$.
\qed
\end{Proposition}

\subsubsection{Surjectivity of the map $\Phi$}\label{sec:Phisurj}

Composing the map $\Phi$ and the specializations at $\rho$ and $\zeta$ we get an algebra homomorphism
\begin{align}\label{Phi1}
\Phi:\bfU_\bmu\to \calA_\mu^\rho
\end{align} 
which takes the commutative subalgebra $\bfU_\bmu^0$ of $\bfU_\bmu$ onto
the commutative subalgebra $\calA_\mu^{\rho,0}$ of $\calA_\mu^\rho$ considered in \eqref{CL}.
The goal of this section is to prove the following.

\begin{Proposition}\label{prop:onto}
The map $\Phi:\bfU_\bmu\to \calA_\mu^\rho$ is surjective.
\end{Proposition}


To prove the proposition we argue as in \cite{We19}. The proof consists of several steps.
Let $B^-\subset G$ be the Borel subgroup of lower triangular matrices.
Set $X=G/B^-$. 
The affine nil-Hecke algebra is the complexified Grothendieck group $\NH=K^T(X)$ 
equipped with the Kostant-Kumar convolution product. 
For each $w\in\W$, let $X_w=B^-wB^-/B^-$,  $X_{\leqslant w}=\bigcup_{v\leqslant w}X_v$ and
$D_w=[\scrO_{X_{\leqslant w}}]$ in $\NH$.
Let $w_0$ be the largest element in $\W$.

Given a cocharacter $\gamma$ of the torus $T$, let $\W_\gamma$ be the stabilizer of $\gamma$ in the Weyl group $\W$.
Recall that $\Gr_\gamma$ is the $G_\calO$-orbit of the point $[\gamma]\in\Gr$.
Let $\calI\subset G_\calO$ be the Iwahori subgroup containing $B^-$.
Let $\calX=G_K/\calI$ be the affine flag manifold and
$\widehat\W=\W\ltimes\Lambda$ the affine Weyl group. 
We define $\calR^\iw=\calX\times_\Gr\calR$
where the fiber product is
relative to the projections $p:\calR\to\Gr$ and $p:\calX\to \Gr$ as in \eqref{p1}.
For each $w\in\widehat\W$ and each cocharacter $\gamma\in\Lambda$ we set 
${}^I\!\calX_w=\calI w\calI/\calI$ and $\calX_\gamma=\pi^{-1}(\Gr_\gamma)$.
Then, we define ${}^I\!\calX_{\leqslant w}$ and $\calX_{\leqslant\gamma}$ in the obvious way.
Finally, we define $\calR_{w}^\iw$, $\calR_{\gamma}^\iw$, 
${}^I\calR_{\leqslant w}^\iw$ and $\calR_{\leqslant\gamma}^\iw$  by base change from $\calX$.
Let $\calR^\ab=\calR_T$ be as in \S\ref{sec:H2}.
To simplify the notation, in this proof we abbreviate $\calA=\calA_\mu^\lambda$ and 
we set
$$\calA^\ab=K^{\dot T_\calO}(\calR^\ab)|_{\zeta}
,\quad
\calA^\iw=K^{\dot \calI}(\calR^\iw)|_{\zeta}
,\quad
\calP=K^{\dot\calI}(\calR)|_{\zeta}
.$$
Set $R_w^\iw=[\scrO_{{}^I\calR_{\leqslant w}^\iw}]$ in $\calA^\iw$.
We equip $\calA^\ab$ and $\calA^\iw$ with the associative products as in \S\ref{sec:H2}.
The obvious projection $\pi:\calR^\iw\to\calR$ is flat and proper.
Hence the pullback and pushforward by $\pi$ are well-defined in K-theory.
Let $\phi:\calA\to\calP$ be the forgetting of the equivariance.
See  \eqref{RH} and \eqref{zetarho} for more details on the notation.
For each character $\gamma$ and each $f\in (R_T)^{\W_\gamma}$ we define
the element $f\diamond[\scrO_{\calR_{\gamma}}]$ in $\gr\calA$ as in \eqref{MO}.

\begin{Lemma}\label{lem:matrix-algebra}\hfill
\begin{enumerate}[label=$\mathrm{(\alph*)}$,leftmargin=8mm,itemsep=1mm]
\item
There are algebra embeddings $i:\calA^\ab\to\calA$, $j:\calA\to\calA^\iw$ and  $k:\NH\to\calA^\iw$.
\item
There is an idempotent $\bfe\in\NH$ such that 
$\calA^\iw=\calA^\iw\star\bfe\star\calA^\iw$,
$j(\calA)=\bfe\star\calA^\iw\star\bfe$ and $j(1)=\bfe$.
\item
If $a\in\calA$ and $a^\iw\in\calA^\iw$ are such that $\phi(a)=\pi_*(a^\iw)$, then 
$j(a)=a^\iw\star\bfe$.
Further $\phi=\pi_*\circ j$.
\item
We have $\calP\cong\calA^\iw\star\bfe$ as an $(\calA^\iw,\calA)$-bimodule.
Its is a free $\calA$-module of rank $\sharp\W$ and
$\calA^\iw\cong\End_{\calA}(\calP)$. 
\item
We have $k(D_w)=R_w^\iw$ for each $w\in\W$.
\end{enumerate}
\end{Lemma}

\begin{proof}
As in \cite[\S2.5,2.6]{We19}, to which we refer for details.
Note that ${}^I\!\calX_{\leqslant w_0}=\calX_0\cong X$, that ${}^I\calR_{\leqslant w_0}^\iw=\calR_0^\iw$ 
is a trivial vector bundle over $X$. 
Here $0$ is the unit (=zero) in $\Lambda$.
Taking the K-theory we get the subalgebra $\calA_0^\iw$ of $\calA^\iw$. 
Let $z$ and $p$ be the zero section and the projection.
The map $z^*$ is an algebra homomorphism $\calA_0^\iw\to\NH$, see \eqref{z*}.
We have 
$$z^*(R_w^\iw)=z^*p^*(D_w)=D_w,\quad w\in \W.$$
Hence $k=p^*=(z^*)^{-1}$ is an algebra isomorphism $\NH\to\calA_0^\iw$.
The embedding $i$ is the pushforward by the immersion $\iota$ as in \eqref{iota*}.
The embedding $j\circ i$ is the pushforward by the 
closed immersion $\iota^\iw:\calR^\ab\to\calR^\iw$.
Note that the restriction of the map $\pi$ to $\calR^\ab$ is an isomorphism onto $\iota(\calR^\ab)$.
The idempotent $\bfe$ is the full idempotent in $\bbC\W\subset\NH$. It is given by
$$\frac{1}{\sharp\W}D_{w_0}\star\Wedge(T^*X)$$
where $\Wedge$ is as in \eqref{Wedge} and $\star$ is the multiplication in $\NH$.
\end{proof}

For each $r=1,\dots,a_i$ let $w_{i,r}$ be the basic cocharacter of the torus $T$ as in \eqref{fundamental1}.
Set  $w_{i,-r}=-w_0 (w_{i,r})$.
For each cocharacter $\gamma$ of the torus $T$, let $\bfr_\gamma^\ab=[\scrO_{\calR_\gamma^\ab}]$ in $\calA^\ab$.
We abbreviate $\bfr_{i,\pm r}^\ab=\bfr_{w_{i,\pm r}}^\ab$.

\begin{Lemma}\label{lem:Aab}
The algebra $\calA^\ab$ is generated by $R_T$ and the subset 
$$\Sigma^\ab=\{f\diamond \bfr_{i,\pm r}^\ab\,;\,i\in I\,,\,1\leqslant r\leqslant a_i\,,\,f\in R_T\}$$
\end{Lemma}

\begin{proof}
The proof is similar to the proof of \cite[prop.~3.5]{We19}.
By induction we are reduced to prove that, given a cocharacter $\gamma$ and $r=1,\dots,a_i$, the element
$\bfr_\gamma^\ab$ is an $R_{T\times T_W}$-linear combination of the products 
$\bfr_{\pm w_{i, r}}^\ab\star\bfr_{\gamma\mp w_{i, r}}^\ab$.
We may assume that $\gamma\neq 0$. 
Assume that there exists a pair $(i,r)$ such that $\gamma_{ir}> 0.$
Recall that
$\bfr_{w_{i,r}}^\ab\star\bfr_{\gamma-w_{i,r}}^\ab=a_{w_{ir},\gamma-w_{ir}}\,\bfr_\gamma^\ab$
where the coefficient $a_{w_{ir},\gamma-w_{ir}}\in R_{T\times T_W}$ is as in \eqref{a-gamma-eta2}.
If $\gamma_{ir}\leqslant 0$ for all pair $(i,r)$ then we argue in a similar way using instead the relation
$$\bfr_{-w_{i,r}}^\ab\star\bfr_{\gamma+w_{i,r}}^\ab=a_{-w_{ir},\gamma+w_{ir}}\,\bfr_\gamma^\ab$$
A computation using  \eqref{a-gamma-eta2} implies that,
up to the multiplication with invertible elements in $R_T$,
we have $a_{w_{ir},\gamma-w_{ir}}=A_{i,r}B_{i,r}$ with
\begin{align}\label{AB}
\begin{split}
A_{i,r}&=\prod_{j,s\in A}(1-\zeta_i^{2+c_{ij}}\cbfw_{i,r}\cbfw_{j,s}^{-1})\cdot
\prod_{j,s,m\in B}(1-\zeta_j^{-2m-c_{ji}}\cbfw_{i,r}\cbfw_{j,s}^{-1})\\
B_{i,r}&=\begin{cases}
\prod_{t=1}^{l_i}(1-\cbfw_{i,r}^{-1}\cbfz_{i,t})&\ \text{if}\ \gamma_{ir}\leqslant 0,\\
1&\ \text{else}
\end{cases}
\end{split}
\end{align}
where the products run over the following sets
\begin{align*}
A=&\{j\to i\,,\,c_{ji}=-1\,,\,\gamma_{ir}+c_{ij}\gamma_{js} \leqslant 0\}\cup
\{i\to j\,,\,c_{ij}<-1\,,\,\gamma_{ir}+c_{ij}\gamma_{js}\leqslant 0\},\\
B=&\{j\to i\,,\,c_{ji}<-1\,,\,c_{ji}\gamma_{ir}+\gamma_{js}\geqslant m
\,,\,0\geqslant m >c_{ji}\}\cup
\{i\to j\,,\,c_{ij}=-1\,,\,c_{ji}\gamma_{ir}+\gamma_{js}\geqslant m
\,,\,0\geqslant m >c_{ji}\}
\end{align*}
For any tuple $(i,r)$ in the set $\Gamma=\{(i,r)\,;\,\gamma_{i,r}>0\}$,
the polynomial $a_{w_{ir},\gamma-w_{ir}}=A_{i,r}$  involves only 
variables $\cbfw_{j,s}$ with $(j,s)\in\Gamma$.
Thus, to prove that $\bfr_\gamma^\ab$ is an $R_{T\times T_W}$-linear combination of the elements
$\bfr_{w_{j,s}}^\ab\cdot\bfr_{\gamma-w_{j,s}}^\ab$ over all tuples $(j,s)\in\Gamma$, it is enough to check that the polynomials 
$A_{j,s}$ with $(j,s)\in\Gamma$ have no common roots.
Assume that the tuple of non-zero complex numbers $(\check w_{j,s}\,;\,(j,s)\in\Gamma)$ is such a root.
In the symmetric case the polynomial $A_{i,r}$ is a product of monomials of the form $(1-\zeta\,\cbfw_{i,r}\cbfw_{j,s}^{-1})$.
Hence, given any $(i_0,r_0)\in\Gamma$ the vanishing of $A_{i_0,r_0}$ gives an element $(i_1,r_1)\in\Gamma$ such that
$\check w_{i_1,r_1}=\zeta\, w_{i_0,r_0}$ and, by induction, we get a chain of coordinates with 
$\check w_{i_k,r_k}=\zeta^k\,\check w_{i_0,r_0}$ for each $k\in\bbN$. 
This is not possible because the set $\Gamma$ is finite and $\zeta$ is not a root of unity.
In types $BCFG$ we compute all polynomials $A_{i,r}$ using \eqref{AB} and the list of weighted quivers in \S\ref{sec:NSCD}.
In this case $A_{i,r}$ is a product of monomials of the form $(1-\zeta^k\,\cbfw_{i,r}\cbfw_{j,s}^{-1})$ where the integers
$k$ all have the same sign and $k=0$ for at most one vertex $i\in I$. Then, the same argument as in the symmetric case proves the claim.
\end{proof}

\begin{Lemma}\label{lem:Ac}
The algebra $\calA^\iw$ is generated by the subset $ji(\calA^\ab)\cup k(\NH)$.
\end{Lemma}

\begin{proof}
Let $\calD^\iw\subset\calA^\iw$ be the subalgebra  generated by the subset $ji(\calA^\ab)\cup k(\NH)$.
We must prove that $\calD^\iw=\calA^\iw$.
We consider the commutative diagram
\begin{align*}
\xymatrix{
\calA^\ab\ar@{^{(}->}[r]^-i&\calA\ar@{^{(}->}[r]^-j\ar@{_{(}->}[rd]_-\phi&\calA^\iw\ar[d]^-{\pi_*}&\NH\ar@{_{(}->}[l]_-k\\
&&\calP&}
\end{align*}
We equip $\calA$, $\calA^\iw$ and $\calP$ with the filtration given by the $G_\calO$-orbits in $\calR$ and $\calR^\iw$ as in \S\ref{sec:gr}.
The morphisms $j$, $\phi$ and $\pi_*$ preserve these filtrations. Being the $T$-equivariant K-theory of $G$-equivariant spaces,
both $\calA^\iw$ and $\calP$ are equipped with a $\W$-action such that the map $\pi_*$ is $\W$-equivariant and
$\phi(\calA)=\calP^\W$. Now, for each dominant cocharacter $\gamma$ and each $f\in (R_T)^{\W_\gamma}$, we fix a lifting 
$M_{\gamma,f}$ to $\calA_{\leqslant\gamma}\subset\calA$ of the element
$f\diamond [\scrO_{\calR_\gamma}]$ in $\gr\calA$. 
For instance, we may use the Koszul-perverse extension as in \cite{CWb}.
Then $\calA$ is spanned by all elements $M_{\gamma,f}$.
Now, to prove that $\calD^\iw=\calA^\iw$ is enough to prove that $j(\calA)\subset\calD^\iw$
because the algebra $\calA^\iw$ is 
generated by the subset $j(\calA)\cup k(\NH)$ by Lemma \ref{lem:matrix-algebra}.
To do this, we claim that 
\begin{enumerate}[label=$\mathrm{(\alph*)}$,leftmargin=8mm,itemsep=1mm]
\item there is an element $C_{\gamma,f}\in\calD^\iw$ such that 
$\phi(M_{\gamma,f})=\pi_*(C_{\gamma,f})$ modulo lower terms in the filtration,
\item $\pi_*(C_{\gamma,f})\in\phi(\calA)$.
\end{enumerate}
By (a) and Lemma \ref{lem:matrix-algebra} we have $j(M_{\gamma,f})=C_{\gamma,f}\cdot\bfe$ modulo lower terms in the filtration.
By (b) we have $j(M_{\gamma,f})-C_{\gamma,f}\cdot\bfe\in j(\calA)$, hence $j(M_{\gamma,f})-C_{\gamma,f}\cdot\bfe$ is a linear combination
of $j(M_{\eta,g})$'s with $\eta<\gamma$.
Since $C_{\gamma,f}\cdot\bfe\in\calD^\iw$, we deduce by induction that $j(M_{\gamma,f})\in\calD^\iw$ as wanted.

Now, we prove the claim.
Since the map $\pi_*$ is $\W$-equivariant and $\phi(\calA)=\calP^\W$, up to applying the full symmetrizer in $\bbC\W$ to $C_{\gamma,f}$ 
it is enough to prove that there is an element $C_{\gamma,f}\in\calD^\iw$ satisfying  (a).
To do this, let $P_\gamma$ be the stabilizer of the point $[\gamma]\in \Gr$ in $G$. 
Assume that $\gamma$ is dominant. Then $P_\gamma$ is a parabolic subgroup of $G$ containing $B^-$.
We consider the following commutative diagram with Cartesian squares
\begin{align*}
\xymatrix{
\calR^\iw\ar[d]_-\pi&\ar@{_{(}->}[l]\calR_\gamma^\iw\ar[d]_-\pi\ar[r]^-p&\calX_\gamma\ar[d]_-\pi\ar[r]^\kappa&X_\gamma\ar[d]^-\pi\\
\calR&\ar@{_{(}->}[l]\calR_\gamma\ar[r]^-p&\Gr_\gamma\ar[r]^\kappa&Y_\gamma}
\end{align*}
where we have the following isomorphisms of $G$-varieties
\begin{align*}
X_\gamma&= G\cdot z^\gamma \cdot G_\calO/\calI\cong G\times_{P_\gamma}G/B^-\cong G/P_\gamma\times G/B^-,\\
Y_\gamma&=G\cdot z^\gamma\cdot  G_\calO/G_\calO\cong G/P_\gamma.
\end{align*}
The left inclusions are the obvious locally closed immersions. 
The maps $\kappa$ are affine bundles given by the evaluation at zero $G_\calO\to G$.
The maps $\pi$ are the obvious projections.
In particular $\pi:X_\gamma\to Y_\gamma$ is identified with the first projection
$G/P_\gamma\times G/B^-\to G/P_\gamma$.
Recall that
$f\diamond [\scrO_{\calR_\gamma}]=p^*\kappa^*\ind(f)$ in $K^G(\calR_\gamma)$,
 where $\ind:(R_T)^{\W_\gamma}\to K^G(Y_\gamma)$ is the induction.
 The affine nil-Hecke algebra $\NH$ acts on $K^T(X_\gamma)$ by right convolution,
with the terminology in \cite[\S 9]{M22}.
This action is compatible with the right $\NH$-action on $\calA^\iw$. 
Since the diagram above is Cartesian, by base change it is enough to observe that
\begin{enumerate}[label=$\mathrm{(\alph*)}$,leftmargin=8mm,itemsep=1mm]
\item[$\mathrm{(c)}$] there is
 an element $D_{\gamma,f}\in K^T(X_\gamma)$ such that
$\phi(\ind(f))=\pi_*(D_{\gamma,f})$ in $K^T(Y_\gamma)$, 
\item[$\mathrm{(d)}$] there is an element $C_{\gamma,f}\in\calD^\iw\cap\calA_{\leqslant\gamma}^\iw$ 
such that
$C_{\gamma,f}|_{\calR_\gamma^\iw}=p^*\kappa^*(D_{\gamma,f})$ in $K^T(\calR_\gamma^\iw)$. 
\end{enumerate}
See  \cite[lem.~2.10]{We19} for a similar argument.
\end{proof}

We can now prove the proposition.

\begin{proof}[Proof of Proposition $\ref{prop:onto}$]
We define
\begin{align*}
\Sigma&=\calA^0\cup\{f\diamond [\scrO_{\calR_{\omega_{i,\pm 1}}}]\,;\,i\in I\,,\,f\in (R_T)^{\W_{\omega_{i,\pm 1}}}\},\\
\Sigma^\iw&=j(\Sigma)\cup k(\NH).
\end{align*} 
Let $\calB$ be the subalgebra of $\calA$ generated by the subset $\Sigma$, 
and $\calB^\iw$ the subalgebra of $\calA^\iw$ generated by $\Sigma^\iw$.
It is enough to prove that $\calA=\calB$.
We claim that $\calB^\iw=\calA^\iw$.
Then, Lemma \ref{lem:matrix-algebra} yields
$$j(\calA)=\bfe\star\calA^\iw\star\bfe=\bfe\star\calB^\iw\star\bfe\subseteq j(\calB),$$
proving the proposition.
To prove the claim we first check that
\begin{align}\label{Sab-Bc}ji(\Sigma^\ab)\subset\calB^\iw.\end{align}
By Lemma \ref{lem:Aab}, we deduce that $ji(\calA^\ab)\subset\calB^\iw$.
Then, the claim follows from Lemma \ref{lem:Ac}.

Now, we prove \eqref{Sab-Bc}.
Let $\bfr_\gamma=i(\bfr_\gamma^\ab)$ and $\bfr_\gamma^\iw=ji(\bfr_\gamma^\ab)$ be the images of 
$\bfr_\gamma^\ab$ in $\calA$ and $\calA^\iw$.
Since the map $\pi$ restricts to an isomorphism $ji(\calR^\ab)\to i(\calR^\ab)$, we have 
$\phi(\bfr_\gamma)=\pi_*(\bfr_\gamma^\iw)$.
Note that the cocharacter $w_{i,\pm r}$ belongs to the $\W$-orbit of the minuscule cocharacter $\omega_{i,\pm 1}$.
Further, for any dominant minuscule cocharacter $\gamma$ and any cocharacter $\eta$ in the $\W$-orbit $\W(\gamma)$, 
there are elements $g_s\in R_T$ and
$f_s\in(R_T)^{\W_\gamma}$ such that the following identity holds in $\calP$
$$\pi_*(\bfr_\eta^\iw)=\sum_sg_s\otimes\phi(f_s\diamond [\scrO_{\calR_\gamma}])
.$$
Hence, Lemma \ref{lem:matrix-algebra} implies that
$\bfr_\eta^\iw\star\bfe=\sum_sg_s\otimes j(f_s\diamond [\scrO_{\calR_\gamma}])$ in $\calA^\iw$.
Then, a  computation similar to the proof of \cite[prop.~2.13]{We19} yields 
finitely many elements  $x_s,y_s\in R_T$ such that in $\calA^\iw$ we have
$$\bfr_\eta^\iw=\sum_sx_s\star j(f_s\diamond [\scrO_{\calR_\gamma}])\star R_{w_0}^\iw\star y_s,$$
where $x_s,$ $y_s$ are viewed as elements in $\NH$. The claim and the proposition follow.
\end{proof}

\begin{Remark} \label{rem:G2}
For symmetric types, one can prove as in \cite[prop.~3.1]{W19}
that the algebra $\calA_\mu^\rho$ is generated by the elements
$f\diamond [\scrO_{\calR_{\gamma}}]$ with $\gamma$ minuscule,
without proving Proposition \ref{prop:onto} first.
Then, in the proof of Lemma \ref{lem:Ac} it is enough to consider only minuscule dominant cocharacters $\gamma$.
Indeed, the complement of the hyperplanes 
$\{\cw_{i,r}=0\}$ and $\{\cw_{j,s}-\cw_{i,r}=0\}$  with
$(i,r)\neq(j,s)$
is a disjoint union of chambers in $\Lambda\otimes\bbR$, such that 
each closed chamber is a semigroup of the form
$$C=\{\gamma\in\Lambda\otimes\bbR\,;\,
\gamma_{i_1,r_1}\geqslant\gamma_{i_2,r_2}\geqslant\cdots\geqslant\gamma_{i_n,r_n}\geqslant 0\geqslant
\gamma_{i_{n+1},r_{n+1}}\geqslant\cdots\geqslant\gamma_{i_a,r_a}\}$$
with $a=\sum_{i\in I}a_i$ and 
$(i_1,r_1),(i_2,r_2),\dots,(i_a,r_a)$ is an ordering of the elements of the set $\{(i,r)\,;\,i\in I\,,\,r\in[1,a_i]\}$. 
The intersection $\Lambda\cap C$ is called an integral chamber.
Lemma \ref{lem:multgrA} implies that for each integral chamber contained in $\Lambda^+$
and each $\gamma,$ $\eta$ in $\Lambda\cap C$ we have
$$[\scrO_{\calR_\gamma}] \star [\scrO_{\calR_\eta}]=[\scrO_{\calR_{\gamma+\eta}}]$$
in the associated graded $\gr\calA_{\mu}^\lambda$.
Therefore, the  algebra $\calA_{\mu}^\lambda$ is generated by
$\calA_{\mu}^{\lambda,0}$ and the elements $f\diamond [\scrO_{\calR_{\leqslant\gamma}}]$ 
where $\gamma$ runs over a set of generators of the integral chambers contained in $\Lambda^+$.
A set of generators of the semigroup $\Lambda\cap C$ is
$$\Big\{\sum_{l=1}^mw_{i_l,r_l}\,;\,1\leqslant m\leqslant n\Big\}\cup
\Big\{-\sum_{l=m}^aw_{i_l,r_l}\,;\,n+1\leqslant m\leqslant a\Big\}.$$
Thus the integral chambers contained in $\Lambda^+$ 
are all generated by a subset of the set of cocharacters 
$\big\{\sum_{i\in I} \omega_{i,\pm r_i}\,;\,r_i\in[1,a_i]\big\}$.
In particular, the integral chambers contained in $\Lambda^+$ 
are all generated by minuscule cocharacters.
For non symmetric types the generators of the integral chambers  are more involved
and the argument in \cite[prop.~3.1]{W19} does not generalize easily.
\end{Remark}

\subsubsection{The truncated integral category $\scrO$}
\label{sec:HCO}

\begin{Lemma}\label{lem:A0}\hfill
\begin{enumerate}[label=$\mathrm{(\alph*)}$,leftmargin=8mm,itemsep=1mm]
\item
There is an algebra isomorphism $\calA_\mu^{\rho,0}\cong R_G$.
\item
The algebras $\calA_\mu^{\rho,0}$ and
$\calA_\mu^\rho$ are finitely generated and Noetherian.
\item
The set $\calA_\mu^{\rho,0}\star x\star\calA_\mu^{\rho,0}$ is finitely generated both as left and right $\calA_\mu^{\rho,0}$-module
for each $x\in\calA_\mu^\rho$.
\end{enumerate}
\end{Lemma}

\begin{proof}
Part (a) follows from Proposition \ref{prop:basic} and \eqref{CL}, using the identification
Part (c) follows from (b).
From (a) we deduce that the algebra $\calA_\mu^{\rho,0}$ is finitely generated and Noetherian,
because the elements $A_{i,r}^\pm$ generate $\bfU_\bmu^0$.
Finally, the argument in \cite[prop.~6.8]{BFNa}  goes through in our setting and proves that the algebra 
$\calA_\mu^\rho$ is finitely generated
and Noetherian, proving Part (b).
More precisely, Proposition \ref{prop:basic} implies that the $R$-algebra $\calA_{\mu,R}^\rho$ 
 is a flat deformation of $\calA_{\mu,1}^\rho$ 
which is commutative, hence Noetherian if it is finitely generated. 
Hence, the quotient algebra $\calA_\mu^\rho$ is finitely generated and Noetherian if
$\calA_{\mu,1}^\rho$ is finitely generated, which is proved as in \cite[prop.~6.8]{BFNa}.
\end{proof}

Let $\bfC^{(\alpha)}$ be the set of all $I$-colored multisets in $\bbC^\times$ of length $\alpha$.
Let $\Max(R_G)$ be the maximal spectrum of $R_G$. 
There is an obvious bijection 
\begin{align}\label{max}
\bfC^{(\alpha)}\cong\Max(R_G).\end{align}
Using the same notation as in \eqref{tau-gamma}, by \eqref{evenrhogamma} there is a map 
\begin{align}\label{sym}
{}^0\!\Lambda_\alpha\to\bfC^{(\alpha)}
,\quad
\gamma\mapsto
\zeta^{2\gamma}=\{\zeta^{2d_i\gamma_{i,r}}\,;\,i\in I\,,\,r\in [1,a_i]\}.
\end{align}
Let ${}^0\!\Max(R_G)$ be the image of ${}^0\!\Lambda_\alpha$  
by the map composed of $\eqref{max}$ and $\eqref{sym}$.
Let 
\begin{align}\label{mgamma}
\m_{\rho,\gamma}\in{}^0\!\Max(R_G)
\end{align} 
be the image of $\gamma\in{}^0\!\Lambda_\alpha$.
Let ${}^0\!\Max(\calA_\mu^{\rho,0})$ be the inverse image of ${}^0\!\Max(R_G)$ by the bijection 
\begin{align}\label{mm}\Max(\calA_\mu^{\rho,0})\cong \Max(R_G),\quad \bfm\mapsto\m\end{align}
in Lemma \ref{lem:A0}.
Let $\bfm_{\rho,\gamma}$ be the pre-image of $\m_{\rho,\gamma}$ by the map \eqref{mm}.
For each $\bfm\in\Max(\calA_\mu^{\rho,0})$ and each $\calA_\mu^\rho$-module $M$ we set
$$\bfW_\bfm(M)=\{x\in M\,;\,\exists k\,,\,\bfm^kx=0\}.$$
We abbreviate
\begin{align}\label{Wgamma}\bfW_{\rho,\gamma}(M)=\bfW_{\bfm_{\rho,\gamma}}(M).
\end{align}
We call $\bfW_{\rho,\gamma}(M)$ the $\ell$-weight subspace of $M$ of $\ell$-weight $\bfm_{\rho,\gamma}$.
We call  $M$ an Harish-Chandra  module if it is the sum of the subspaces
$\bfW_\bfm(M)$ as $\bfm$ runs over the maximal spectrum
$\Max(\calA_\mu^{\rho,0})$. 
We say that $M$ is even if  $\bfW_\bfm(M)=0$ unless 
$\bfm\in{}^0\!\Max(\calA_\mu^{\rho,0})$.
For each $\gamma,$ $\eta$ in ${}^0\!\Lambda_\alpha$ we set
\begin{align}\label{Agammaeta}
\calA_\mu^\rho(\gamma,\eta)=\lim_k\calA_\mu^\rho\,/\,(\bfm_{\rho,\gamma}^k\star\calA_\mu^\rho+
\calA_\mu^\rho\star\bfm_{\rho,\eta}^k).
\end{align}
By Lemma \ref{lem:loc2} below, we have
$$\calA_\mu^\rho(\gamma,\eta)\neq 0\Rightarrow\gamma,\eta\in{}^0\!\Lambda_\dalpha\ \text{for\ some}\ \dalpha.$$
Let $\dmu$ be as in \eqref{dlam1}. We set
\begin{align}\label{Amu'}
\calA_{\dmu}^{\rho,\wedge}=\bigoplus_{\gamma,\eta\in{}^0\!\Lambda_\dalpha^+}\calA_\mu^\rho(\gamma,\eta)
\end{align}
Taking the sum over the set of all refinements $\dalpha$ of $\alpha$,  we set
\begin{align}\label{Amu}
\calA_\mu^{\rho,\wedge}=\bigoplus_{\dmu}\calA_{\dmu}^{\rho,\wedge}.\end{align}
Following the terminology in \cite{DFO94},
the algebra $\calA_\mu^{\rho,0}$ is an Harish-Chandra subalgebra of $\calA_\mu^\rho$, hence
the multiplication in $\calA_\mu^\rho$ equips $\calA_\mu^{\rho,\wedge}$ with a complete topological algebra structure.
For each $\gamma$ in ${}^0\!\Lambda_\dalpha^+$, the unit of the summand $\calA_\mu^\rho(\gamma,\gamma)$
is an idempotent.
Let 
\begin{align}\label{eA}\bfe_{\rho,\gamma}\in\calA_\mu^\rho(\gamma,\gamma)\end{align} 
denote this idempotent.
Then, we have
\begin{align}\label{egamma}
\calA_\mu^\rho(\gamma,\eta)=\bfe_{\rho,\gamma}\star\calA_\mu^{\rho,\wedge} \star\bfe_{\rho,\eta}.
\end{align} 
For each module $M$ the $\calA_\mu^\rho$-action yields a map
$$\calA_\mu^\rho(\gamma,\eta)\times \bfW_{\rho,\eta}(M)\to \bfW_{\rho,\gamma}(M).$$
This yields a smooth representation of the topological algebra $\calA_\mu^{\rho,\wedge}$ on the sum 
\begin{align}\label{WM}\bfW(M)=\bigoplus_{\gamma}\bfW_{\rho,\gamma}(M),\end{align}
i.e., a representation with a continuous action where $\bfW(M)$ is given the discrete topology, such that
\begin{align}\label{egamma}\bfe_{\rho,\gamma} \bfW(M)=\bfW_{\rho,\gamma}(M).
\end{align}
The sum in \eqref{WM} runs over the set of all refinements $\dalpha$ of $\alpha$ and all $\gamma\in{}^0\!\Lambda_\dalpha^+$.
Since the morphism $\Phi:\bfU_\bmu\to\calA_\mu^\rho$ is surjective, 
the pullback yields an embedding of categories 
\begin{align}\label{Phi*}\Phi^*:\calA_\mu^\rho\text{-}\mod\subset\bfU_\bmu\text{-}\mod.\end{align}
For any even coweight $\gamma\in{}^0\!\Lambda_\alpha$ we abbreviate
\begin{align}\label{p1gamma}
\begin{split}
\|\gamma_i\|&=\gamma_{i,1}+\gamma_{i,2}+\cdots+\gamma_{i,a_i},\\
\|\gamma\|&=\sum_{i\in I}d_i\|\gamma_i\|
\end{split}
\end{align}

\begin{Definition}\label{def:O2}
We define the truncated integral categories ${}^0\scrO_\bmu^\rho$ and ${}^0\scrO^\rho$ to be
the subcategories of ${}^0\scrO_\bmu$ and ${}^0\scrO$ given by
\begin{align}\label{Oint2}{}^0\scrO_\bmu^\rho={}^0\scrO_\bmu\cap\Phi^*(\calA_\mu^\rho\text{-}\mod)
,\quad
{}^0\scrO^\rho=\bigoplus_\bmu{}^0\scrO_\bmu^\rho.\end{align}
\end{Definition}

\begin{Proposition}\label{prop:HC}
The functor $\bfW$ is an equivalence from ${}^0\scrO_\bmu^\rho$
to the category of finitely generated smooth $\calA_\mu^{\rho,\wedge}$-modules
with finite dimensional $\ell$-weight spaces such that
the set $\{\|\gamma\|\,;\,\bfW_{\rho,\gamma}(M)\neq 0\}$ is bounded above
\end{Proposition}

\begin{proof}
By \eqref{A0} and Theorem \ref{thm:Main1} below, we have
\begin{align}
\label{weight}
\Phi(\phi_i^+)\star\bfe_{\rho,\gamma}=
\zeta_i^{\|\gamma_i\|}\,\bfe_{\rho,\gamma}
\end{align}
Further the element $\phi_i^+$ of $\bfU_\bmu$
is a quantum analogue of the $i$th fundamental coweight, see \S\ref{sec:SQG}.
Hence, the claim follows from \cite[thm.~17]{DFO94}, which relates the category 
of smooth $\calA_\mu^{\rho,\wedge}$-modules with the category of $\calA_\mu^{\rho}$-modules
which are sum of their $\ell$-weight spaces.
\end{proof}

\begin{Remark}
The map $\Phi$ restricts to a surjection $\bfU_\bmu^0\to\calA_\mu^{\rho,0}$ by
Proposition \ref{prop:onto}. 
By \eqref{formula} its kernel contains the  elements
$
\bfA_{i,\pm s}^\pm
$,
$
\bfA_{i,0}^+\bfA_{i,a_i}^+-(-1)^{a_i}
$
and
$
\bfA_{i,-r}^--\zeta_i^{a_i}\bfA_{i,a_i-r}^+
$
with
$
0\leqslant r\leqslant a_i<s.
$
\end{Remark}

\subsection{Localization of Coulomb branches}

Proposition \ref{prop:HC} relates the integral truncated shifted category ${}^0\scrO_\bmu^\rho$ and the completion
$\calA_\mu^{\rho,\wedge}$ in \eqref{Amu} of the quantized Coulomb branch $\calA_\mu^\rho$ introduced in \eqref{zetarho}.
The goal of this section is to compute the topological algebra $\calA_\mu^{\rho,\wedge}$.
The main result is Theorem \ref{thm:Main1}, whose proof relies on Lemmas \ref{lem:r}, \ref{lem:loc1} and \ref{lem:loc2}.

\subsubsection{The localization theorem}\label{sec:LCB1}
Let $V$, $W$ be finite dimensional  $I$-graded vector spaces.
Let $\alpha$, $\lambda$, $\mu$ be as in \eqref{lam1}, and $\rho\in{}^0\!\Lambda_\lambda$.
We define the $\dI$-graded vector space $\dW$ as in \eqref{VW'},
and the sequence of $\dI$-graded vector spaces $(\dW^k)$ as in \eqref{dbfVW}.
Let $\dlambda$ and $\dw=(\dw_k)$ be given by
\begin{align}\label{lambdaW}
\dlambda=\dim_{\dI}\dW
,\quad
\dw_k=\dim_{\dI}\dW^k.
\end{align}
Fix an $\dI$-graded vector space $\dV$ of dimension $\dalpha$ which refines the $I$-graded vector space $V$.
We define the dimension vector $\dmu$ as in \eqref{dlam1}.
Recall the $\dG$-scheme
${}^0\scrX_{\dmu}^{\dw}$ introduced in \eqref{0X}, and the $\bbZ$-weighted QH algebras 
${}^0\widetilde\scrT_{\dmu}^{\dw}$ and ${}^0\scrT_{\dmu}^{\dw}$ introduced in \eqref{0tT} and \eqref{0T1}.
Now, let $\dalpha$ run over the set of all dimension vectors in $\bbN \dI$ which refine $\alpha$.
Set
\begin{align}\label{0Tbmu}{}^0\scrT_\bmu^{\dw}=\bigoplus_{\dmu}{}^0\scrT_{\dmu}^{\dw}
,\quad
{}^0{\widetilde\scrT_\bmu}^{\dw}=\bigoplus_{\dmu}{}^0{\widetilde\scrT_\dmu}^{\dw}\end{align}
We abbreviate
\begin{align}\label{Trho}
{}^0\scrX_{\dmu}^\rho={}^0\scrX_{\dmu}^{\dw}
,\quad
{}^0\widetilde\scrT_{\dmu}^\rho={}^0\widetilde\scrT_{\dmu}^{\dw}
,\quad
{}^0\scrT_{\dmu}^\rho={}^0\scrT_{\dmu}^{\dw}
,\quad\text{etc}.
\end{align}
We define the category of nilpotent modules ${}^0\scrT_\bmu^\rho\-\nilp$ as in Definition \ref{def:0T}.
Our goal in this section is to compare the category ${}^0\scrO_\bmu^\rho$ with the category 
${}^0\scrT_\bmu^\rho\-\nilp$.
Before to go on, we briefly recall the definition of the $\bbZ$-weighted QH algebra.
The affine group $\dG$ is as in \eqref{G'}, and
the affine $\dG$-variety $N_\dmu^\dlambda$ as in \eqref{dN}.
Recall the set
${}^0\P(\dalpha)$ of even sequences  of dimension vectors in $\bbN\dI$ with sum $\dalpha$ in \eqref{0Palpha}, and
the sets of coweights ${}^0\!\Lambda_\dalpha$ and ${}^0\!\Lambda_\dalpha^+$ in \eqref{0Lalpha}.
Following \eqref{Xrhogamma}, we abbreviate
$\scrX_\gamma^\rho=\scrX_{\dv_\gamma}^{\rho}.$
Given two even coweights $\gamma$ and $\eta$ in ${}^0\!\Lambda_\dalpha^+$ we consider the  fiber product
\begin{align}\label{Z}
\scrZ_{\gamma,\eta}^\rho=\scrX_\gamma^\rho\times_{N_\dmu^\dlambda}\scrX_\eta^\rho.\end{align}
Taking the disjoint union over all  $\gamma$, $\eta$ we get the following scheme
$${}^0\scrZ_{\dmu}^\rho=\bigsqcup_{\gamma,\eta\in{}^0\!\Lambda_\dalpha^+}\scrZ_{\gamma,\eta}^\rho.$$
Finally, applying Definition \ref{QH} to ${}^0\scrZ_{\dmu}^\rho$ yields the algebras 
\begin{align}\label{KR'}
{}_\K^0\widetilde\scrT_{\dmu}^\rho= K^{\dG}({}^0\scrZ_{\dmu}^\rho)
,\quad
{}^0\widetilde\scrT_{\dmu}^\rho= H_\bullet^{\dG}({}^0\scrZ_{\dmu}^\rho,\bbC).
\end{align}
Let ${}^0\widetilde\scrT_\bmu^{\rho,\wedge}$ be the nilpotent completion of 
${}^0\widetilde\scrT_\bmu^\rho$, 
and ${}_\K^0\widetilde\scrT_\bmu^{\rho,\wedge}$ the unipotent completion of 
${}_\K^0\widetilde\scrT_\bmu^\rho$, as in \S\ref{sec:Chern}.
Let 
\begin{align}\label{eQH2}e_{\rho,\gamma}=e_{\dw,\dv}\end{align} 
denote the idempotent in ${}^0\widetilde\scrT_\bmu^\rho$ given in \eqref{eQH1}, with  $\dw=\dw_\rho$ and $\dv=\dv_\gamma$. 
For each even coweight
$\gamma\in{}^0\!\Lambda_\dalpha^+$, each integer $r\in\bbN$ and each vertex $i\in I$, 
let $e_r(\zeta_i^{2\gamma_i})$ be the evaluation 
of the $r$th elementary symmetric function $e_r$ at the $a_i$-tuple
$\zeta_i^{2\gamma_i}$
given by \eqref{zetagamma} with $\tau=\zeta$.
Our goal  is to prove the following.

\begin{Theorem}\label{thm:Main1}\hfill
\begin{enumerate}[label=$\mathrm{(\alph*)}$,leftmargin=8mm,itemsep=1mm]
\item
There is a topological algebra isomorphism 
$\Theta:\calA_\mu^{\rho,\wedge}\to{}^0\widetilde\scrT_\bmu^{\rho,\wedge}$
which maps the element $A_{i,r}^+\star\bfe_{\rho,\gamma}$ to
$(-1)^r\zeta_i^{-\|\gamma_i\|}e_r(\zeta_i^{2\gamma_i})\,e_{\rho,\gamma}$
plus some topologically nilpotent terms
for each $\gamma\in{}^0\!\Lambda_\dalpha^+$, $i\in I$ and  $r\in\bbN$.
\item
The pushforward along the map $\Theta$ yields an equivalence of categories 
$\Theta_*:{}^0\scrO_\bmu^\rho\to{}^0\scrT_\bmu^\rho\-\nilp$.
\end{enumerate}
\end{Theorem}

\subsubsection{Preliminary lemmas}
In this section we gather some material to prove Theorem $\ref{thm:Main1}$.
Fix an even coweight $\gamma\in{}^0\!\Lambda$.
Recall the fixed point sets $\Gr^{\tilde\gamma}$ and $\calT^{\dot\gamma}$ introduced in \eqref{fpdg}.
We factorize the embedding of the fixed point locus 
$\calT^{\dot\gamma}$ into $\calT$ in the following way
\begin{align*}
\xymatrix{\calT^{\dot\gamma}\ar[r]^-{\sigma_\gamma}&\Gr^{\tilde\gamma}\times_{\Gr}\calT\ar[r]^-{\iota_\gamma}&\calT}
\end{align*}
Following the notation \eqref{tau-gamma}, let $\zeta^{2\gamma}$, $\zeta^{2\tilde\gamma}$ be the elements of 
$T$, $\widetilde T$ given by the
tuples $(\zeta_i^{2\gamma_i})$ and $(\zeta_i^{2\gamma_i},\zeta^2)$.
Recall that $(G_K)^{\tilde\gamma}$ and $(G_\calO)^{\tilde\gamma}$ are the stabilizers 
in $G_K$ and $G_\calO$ of the semisimple element  $\zeta^{2\tilde\gamma}$, see \eqref{GFP1}.
By \eqref{lim} and \eqref{sec:B.i}, we have
\begin{align*}
\Coh^{(\dot G_\calO)^{\tilde\gamma}}(\Gr^{\tilde\gamma}\times_{\Gr}\calT)
&=\colim_{n,\ell}\Coh^{(\dot G_\calO)^{\tilde\gamma}}(\Gr_n^{\tilde\gamma}\times_{\Gr_n}\calT_n^\ell),\\
\Coh^{(\dot G_\calO)^{\tilde\gamma}}(\calT^{\dot\gamma})
&=\colim_{n,\ell}\Coh^{(\dot G_\calO)^{\tilde\gamma}}\llp\calT_n^\ell)^{\dot\gamma})
\end{align*}
The map $\sigma_\gamma$ has coherent pullbacks, since the inclusion
$$(\calT_n^\ell)^{\dot\gamma}\subset\Gr_n^{\tilde\gamma}\times_{\Gr_n}\calT_n^\ell$$
is an embedding of vector bundles over $\Gr_n^{\tilde\gamma}$.
Both stacks 
$$(\Gr^{\tilde\gamma}\times_{\Gr}\calT)/(\dot G_\calO)^{\tilde\gamma}
,\quad
\calT/(\dot G_\calO)^{\tilde\gamma}$$ are ind-geometric  and ind-tamely presented by \S\ref{sec:BFN1}.
The map $\iota_\gamma$ is an almost ind-finitely presented closed immersion by \S\ref{sec:B.h}, because 
the inclusion $\Gr^{\tilde\gamma}\subset\Gr$ is almost ind-finitely presented.
Hence the pushforward $(\iota_\gamma)_*$ preserves coherence by \S\ref{sec:B.k}.
We have the functors
\begin{align*}
\xymatrix{
\Coh^{(\dot G_\calO)^{\tilde\gamma}}(\calT^{\dot\gamma})&\ar[l]_-{(\sigma_\gamma)^*}
\Coh^{(\dot G_\calO)^{\tilde\gamma}}(\Gr^{\tilde\gamma}\times_{\Gr}\calT\ar[r]^-{(\iota_\gamma)_*})
&\Coh^{(\dot G_\calO)^{\tilde\gamma}}(\calT).}
\end{align*}
Taking  $\frakR$ instead of $\calT$ we consider the ind-geometric stack 
$\frakR^{\dot\gamma}/(\dot G_\calO)^{\tilde\gamma}$ with
$\frakR^{\dot\gamma}=\frakR\times_\calT\calT^{\dot\gamma}$ and the diagram
\begin{align*}
\xymatrix{\frakR^{\dot\gamma}\ar[r]^-{\sigma_\gamma}&\Gr^{\tilde\gamma}\times_{\Gr}\frakR\ar[r]^-{\iota_\gamma}&\frakR}
\end{align*}
The map $\iota_\gamma$
is proper and almost ind-finitely presented by \S\ref{sec:B.h}.
Hence the pushforward $(\iota_\gamma)_*$ preserves coherence
by \S\ref{sec:B.k}.
The map $i_2:\frakR\to\calT$ in \eqref{R1} is an almost ind-finitely presented closed immersion.
Hence it is tamely presented by \S\ref{sec:B.g}.
Thus the map $\sigma_\gamma$ has coherent pullback by \S\ref{sec:B.j}, yielding the following functors 
\begin{align*}
\xymatrix{
\Coh^{(\dot G_\calO)^{\tilde\gamma}}(\frakR^{\dot\gamma})&\ar[l]_-{(\sigma_\gamma)^*}
\Coh^{(\dot G_\calO)^{\tilde\gamma}}(\Gr^{\tilde\gamma}\times_{\Gr}\frakR)\ar[r]^-{(\iota_\gamma)_*}
&\Coh^{(\dot G_\calO)^{\tilde\gamma}}(\frakR)}.
\end{align*}
The K-theory is invariant under nilpotent extensions and derived enhancements.
Thus, we get the diagram
\begin{align}\label{loc2}
\begin{split}
\xymatrix{
K^{(\dot G_\calO)^{\tilde\gamma}}(\calR^{\dot\gamma})&\ar[l]_-{(\sigma_\gamma)^*}
K^{(\dot G_\calO)^{\tilde\gamma}}(\Gr^{\tilde\gamma}\times_{\Gr}\calR\ar[r]^-{(\iota_\gamma)_*})
&K^{(\dot G_\calO)^{\tilde\gamma}}(\calR)}
\end{split}
\end{align}
The tensor product with finite dimensional
representations yields an $R_{(\dot G_\calO)^{\tilde\gamma}}$-module structure
on each term of \eqref{loc2}, and 
the maps $(\sigma_\gamma)^*$ and $(\iota_\gamma)_*$ are $R_{(\dot G_\calO)^{\tilde\gamma}}$-linear.
Let $\calR_0^{\dot\gamma}$ be the $\dot\gamma$-fixed point set in the scheme $\calR_0$ considered in \eqref{unit}.
Let $\bf1_\gamma$ be the class in K-theory
of the structural sheaf $\scrO_{\calR_0^{\dot\gamma}}$.
After completing at the element $\zeta^{2\gamma}\in T$ 
we get the following version of the Segal-Thomason localization theorem.
Using \eqref{GPISOM} and \eqref{FPG}, we get the isomorphisms
\begin{align}\label{RR}
 R_{(\widetilde G_\calO)^{\tilde\gamma}}|_\zeta
\cong R_{\widetilde\dP_\gamma}|_\zeta\cong R_{\dP_\gamma\times\bbG_m}|_\zeta\cong R_{\dP_\gamma}
\cong R_{(G_\calO)^{\tilde\gamma}}.
\end{align}
Therefore, as in \eqref{formal}, we can identify
\begin{align}\label{spe1}
K^{(G_\calO)^{\tilde\gamma}}(\calR^{\dot\gamma})\cong K^{(\widetilde G_\calO)^{\tilde\gamma}}(\calR^{\dot\gamma})|_\zeta
,\quad
K^{(\widetilde G_\calO)^{\tilde\gamma}}(\calR^{\dot\gamma})\cong K^{(\dot G_\calO)^{\tilde\gamma}}(\calR^{\dot\gamma})|_\rho.
\end{align}

\begin{Lemma}\label{lem:r}\hfill
\begin{enumerate}[label=$\mathrm{(\alph*)}$,leftmargin=8mm,itemsep=1mm]
\item
The pushforward
$(\iota_\gamma)_*:K^{(G_\calO)^{\tilde\gamma}}(\Gr^{\tilde\gamma}\times_{\Gr}\calR)_{\widehat{\zeta^{2\gamma}}}
\to K^{(G_\calO)^{\tilde\gamma}}(\calR)_{\widehat{\zeta^{2\gamma}}}$
is invertible.
\item
The composed map $\bfr_\gamma=(\sigma_\gamma)^*\circ(\iota_\gamma)_*^{-1}$ is an isomorphism
$\bfr_\gamma:
K^{(G_\calO)^{\tilde\gamma}}(\calR)_{\widehat{\zeta^{2\gamma}}}
\to{K^{(G_\calO)^{\tilde\gamma}}(\calR^{\dot\gamma})}_{\widehat{\zeta^{2\gamma}}}.$
\item
We have $\bfr_\gamma(\bf1)=\bf1_\gamma$.
\end{enumerate}
\end{Lemma}

\begin{proof}
Since K-theory preserves filtered colimits, from  \eqref{colim}, \eqref{lim} and
\cite[lem.~4.14]{CWb}, \S\ref{sec:B.i} we get
\begin{align}\label{limcolim}
\begin{split}
K^{(G_\calO)^{\tilde\gamma}}(\calR)&\cong\colim_{n,\ell} K^{(G_\calO)^{\tilde\gamma}}(\calR_n^\ell),\\
K^{(G_\calO)^{\tilde\gamma}}(\calR^{\dot\gamma})&\cong
\colim_{n,\ell} K^{(G_\calO)^{\tilde\gamma}}(\calR_n^{\ell,\dot\gamma}).
\end{split}
\end{align}
Let $\calR_n^{\ell,\dot\gamma}$ be the $\dot\gamma$-fixed point locus in $\calR_n^{\ell}$.
The map $\bfr_\gamma$ is compatible with the  colimits in \eqref{limcolim}. For each $n$, $\ell$ it yields a map
\begin{align}\label{STloc}
K^{(G_\calO)^{\tilde\gamma}}(\calR_n^\ell)_{\widehat{\zeta^{2\gamma}}}
\to{K^{(G_\calO)^{\tilde\gamma}}(\calR_n^{\ell,\dot\gamma})}_{\widehat{\zeta^{2\gamma}}}
\end{align}
The Segal-Thomason localization theorem applied to the scheme of finite type $\calR_n^{\ell}$ 
implies that the map \eqref{STloc} is invertible.
\end{proof}

Now, we apply the fixed point computations given in \S\ref{sec:FP2}.

\begin{Lemma}\label{lem:loc1}\hfill
\begin{enumerate}[label=$\mathrm{(\alph*)}$,leftmargin=8mm,itemsep=1mm]
\item
There is a topological  algebra isomorphism 
$(R_{(G_\calO)^{\tilde\gamma}})_{\widehat{\zeta^{2\gamma}}}
\cong(R_{\dP_\gamma})_{\widehat 1}$.
\item
The induction from $\dP_\gamma$ to $\dG$ yields a topological
$(R_{\dG})_{\widehat 1}$-module isomorphism
$$K^{(G_\calO)^{\tilde\gamma}}(\calR^{\dot\gamma,\eta})_{\widehat{\zeta^{2\gamma}}}
\cong K^{\dG}(\scrZ_{\gamma,\eta}^\rho)_{\widehat 1}.$$
\end{enumerate}
\end{Lemma}

\begin{proof} 
By Proposition \ref{prop:fixedlocus}, we have the equivariant isomorphisms of varieties
\begin{align}\label{RX}
\calR^{\dot\gamma,\eta}\cong \scrX_\eta^\rho\times_{N_\dmu^\dlambda} N_\gamma^\rho
,\quad
\scrX_\eta^\rho\cong\dG\times_{\dP_\eta} N_\eta^\rho.
\end{align}
Taking the complexified Grothendieck groups, this yields
\begin{align}\label{isom0}
\begin{split}
K^{(\dot G_\calO)^{\tilde\gamma}}(\calR^{\dot\gamma,\eta})
&\cong K^{\widetilde\dP_\gamma\times T_W}(\scrX_\eta^\rho\times_{N_\dmu^\dlambda} N_\gamma^\rho).
\end{split}
\end{align}
Now, we have the following consequence of the tensor identity.

\begin{Claim}\label{claim:induction} Let $G$ be an affine group, $P$ a parabolic subgroup, and $Y$ a $P$-variety.
Then, the induction yields an $R_G$-module isomorphism $K^P(Y)\cong K^G(G\times_PY)$.
Let $g\in P$ be a semisimple element with associated maximal ideals
$\m_P$, $\m_G$ in $R_P$ and $R_G$. We have $\m_G=\m_P\cap R_G$.
Assume that the covering $\Spec R_P\to\Spec R_G$ is \'etale at $g$.
Then the restriction  yields a topological ring isomorphism
$(R_P)_{\widehat\m_P} \cong (R_G)_{\widehat\m_G}$ and
the induction yields a topological $(R_G)_{\widehat\m_G}$-module isomorphism
$K^P(Y)_{\widehat\m_{P}}\cong K^G(G\times_PY)_{\widehat\m_{G}}.$
\end{Claim}

\noindent 
We apply the claim to the triple $Y$, $P$, $G$ such that
$$Y=\calR^{\dot\gamma,\eta},\quad
P=\widetilde\dP_\gamma\times T_W
,\quad
G=\widetilde\dG\times T_W.$$
By \eqref{Z} we have
$\scrZ_{\gamma,\eta}^\rho=\scrX_\gamma^\rho\times_{N_\dmu^\dlambda}\scrX_\eta^\rho.$
Hence \eqref{RX} yields the following isomorphism of varieties
\begin{align*}
\scrZ_{\gamma,\eta}^\rho
\cong\dG\times_{\dP_\gamma}(\scrX_\eta^\rho\times_{N_\dmu^\dlambda} N_\gamma^\rho).
\end{align*}
The isomorphism \eqref{isom0} and the induction from $\dP_\gamma$ to $\dG$ yield the following isomorphism
\begin{align}\label{isom1}
\begin{split}
K^{(\dot G_\calO)^{\tilde\gamma}}(\calR^{\dot\gamma,\eta})
&\cong K^{\widetilde\dG\times T_W}(\scrZ_{\gamma,\eta}^\rho)
\end{split}
\end{align}
Specializing \eqref{isom1} at the element $\zeta^{2\rho}$ in $T_W$ we get the isomorphism
\begin{align}\label{isom4}K^{(\widetilde G_\calO)^{\tilde\gamma}}(\calR^{\dot\gamma,\eta})
\cong K^{\widetilde\dG}(\scrZ_{\gamma,\eta}^\rho).\end{align}
Composing \eqref{isom4} with the group isomorphism \eqref{GPISOM}, we get an isomorphism
\begin{align}\label{isom4bis}K^{(\widetilde G_\calO)^{\tilde\gamma}}(\calR^{\dot\gamma,\eta})
\cong K^{\dG\times\bbG_m}(\scrZ_{\gamma,\eta}^\rho)\end{align}
which intertwines the completions at $\zeta^{2\tilde\gamma}$ and $(1,\zeta^2)$.
The lemma follows by specializing the $\bbG_m$-action on both sides.
More precisely, the $\dG\times\bbG_m$-action on  $\scrZ_{\gamma,\eta}^\rho$
factorizes through the first projection $\dG\times\bbG_m\to\dG$.
By equivariant formality, we deduce that
$$K^{\dG\times\bbG_m}(\scrZ_{\gamma,\eta}^\rho)\cong R_{\dG\times\bbG_m}\otimes_{R_\dG}K^{\dG}(\scrZ_{\gamma,\eta}^\rho)$$
and the lemma follows.
\end{proof}

Next, we compare the multiplications in the algebras $\calA_\mu^{\rho,\wedge}$ and
${}_\K^0\widetilde\scrT_\bmu^{\rho,\wedge}$.

\begin{Lemma}\label{lem:loc2}
There is a topological algebra isomorphism
$\Theta:\calA_\mu^{\rho,\wedge}\to{}_\K^0\widetilde\scrT_\bmu^{\rho,\wedge}$
which maps
$A_{i,r}^+\star\bfe_{\rho,\gamma}$ to 
$(-1)^r\zeta_i^{-\|\gamma_i\|}e_r(\zeta_i^{2\gamma_i})\,e_{\rho,\gamma}$
plus some topologically nilpotent terms
for each $\gamma\in{}^0\!\Lambda_\dalpha^+$, $i\in I$ and $r\in\bbN$.
\end{Lemma}

\begin{proof}
By \eqref{KR'} and the definition of the unipotent completion in \eqref{unipotent}, we have
$${}_\K^0\widetilde\scrT_{\dmu}^{\rho,\wedge}
\cong\bigoplus_{\gamma,\eta\in{}^0\!\Lambda_\dalpha^+}K^{\dG}(\scrZ_{\gamma,\eta}^\rho)_{\widehat 1}.$$
We have the following chain of isomorphisms
\begin{align}\label{isom3}
\begin{split}
\bigoplus_{\eta\in{}^0\!\Lambda_\dalpha^+}K^{\dG}(\scrZ_{\gamma,\eta}^\rho)_{\widehat 1}
&\cong K^{(\widetilde G_\calO)^{\tilde\gamma}}(\calR^{\dot\gamma})_{\widehat{\zeta^{2\tilde\gamma}}}\\[-3mm]
&\cong K^{(\widetilde G_\calO)^{\tilde\gamma}}(\calR)_{\widehat{\zeta^{2\tilde \gamma}}}\\
&\cong K^{\widetilde G_\calO}(\calR)_{\widehat{\zeta^{2\tilde\gamma}}}\\
&\cong(\calA_\mu^\rho)_{\widehat{\zeta^{2\gamma}}}
\end{split}
\end{align}
The first map is Lemma \ref{lem:loc1}, the second one is the isomorphism $\bfr_\gamma$ in Lemma \ref{lem:r},
the third one is as in \cite[cor.~5.4]{K06},  see also Claim \ref{claim:induction}, 
and the last one is the definition of $\calA_\mu^\rho$ in \eqref{zetarho}.
In the above $K^{\dG}(\scrZ_{\gamma,\eta}^\rho)_{\widehat 1}$ is an $(R_{\dG})_{\hat 1}$-module,
$(\calA_\mu^\rho)_{\widehat{\zeta^{2\gamma}}}$ is an $(R_{G_\calO})_{\widehat{\zeta^{2\gamma}}}$-module, 
where  $R_{G_\calO}$ is identified with $R_{\widetilde G_\calO}|_\zeta$ as in \eqref{RGG}.
Note that we have a chain of topological algebra homomorphisms
$$(R_\dG)_{\widehat 1}\to (R_{\dP_\gamma})_{\widehat 1}\cong
(R_{(\widetilde G_\calO)^{\tilde\gamma}})_{\widehat{\zeta^{2\tilde\gamma}}}\cong
(R_{\widetilde G_\calO})_{\widehat{\zeta^{2\tilde\gamma}}}$$
and that  
$(R_{(\widetilde G_\calO)^{\tilde\gamma}})_{\widehat{\zeta^{2\tilde\gamma}}}\cong
(R_{(G_\calO)^{\tilde\gamma}})_{\widehat{\zeta^{2\gamma}}}$ as in \eqref{RR},
and 
$(R_{\widetilde G_\calO})_{\widehat{\zeta^{2\tilde\gamma}}}\cong(R_{G_\calO})_{\widehat{\zeta^{2\gamma}}}$
as in \eqref{RGG}.
Both module structures are given by tensor product of equivariant sheaves with representations.

The algebra $\calA_\mu^\rho$ has three  $R_G$-module structures: 
the left and right monoidal product with $\calA_\mu^{\rho,0}$ composed with the isomorphism
$R_G\cong\calA_\mu^{\rho,0}$ in \eqref{CL}, 
and, by \eqref{formal},  the action of $R_G$ by tensor product.
Each element $f\in R_G$ yields both an element $f\otimes\bf1$ in $\calA_\mu^{\rho,0}$ and a virtual equivariant vector bundle
$\ind(f)$ on $\calR$ as in \eqref{induction}.
For each $x\in \calA_\mu^\rho$, Proposition \ref{prop:basic} yields
\begin{align}
(f\otimes{\bf1})\star x&=f\otimes x,\label{a1}\\
x\star(f\otimes{\bf1})&=x\otimes \ind(f).\label{a2}
\end{align}
In particular, the left monoidal product by $\calA_\mu^{\rho,0}$ and the tensor product by $R_G$ are the same.
Using \eqref{TT} and the definition of the completion of $\calA_\mu^\rho$ in \eqref{complete}, 
the isomorphism \eqref{isom3} can be rewritten in the following way
\begin{align}\label{isom6}
\begin{split}
\bigoplus_{\eta\in{}^0\!\Lambda_\dalpha^+}K^{\dG}(\scrZ_{\gamma,\eta}^\rho)_{\widehat 1}
\cong\bigoplus_{\eta\in{}^0\!\Lambda_\dalpha^+}K^{(G_\calO)^{\tilde\gamma}}(\calR^{\dot\gamma,\eta})_{\widehat{\zeta^{2\gamma}}}
\cong\lim_k\calA_\mu^\rho\,/\,((\bfm_{\rho,\gamma})^k\star\calA_\mu^\rho)
\end{split}
\end{align}

Next, we claim that the summand $K^{\dG}(\scrZ_{\gamma,\eta}^\rho)_{\widehat 1}$ in the left hand side is identified with 
$$\lim_k\calA_\mu^\rho\,/\,((\bfm_{\rho,\gamma})^k\star\calA_\mu^\rho+\calA_\mu^\rho\star(\bfm_{\rho,\eta})^k)$$
in the right hand side.
To do this, we must check that the action of the ideal $\bfm_{\rho,\eta}$ of $\calA_\mu^{\rho,0}$ in \eqref{mgamma}
by the right monoidal product on the summand
$K^{(G_\calO)^{\tilde\gamma}}(\calR^{\dot\gamma,\eta})$ 
is topologically nilpotent.
Composing \eqref{RGG}, the induction and the restriction, we get a map
\begin{align}\label{map5}
\xymatrix{
\ind:R_G\ar@{=}[r]^{\eqref{RGG}}&
R_{\widetilde G_\calO}|_\zeta\ar[r]^-{\eqref{inductiona}}&K^{\widetilde G_\calO}(\calT)|_\zeta\ar[r]&
K^{(\widetilde G_\calO)^{\tilde\gamma}}(\calR^{\dot\gamma,\eta})|_\zeta=K^{(G_\calO)^{\tilde\gamma}}(\calR^{\dot\gamma,\eta})}.
\end{align}
By \eqref{a2} the action of $\bfm_{\rho,\eta}$ 
by the right monoidal product on 
$K^{(G_\calO)^{\tilde\gamma}}(\calR^{\dot\gamma,\eta})$ 
coincides with the tensor product by the image of the ideal
$\m_{\rho,\eta}$ of $R_G$ under the map \eqref{map5}.
Proposition \ref{prop:fixedlocus} yields an isomorphism
$$K^{(G_\calO)^{\tilde\gamma}}(\calR^{\dot\gamma,\eta})\cong K^{\dP_\gamma}(\scrX^\rho_\eta)\cong K^{\dG}(\scrZ^\rho_{\gamma,\eta})$$
which intertwines the tensor product by the element $\ind(f)$ in \eqref{map5} for each $f\in R_G$ with
the tensor product by $\ind(f_\eta)$ where the induction is as in \eqref{inductionb} and
$f_\eta\in R_{\dP_\eta}$ is the image of $f$ by the composed map
$$\xymatrix{R_G\ar@{=}[r]&R_{\widetilde G_\calO}|_\zeta\ar[r]^-{\res}&R_{(\widetilde G_\calO)^{\tilde\eta}}|_\zeta\ar@{=}[r]&
R_{\widetilde\dP_\eta}|_\zeta\ar@{=}[r]&R_{\dP_\eta}}$$
where the third isomorphism is Lemma \ref{lem:GP} with $\gamma=\eta$.
Finally, since $f\in\m_{\rho,\eta}$ we have $f(\zeta^{2\eta})=0$, hence $f_\eta(1)=0$,
from which the claim follows.

We deduce that the isomorphism \eqref{isom6} takes the following form
\begin{align}\label{isom8}
\Theta:\bigoplus_{\eta\in{}^0\!\Lambda_\dalpha^+}K^{\dG}(\scrZ_{\gamma,\eta}^\rho)_{\widehat 1}
\to
\bigoplus_{\eta\in{}^0\!\Lambda_\dalpha^+}
\lim_k\calA_\mu^\rho\,/\,((\bfm_{\rho,\gamma})^k\star\calA_\mu^\rho+\calA_\mu^\rho\star(\bfm_{\rho,\eta})^k)
\end{align}
Taking the sum over all $\gamma$ and all refinements $\dmu$ of $\mu$ and using \eqref{Amu'} and \eqref{Amu}, 
we get the isomorphism
\begin{align}\label{KLRA}
\Theta:{}_\K^0\widetilde\scrT_\bmu^{\rho,\wedge}
\to\calA_\mu^{\rho,\wedge}
\end{align}

By \eqref{TT} and  \eqref{isom1} we have the isomorphism
\begin{align}\label{isom9}
\Xi:K^{(G_\calO)^{\tilde\gamma}}(\calR^{\dot\gamma})
\to
\bigoplus_{\eta\in{}^0\!\Lambda_\dalpha^+}K^{\dG}(\scrZ_{\gamma,\eta}^\rho)
\end{align}
which takes $\bf1_\gamma$ to the idempotent $e_{\rho,\gamma}$. 
Lemma \ref{lem:r} yields
$\bfr_\gamma({\bf1})=\bf1_\gamma$.
Hence, we have  $\Theta(e_{\rho,\gamma})=\bfe_{\rho,\gamma}$.
Since the map $\bfr_\gamma$ is $R_{(G_\calO)^{\tilde\gamma}}$-linear,
from \eqref{a1}, \eqref{L} and \eqref{formula} we deduce that
\begin{align*}
(\Theta\circ\Xi)\big((-1)^rD_i^{-1}\otimes\Lambda_{i,r}\otimes {\bf1_\gamma}\big)
=((-1)^rD_i^{-1}\otimes\Lambda_{i,r}\otimes{\bf1})\star\bfe_{\rho,\gamma}
=A_{i,r}^+\star\bfe_{\rho,\gamma}.
\end{align*}
Next, in $K^{(G_\calO)^{\tilde\gamma}}(\calR^{\dot\gamma})_{\widehat{\zeta^{2\gamma}}}$ we have
$$(-1)^rD_i^{-1}\otimes\Lambda_{i,r}\otimes {\bf1}_\gamma=(-1)^r\zeta_i^{-\|\gamma_i\|}e_r(\zeta_i^{2\gamma_i})\,\bf1_\gamma$$
plus some topologically nilpotent terms.
Therefore 
$$\Theta\big((-1)^r\zeta_i^{-\|\gamma_i\|}e_r(\zeta_i^{2\gamma_i})\,e_{\rho,\gamma})=
A_{i,r}^+\star\bfe_{\rho,\gamma}$$
plus some topologically nilpotent terms.

Finally, we claim that the isomorphism \eqref{KLRA} commutes with the multiplication.
We must check that the sum $\bfr_\dmu=\bigoplus_\gamma\bfr_\gamma$ is an algebra homomorphism
$$\bfr_\dmu:
\bigoplus_{\gamma\in{}^0\!\Lambda_\dalpha^+} K^{(G_\calO)^{\tilde\gamma}}(\calR)_{\widehat{\zeta^{2\gamma}}}\to 
\bigoplus_{\gamma\in{}^0\!\Lambda_\dalpha^+} K^{(G_\calO)^{\tilde\gamma}}(\calR^{\dot\gamma})_{\widehat{\zeta^{2\gamma}}}
\cong K^{\dG}(\calZ_\dmu^\rho).$$
To do this, we consider the following fiber diagram
\begin{align}
\label{diag7}
\begin{split}
\xymatrix{
&\Gr^{\tilde\gamma}\times_\Gr\frakR\ar[dl]_-{\iota_\gamma}&\frakR^{\dot\gamma}\ar[l]\ar[l]_-{\sigma_\gamma}\\
\frakR&\frakR_T\ar[l]_-\iota\ar[u]^-a\ar[d]_-{i_2}&\frakR_T^{\dot\gamma}\ar[l]_{b}\ar[u]_-{\iota_1}\ar[d]^-{i_1}&\\
&\calT_T&\calT_T^{\dot\gamma}\ar[l]_-c&\\
&\Gr_T\ar[u]^-{\sigma_2}\ar[ru]_-{\sigma_1}&&}
\end{split}
\end{align}
The diagram of ind-geometric stacks
\begin{align*}
\xymatrix{\frakR^{\dot\gamma}/\widetilde T_\calO\ar[r]^-{\sigma_\gamma}&
(\Gr^{\tilde\gamma}\times_{\Gr}\frakR)/\widetilde T_\calO\ar[r]^-{\iota_\gamma}&\frakR/\widetilde T_\calO}
\end{align*}
Let $\loc$ denote the localization with respect to all non zero elements in $R_{\widetilde T_\calO}$.
By definition, the map  $\bfr_\gamma$ is obtained by base change from
$$\sigma_\gamma^*\circ(\iota_\gamma)_*^{-1}: 
K^{\widetilde T_\calO}(\calR)_\loc
\to{K^{\widetilde T_\calO}(\calR^{\dot\gamma})}_\loc$$
Next, setting $H=T$ in  \S\ref{sec:H2} we get the maps
\begin{align*}
\xymatrix{
K^{\widetilde T_\calO}(\Lambda)&\ar[l]_-{z^*}
K^{\widetilde T_\calO}(\calR_T)\ar[r]^-{\iota_*}&
K^{\widetilde T_\calO}(\calR)
}
\end{align*}
This yields an isomorphism
$$z^*\circ\iota_*^{-1}:K^{\widetilde T_\calO}(\calR)_\loc
\to K^{\widetilde T_\calO}(\Lambda)_\loc$$
Taking the sum over all $\gamma$'s we get the map
$$\bfr_2:\bigoplus_{\gamma\in{}^0\!\Lambda_\dalpha^+} K^{(\widetilde G_\calO)^{\tilde\gamma}}(\calR)_{\widehat{\zeta^{2\tilde\gamma}}}
\to K^{\widetilde T_\calO}(\Lambda)_\loc.$$
We define similarly $z_1^*=\sigma_1^*\circ (i_1)_*$ and we consider the isomorphism
$$z_1^*\circ(\iota_1)_*^{-1}:K^{\widetilde T_\calO}(\calR^{\dot\gamma})_\loc
\to K^{\widetilde T_\calO}(\Lambda)_\loc$$
Taking the sum over all $\gamma$'s, we get the map
$$\bfr_1:
\bigoplus_{\gamma\in{}^0\!\Lambda_\dalpha^+} K^{(\widetilde G_\calO)^{\tilde\gamma}}(\calR^{\dot\gamma})_{\widehat{\zeta^{2\tilde\gamma}}}
\to K^{\widetilde T_\calO}(\Lambda)_\loc$$
The base change along the Cartesian squares in the diagram \eqref{diag7} yields 
$$\sigma_\gamma^*\circ a_*=(\iota_1)_*\circ b^*
,\quad
(i_1)_*\circ b^*=c^*\circ (i_2)_*$$
We deduce that
\begin{align*}
z^*\circ\iota_*^{-1}
&=\sigma_1^*\circ c^*\circ (i_2)_*\circ a_*^{-1}\circ (\iota_\gamma)_*^{-1}\\
&=\sigma_1^*\circ(i_1)_*\circ b^* \circ a_*^{-1}\circ (\iota_\gamma)_*^{-1}\\
&=\sigma_1^*\circ(i_1)_*\circ (\iota_1)_*^{-1}\circ \sigma_\gamma^*\circ (\iota_\gamma)_*^{-1}\\
&=z_1^*\circ (\iota_1)_*^{-1}\circ \sigma_\gamma^*\circ (\iota_\gamma)_*^{-1}
\end{align*}
Taking the sum over all $\gamma$'s we get the following equality $\bfr_2=\bfr_1\circ\bfr_\dmu$.
The map $\bfr_1$ is an algebra embedding because, since $\calZ_\dmu^\rho$ is of finite type,
under the isomorphism \eqref{isom9}
it reduces to \cite[thm.~5.11.1]{CG}.
The map $\bfr_2$ is also an algebra homomorphism by \S\ref{sec:H2}.
Hence $\bfr_\dmu$ is an algebra homomorphism, proving the claim, hence the theorem.
\end{proof}

\subsubsection{Proof of Theorem $\ref{thm:Main1}$}
We now prove Theorem \ref{thm:Main1}.
The isomorphism in (a)  is the composition of the Chern character in
Lemma \ref{lem:Chern} and the isomorphism in Lemma \ref{lem:loc2}.
To prove (b) we fix a smooth module $M$ of $\scrA_\mu^{\rho,\wedge}$.
Part (a) yields a ${}^0\widetilde\scrT_\bmu^\rho$-action on $M$.
Let $\Theta_*(M)$ be the resulting ${}^0\widetilde\scrT_\bmu^{\rho,\wedge}$-module.
By Proposition \ref{prop:HC}, to prove the claim it is enough to check that the set 
$$\{\|\gamma\|\,;\,\bfW_{\rho,\gamma}(M)\neq 0\}$$
is bounded above if and only if the ${}^0\widetilde\scrT_\bmu^\rho$-action 
on $\Theta_*(M)$ descends to the quotient algebra ${}^0\scrT_\bmu^\rho$.
By Lemma \ref{lem:loc2}, the isomorphism 
$\calA_\mu^{\rho,\wedge}\cong{}^0\widetilde\scrT_\bmu^{\rho,\wedge}$ 
identifies the idempotents $\bfe_{\rho,\gamma}$ and $e_{\rho,\gamma}$.
Hence, by \eqref{egamma}, we must prove that the set  
$$\{\|\gamma\|\,;\,e_{\rho,\gamma}\Theta_*(M)\neq 0\}$$
is bounded above if and only if the ${}^0\widetilde\scrT_\bmu^\rho$-action on $\Theta_*(M)$ descends to the quotient algebra 
${}^0\scrT_\bmu^\rho$.
By \eqref{cyclo}, the cyclotomic ideal of ${}^0\widetilde\scrT_{\dmu}^\rho$ is generated by the idempotents 
$e_{\rho,\gamma}$ such that
$$\max\{\rho_{\dii,s}\,;\,\dii\in\dI\,,\,s\in[1,l_i]\}<\max\{\gamma_{\dii,r}\,;\,\dii\in\dI\,,\,r\in[1,a_i]\}$$
Hence, we must check that 
\begin{align}\label{equiv}\{\|\gamma\|\,;\,e_{\rho,\gamma} \Theta_*(M)\neq 0\}\ \text{is\ bounded\ above}\ \iff
e_{\rho,\gamma} \Theta_*(M)=0\ \text{if}\ \max\{\rho_{\dii,s}\}<\max\{\gamma_{\dii,r}\}\end{align}

First, we check the inverse implication. If the right hand side of \eqref{equiv} holds, then for all 
$\gamma$ such that $e_{\rho,\gamma} \Theta_*(M)\neq 0$ we have $\max\{\gamma_{\dii,r}\}\leqslant\max\{\rho_{\dii,s}\}$,
hence $\|\gamma\|$ is bounded on $M$ by \eqref{p1gamma}.

Now, we prove the direct implication.
Assume that the left hand side of \eqref{equiv} holds, 
and that there is a coweight $\gamma\in{}^0\!\Lambda_\dalpha^+$ such that
$e_{\rho,\gamma} \Theta_*(M)\neq 0$ and 
$$\max\{\rho_{\dii,s}\}<\max\{\gamma_{\dii,r}\}.$$
Set $N=\max\{\gamma_{\dii,r}\}$.
Choose any $\gamma'\in{}^0\!\Lambda_\dalpha^+$ such that
\begin{enumerate}[label=$\mathrm{(\alph*)}$,leftmargin=8mm,itemsep=1mm]
\item[$\bullet$]
$\gamma_{\dii,r}=\gamma'_{\dii,r}$ if $\gamma_{\dii,r}\neq N$
\item[$\bullet$]
$\gamma'_{\dii,r}=N+1$  if $\gamma_{\dii,r}= N$
\end{enumerate}
To the coweights $\gamma$, $\gamma'$ we associate the sequences $\dv_\gamma,$ $\dv_{\gamma'}$ in ${}^0\P(\dalpha)$ as in 
\eqref{bijection}. 
Under the assumptions above,  from \eqref{v1} and \eqref{kappa1} we deduce that 
$\pi(\dv_\gamma)=\pi(\dv_{\gamma'})$, hence \eqref{flag} implies that
$\Fl_{\dv_\gamma}=\Fl_{\dv_{\gamma'}}.$
Thus, the diagram \eqref{diag0} yields an isomorphism of varieties
$\scrX_\gamma^\rho\cong \scrX_{\gamma'}^\rho.$
Hence, we have 
$e_{\rho,\gamma'} \Theta_*(M)\neq 0$ 
yielding a contradiction.

\subsection{Decategorification}
Let $\balpha$, $\blambda$, $\bmu$ be as in \eqref{blam} and
$\alpha$, $\lambda$, $\mu$ as in \eqref{lam2}.
Fix an even coweight $\rho\in{}^0\!\Lambda_\lambda$.
Our goal is to compute the Grothendieck groups of the categories ${}^0\scrO_\bmu^\rho$ and ${}^0\scrO^\rho$
introduced in \eqref{Oint2}.
Let the dimension vector $\dlambda$ and the dimension sequence $\dw$ be as in \eqref{lambdaW}.
Let $m$ be the number of non zero parts in $\dw$, 
and $\dww=\tau(\dw)$ be the $m$-tuple obtained by removing the zero parts.
Let ${}^0\!L(\dww)$ be the $\dfrakg$-module in \eqref{0L} and
let ${}^0\!L(\dww)_\bmu$ be the sum of the weight spaces ${}^0\!L(\dww)_\dmu$ 
over the set of all refinements $\dalpha$ of $\alpha$ with $\dmu$ is in \eqref{dlam1}.

\subsubsection{The main theorem}
Recall that an $\ell$-weight $\Psi$ is an $I$-tuple  $(\Psi_i(u\rrp$ of rational functions in $\bbC(u)$.
For each coweight $\gamma\in{}^0\!\Lambda_\alpha^+$,
we consider the $\ell$-weight $Z_{\rho,\gamma}=Z_\rho\cdot A_\gamma^{-1}$,
where $Z_\rho=(Z_{\rho,i})$, $A_\gamma=(A_{\gamma,i})$, the product of $I$-tuples is the componentwise multiplication and 
\begin{align}
\label{Psigamma}
\begin{split}
Z_{\rho,i}(u)&=\prod_{s=1}^{l_i}\big(1-\zeta_i^{2+2\rho_{i,s}}u^{-1}\big),\\
A_{\gamma,i}(u)&= A_{\gamma_i}(u)\cdot A_{\gamma_i}(\zeta_i^{-2}u)\cdot\prod_{c_{ji}<0}
\prod_{n=1}^{-c_{ji}}A_{\gamma_j}(\zeta_j^{-c_{ji}-2n}u)^{-1},\\
A_{\gamma_i}(u)&=\zeta_i^{-\|\gamma_i\|}\prod_{r=1}^{a_i}\big(1-\zeta_i^{2\gamma_{i,r}}u^{-1}\big),\\
\end{split}
\end{align}
where $\|\gamma_i\|$ is as in \eqref{p1gamma}.
Compare with formula \eqref{AZ}.
Thus, we have
\begin{align}\label{ZrhoAgamma}
Z_{\rho,i}=\prod_{s=1}^{l_i}Z_{i,2d_i+2d_i\rho_{i,s}},\quad
A_{\gamma,i}=\prod_{r=1}^{a_i}A_{i,2d_i\gamma_{i,r}}
\end{align}
where $Z_{i,r}$ is the prefundamental $\ell$-weight introduced in \S\ref{def:O1} and
$A_{i,k}$ is the $\ell$-weight given by 
\begin{align}\label{Aik}
A_{i,k}=\zeta^{-k/2}\cdot Z_{i,k}\cdot Z_{i,k+2d_i}\cdot\prod_{c_{ij}<0}\prod_{n=1}^{-c_{ij}}(Z_{j,k+b_{ij}+2d_in})^{-1}.
\end{align}
Let $Z_{\rho,\gamma,i}^+(u)$ denote the expansion of the rational function $Z_{\rho,\gamma,i}(u)$
as negative powers of $u$.
The $\ell$-weight $Z_{\rho,\gamma}$ satisfies the condition \eqref{even2}
because the coweights $\rho$ and $\gamma$ are even.
Therefore,  the simple module $L(Z_{\rho,\gamma})$ belongs to the category ${}^0\scrO_\bmu$. 
We compose the maps $\Phi$ in \eqref{Phi1} and 
$\Theta$ in Theorem \ref{thm:Main1}. This yields the map  
$$\Omega=\Theta\circ\Phi:\bfU_\bmu\to{}^0\widetilde\scrT_\bmu^{\rho,\wedge}.$$
Let $\Irr({}^0\scrO_\bmu^\rho)$ denote the set of isomorphism classes of simple objects in ${}^0\scrO_\bmu^\rho$.

\begin{Proposition}\label{prop:Main}\hfill
\begin{enumerate}[label=$\mathrm{(\alph*)}$,leftmargin=8mm,itemsep=1mm]
\item
The  map $\Omega$ is an algebra homomorphism
$\bfU_\bmu\to{}^0\widetilde\scrT_\bmu^{\rho,\wedge}$ such that
\begin{align*}
\bfA_i^+(u)\mapsto \bigoplus_{\gamma}A_{\gamma_i}(u)\cdot e_{\rho,\gamma}
,\quad
\psi_i^+(u)\mapsto\bigoplus_{\gamma}Z_{\rho,\gamma,i}^+(u)\cdot e_{\rho,\gamma}
\end{align*}
modulo nilpotent terms.
\item
We have $\Irr({}^0\scrO_\bmu^\rho)\subset\{L(Z_{\rho,\gamma})\,;\,\gamma\in{}^0\!\Lambda_\alpha^+\}$.
\item
There is an equivalence of categories 
${}^0\scrO^\rho\cong{}^0\T^\dww\text{-}\nilp$.
\end{enumerate}
\end{Proposition}

\begin{proof}
Part (a) follows from Proposition \ref{prop:Phi}, Theorem \ref{thm:Main1} and \eqref{AZ}.
The equivalence of categories $\Theta_*$ in Theorem \ref{thm:Main1} identifies the simple objects in ${}^0\scrO_\bmu^\rho$
with simple modules of the algebra ${}^0\widetilde\scrT_\bmu^{\rho,\wedge}$.
By Part (a), the $\ell$-weights of such modules are all of the form $Z_{\rho,\gamma}$, proving Part (b).
Part (c) follows from Theorem \ref{thm:Main1} and the equivalence \eqref{0TT}.
\end{proof}

We can now state the main  theorem.

\begin{Theorem}\label{thm:Main2}\hfill
\begin{enumerate}[label=$\mathrm{(\alph*)}$,leftmargin=8mm,itemsep=1mm]
\item
There is a linear isomorphism 
$\Psi:K({}^0\scrO^\rho)\to{}^0\!L(\dww)^\vee$ which identifies the subspaces
$K({}^0\scrO_\bmu^\rho)$ with ${}^0\!L(\dww)_\bmu^\vee$.
\item
There is a representation of $\dfrakg$ on the category ${}^0\scrO^\rho$ such that the $\dfrakg$-action on 
$K({}^0\scrO^\rho)$ is identified with the $\dfrakg$-module ${}^0\!L(\dww)^\vee$ via $\Psi$.
\item
The map $\Psi$ takes the classes of the simple modules in $K({}^0\scrO^\rho)$ to the elements of the
dual canonical basis. 
\end{enumerate}
\end{Theorem}

\begin{proof}
Follows from Propositions \ref{prop:Parity} and \ref{prop:Main}.
In particular, the proof of Proposition \ref{prop:Parity} implies that the representation of $\dfrakg$ on the category
$\T^\dww\text{-}\nilp$ given by the functors $\calE_\dii$ and $\calF_\dii$ in \eqref{EF}
descends to the quotient category ${}^0\T^\dww\text{-}\nilp$.
\end{proof}

\begin{Remark} Theorem \ref{thm:Main2} has the following consequences.\hfill
\begin{enumerate}[label=$\mathrm{(\alph*)}$,leftmargin=8mm,itemsep=1mm]
\item
$L(\dlambda)\subset{}^0\!L(\dww)$.
\item
If $m=1$ then ${}^0\!L(\dww)=L(\dww)=L(\dlambda)$.

\item
If $\dww$ is generic  then ${}^0\!L(\dww)=L(\dww)=L(\domega)$.

\end{enumerate}
\end{Remark}

\subsubsection{The $q$-characters of simple modules in ${}^0\scrO_\bmu^\rho$}
The equivalence of Proposition \ref{prop:Main} allows us to compute the $q$-characters of the
modules in ${}^0\scrO_\bmu^\rho$ 
in  terms of characters
over the integral $\bbZ$-weighted QH algebras. Indeed, 
for each module $M$ in $\scrO_\bmu^\rho$ the $\ell$-weight subspaces 
$\bfW_{\rho,\gamma}(M)$ and $M_{\Psi_{\rho,\gamma}}$ in \eqref{Wgamma}, \eqref{MPsi} coincide.
Thus, the $q$-character of  $M$ is
$$q\text{-}\!\ch(M)=\sum_{\gamma}\dim \bfW_{\rho,\gamma}(M)\cdot Z_{\rho,\gamma}.$$
Following \cite[\S 2.5]{KL09}, we define the $q$-character of nilpotent modules over integral $\bbZ$-weighted QH algebras as follows.

\begin{Definition}
The $q$-character of the nilpotent  ${}^0\scrT_\bmu^\rho$-module $N$ is the following formal sum
\begin{align*}
q\text{-}\!\ch(N)=\sum_{\gamma}\dim (e_{\rho,\gamma}\,N)\cdot Z_{\rho,\gamma}
\end{align*}
\end{Definition}

\begin{Corollary}\label{cor:qchar}
The equivalence of categories $\Theta_*:{}^0\scrO_\bmu^\rho\to{}^0\scrT_\bmu^\rho\-\nilp$
in Theorem $\ref{thm:Main1}$ preserves the $q$-characters. More precisely,
we have
$q\text{-}\!\ch(M)=q\text{-}\!\ch(\Theta_* M)$
for all $M\in{}^0\scrO_\bmu^\rho$.
\qed
\end{Corollary}

\subsubsection{Relation with monomials and crystals}\label{sec:crystal}

The $\dfrakg$-module ${}^0\!L(\dww)$
in \eqref{0L}  is spanned by a subset of the dual canonical basis of $L(\dww)$
by Proposition \ref{prop:Parity}.
It admits a crystal, which is a subcrystal of the tensor product crystal of $L(\dww)$.
By \eqref{0TT} and Proposition \ref{prop:Parity},  the Kashiwara operators can be described
in terms of the simple nilpotent modules of the algebra ${}^0\T^\dww$  introduced in \eqref{0T2}
in the following way: for any simple nilpotent module $M$ the module $\widetilde f_\dii(M)$ 
is the head of the module $\calF_\dii(M)$ in \eqref{EF1}, 
and $\widetilde e_\dii(M)$ is the head of the module $\calE_\dii(M)$ in \eqref{EF2}.
Let ${}^0\!B(\rho)$ denote the crystal basis of the module ${}^0\!L(\dww)$.
By Theorem \ref{thm:Main2}, the crystal basis of ${}^0\!B(\rho)$ is canonically identified with the set 
$\Irr({}^0\scrO^\rho)$ of isomorphism classes of simple objects in the  category ${}^0\scrO^\rho$.
The goal of this section is to give a combinatorial description of this crystal.

To do this, we consider the set of all Laurent monomials in the variables $\bfz_{i,k}$ for $(i,k)\in I^\bullet$.
For any sequence $u=(u_{i,k})$ of integers we consider the monomial
$$\bfz^u=\prod_{(i,k)\in I^\bullet}(\bfz_{i,k})^{u_{i,k}}.$$
We say that the monomial $\bfz^u$ is even if $u_{i,k}=0$ whenever $(i,k)\notin {}^0\!I$.
Let $\calM$ be the set of all even monomials.
Following \eqref{Aik} we define
\begin{align}\label{bfa}
\bfa_{i,k}=\bfz_{i,k}\cdot \bfz_{i,k+2d_i}\cdot\prod_{c_{ij}<0}\prod_{n=1}^{-c_{ij}}(\bfz_{j,k+b_{ij}+2d_in})^{-1}.
\end{align}
Note that
$$\bfz_{i,k}\in\calM\Rightarrow\bfa_{i,k}\in\calM.$$
Recall that a pair $(\dii,k)\in\dI^\bullet$ with $\dii=(i,r)$ is even if $k\equiv_{2d_i}r$, see \eqref{even1}.
If  $(\dii,k)$ and $\bfz^u$ are both even, let
$$u_{\dii,k}=u_{i,k}
,\quad
\bfz_{\dii,k}=\bfz_{i,k}
,\quad 
\bfa_{\dii,k}=\bfa_{i,k}.$$
Now, we define 
\begin{align}\label{crystal0}
\begin{split}
\text{wt}(\bfz^u)&=\sum_{k,\dii}u_{\dii,k}\,\delta_\dii, \\
\varepsilon_{\dii,l}(\bfz^u)&=-\sum_{k\leqslant l}u_{\dii,k},\\
\varphi_{\dii,l}(\bfz^u)&=\sum_{k> l}u_{\dii,k},\\
\varepsilon_\dii(\bfz^u)&=\max\{\varepsilon_{\dii,l}(\bfz^u)\,;\,l\in\bbZ\},\\[2mm]
\varphi_\dii(\bfz^u)&=\max\{\varphi_{\dii,l}(\bfz^u)\,;\,l\in\bbZ\},
\end{split}
\end{align}
and
\begin{align}\label{crystal1}
\begin{split}
m_\dii&=\min\{l\,;\,\varepsilon_{\dii,l}(\bfz^u)=\varepsilon_\dii(\bfz^u)\}\\
&=\min\{l\,;\,\varphi_{\dii,l}(\bfz^u)=\varphi_\dii(\bfz^u)\},\\
n_\dii&=\max\{l\,;\,\varepsilon_{\dii,l}(\bfz^u)=\varepsilon_\dii(\bfz^u)\}\\
&=\max\{l\,;\,\varphi_{\dii,l}(\bfz^u)=\varphi_\dii(\bfz^u)\}.
\end{split}
\end{align}
To avoid confusions we may write 
$n_\dii(\bfz^u)=n_\dii$ and $m_\dii(\bfz^u)=m_\dii$.
Note that 
\begin{align*}
\varepsilon_\dii(\bfz^u)\,,\,&\varphi_\dii(\bfz^u)\geqslant 0,\\
\varepsilon_\dii(\bfz^u)>0&\Rightarrow (\dii,m_\dii)\ \text{even},\\
\varphi_\dii(\bfz^u)>0&\Rightarrow (\dii,n_\dii)\ \text{even}
\end{align*}
We consider the operators $\widetilde e_\dii, \widetilde f_\dii:\calM\to\calM\sqcup\{0\}$ such that
\begin{align}\label{crystal2}
\begin{split}
\widetilde e_\dii(\bfz^u)&=\begin{cases}0&\text{if}\ \varepsilon_\dii(\bfz^u)=0\\ \bfz^u\bfa_{\dii,m_\dii}&\text{else,}\end{cases}\\
\widetilde f_\dii(\bfz^u)&=\begin{cases}0&\text{if}\ \varphi_\dii(\bfz^u)=0\\ \bfz^u\bfa_{\dii,n_\dii}^{-1}&\text{else.}\end{cases}
\end{split}
\end{align}

\begin{Proposition}\label{prop:crystal}
The tuple $(\calM, \varepsilon_\dii, \varphi_\dii, \widetilde e_\dii, \widetilde f_\dii)$ is a crystal of type $\dC$.
\end{Proposition}

\begin{proof}
We must check the axioms of a crystal in \cite[\S7.2]{K95}.
From \eqref{crystal0}, \eqref{crystal1} we deduce that 
\begin{align}\label{wtef}
\text{wt}(\bfz^u)=\sum_{\dii}(\varphi_{\dii}(\bfz^u)-\varepsilon_{\dii}(\bfz^u\rrp\,\delta_\dii.
\end{align}
Let $\bfz^u$ be an even monomial such that $\varphi_\dii(\bfz^u)>0$. Set $\bfz^{u'}=\widetilde f_\dii(\bfz^u)$.
From \eqref{bfa}, \eqref{crystal2} we deduce that
\begin{align}\label{u'}
\begin{split}
u_{\dii,k}&=u'_{\dii,k}+\delta_{k=n_\dii}+\delta_{k=n_\dii+2d_i},\\
\varphi_{\dii,l}(\bfz^u)&=\varphi_{\dii,l}(\bfz^{u'})+\delta_{l<n_\dii}+\delta_{l< n_\dii+2d_i}.
\end{split}
\end{align}
Hence, we have the following picture:
\begin{enumerate}[label=$\mathrm{(\alph*)}$,leftmargin=8mm,itemsep=1mm]
\item
if $l<n_\dii$ then $\varphi_{\dii,l}(\bfz^{u'})=\varphi_{\dii,l}(\bfz^u)-2<\varphi_\dii(\bfz^u)-1$
and $\varphi_{\dii,l}(\bfz^u)\leqslant\varphi_\dii(\bfz^u)$,
\item
if $l=n_\dii$ then
$\varphi_{\dii,l}(\bfz^{u'})=\varphi_\dii(\bfz^u)-1$ and $\varphi_{\dii,l}(\bfz^u)=\varphi_\dii(\bfz^u)$,
\item
if $l>n_\dii$ then $\varphi_{\dii,l}(\bfz^{u'})=\varphi_{\dii,l}(\bfz^u)\leqslant\varphi_\dii(\bfz^u)-1$ and
$\varphi_{\dii,l}(\bfz^u)<\varphi_\dii(\bfz^u)$.
\end{enumerate}
Thus $\varphi_\dii(\bfz^{u'})=\varphi_\dii(\bfz^u)-1$.
By \eqref{bfa} we have $\wt(\bfz^{u'})=\wt(\bfz^u)-\dC\cdot\delta_\dii$.
Hence we also have $\varepsilon_\dii(\bfz^{u'})=\varepsilon_\dii(\bfz^u)+1$ by \eqref{wtef}.
In particular $\varepsilon_\dii(\bfz^{u'})$ is positive.
From (a), (b), (c) above we deduce that $m_\dii(\bfz^{u'})=n_\dii$, thus
$\widetilde e_\dii(\bfz^{u'})=\bfz^{u'}\bfa_{\dii,n_\dii}=\bfz^u$.
Therefore, we have
$$\widetilde f_\dii(\bfz^{u})=\bfz^{u'}\Rightarrow \bfz^{u}=\widetilde e_\dii(\bfz^{u'}).$$
The converse is proved in a similar way.
\end{proof}

We say that the monomial $\bfz^u$ is $\ell$-dominant if $u_{i,k}\geqslant 0$ for all $i,k$.
Let $\calM^+\subset\calM$ be the subset of $\ell$-dominant monomials.
We consider the  map
$$\bfz:{}^0\!B(\rho)\to\calM
,\quad
L(Z_{\rho,\gamma})\mapsto \bfz_\rho \bfa_\gamma^{-1}.$$
For each $\ell$-dominant monomial $\bfz^u\in\calM^+$ let $\calM(\bfz^u)$ be the connected component of the crystal $\calM$
containing $\bfz^u$.
If $\bfz^u=\bfz_{i_1,k_1}\cdot\bfz_{i_2,k_2}\cdots\bfz_{i_n,k_n}$ we define 
$$\calM^1(\bfz^u)=\{\bfz^{u_1}\cdot\bfz^{u_2}\cdots\bfz^{u_n}\,;\,\bfz^{u_r}\in\calM(\bfz_{i_r,k_r})\,,\,r=1,2,\dots,n\}$$ 
For any positive integer $n$ we define inductively $\calM^{n+1}(\bfz^u)$ to be the union of all
$\calM^1(\bfz^{u'})$ where $\bfz^{u'}$ runs over the set of all $\ell$-dominant monomials in $\calM^n(\bfz^u)$.
Let $\calM^\infty(\bfz^u)$ be the subcrystal of $\calM$ generated by the subset $\bigcup_{n>0}\calM^n(\bfz^u)$.
Following \eqref{ZrhoAgamma}, we define the following monomials
\begin{align*}
\bfz_\rho=\prod_{i\in I}\prod_{s=1}^{l_i}\bfz_{i,2d_i+2d_i\rho_{i,s}},\quad
\bfa_\gamma=\prod_{i\in I}\prod_{r=1}^{a_i}\bfa_{i,2d_i\gamma_{i,r}}.
\end{align*}
The following generalization of Proposition \ref{prop:Parity} allows for non-symmetric Cartan matrices.

\begin{Conjecture}\hfill
\begin{enumerate}[label=$\mathrm{(\alph*)}$,leftmargin=8mm,itemsep=1mm]
\item
The map  ${}^0\!B(\rho)\subset\calM$
such that $L(Z_{\rho,\gamma})\mapsto \bfz_\rho \bfa_\gamma^{-1}$
 is an embedding of crystals.
\item
We have an isomorphism of crystals ${}^0\!B(\rho)\cong\calM^\infty(\rho)$.
\end{enumerate}
\end{Conjecture}

We will come back to this elsewhere, see \cite{VV26}.

\begin{Remark}\hfill
\begin{enumerate}[label=$\mathrm{(\alph*)}$,leftmargin=8mm,itemsep=1mm]
\item
A conjecture is presented in  \cite[\S12]{H23} regarding the set of simple modules in ${}^0\scrO^\rho$.
This conjecture is not consistent with the computations below, and actually describes only a subset of 
$\Irr({}^0\scrO^\rho)$. D. Hernandez has informed us that his conjecture can be corrected.
\item
The relation of the crystal $\calM$ with the Nakajima monomial crystal considered in \cite{N03} and \cite{K03} is unclear, because 
the parity conditions used there are not satisfied by our set $\calM$.
\end{enumerate}
\end{Remark}

\subsubsection{Example: the type $B_2$}\label{sec:B2}
Let $\bfc=B_2$, $I=\{1,2\}$ with $1>2$, $d_1=2$ and $d_2=1$. 
The fundamental coweights and coroots are $\bomega_1$, $\bomega_2$, 
$\balpha_1=2\bomega_1-\bomega_2$ and $\balpha_2=-2\bomega_1+2\bomega_2$.
We have $\dC=A_3$ with 
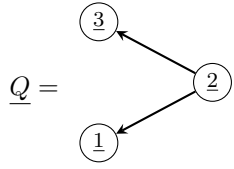
\begin{align*}
\dQ=\begin{tabular}{c@{\hspace{2cm}}c}  
\begin{tikzpicture}[scale=1, 
every node/.style={circle, draw, minimum size=0.5cm, inner sep=0pt, font=\footnotesize}, >=stealth]
\node (1) at (3,0) {\UD2};
\node (2) at (1.5,0.8) {\UD3};
\node (3) at (1.5,-0.8) {\UD1 };
\draw[->, line width=0.8pt, scale=1.5] (1) -- (2);  
\draw[->, line width=0.8pt, scale=1.5] (1) -- (3);  
\end{tikzpicture}
\end{tabular}
\end{align*}
We set $\UD1=(1,4)$, $\UD2=(2,2)$, $\UD3=(1,2)$ with $\UD1>\UD2$ and $\UD3>\UD2$.
Recall that $K({}^0\scrO^\rho)\cong{}^0\!L(\dww)^\vee$ as $\dfrakg$-modules.
Let us compare ${}^0\!L(\dww)^\vee$ with $L(\dww)^\vee$.
For each even coweight $\gamma$ let $[L(Z_{\rho,\gamma})]$
be the class in ${}^0\!L(\dww)^\vee$ of the simple object $L(Z_{\rho,\gamma})$.
If $\alpha=0$ then we write $\Lambda_\alpha=\{\emptyset\}$ in \eqref{GTLL2}.

\begin{enumerate}[label=$\mathrm{(\alph*)}$,leftmargin=8mm,itemsep=2mm]
\item
Assume that $Z_\rho(u)= \big(1-u^{-1}\,,\,1\big)$.
Then $\blambda=\bomega_1$, $\dlambda=\dww=\delta_{\UD1}$, $m=1$
and ${}^0\!L(\dww)=L(\dww)=L(\delta_{\UD1})$. 
The shift $\bmu$ and the weight  $\dmu$ are related as follows.
The weight spaces $L(\dww)_\dmu^\vee$ 
are 1-dimensional and are spanned by the classes $[L(Z_{\rho,\gamma})]$ with the following coweights $\gamma$: 
$$\begin{tabular}{|c|c|c|c|c|c|c|}
\hline
 $\bmu$ & $\bomega_1$  &$-\bomega_1+\bomega_2$   &$ \bomega_1-\bomega_2$&$-\bomega_1$\\
\hline
$\dmu$& $\delta_{\UD1}$  & $-\delta_{\UD1}+\delta_{\UD2}$  &
$-\delta_{\UD2}+\delta_{\UD3}$&$-\delta_{\UD3}$\\
\hline
$\gamma_1;\gamma_2$& $\emptyset\,;\,\emptyset$  & $-1\,;\,\emptyset$  & $-1;-2$  &$(-1,-\frac{3}{2})\,;\,-2$\\
 \hline
\end{tabular}
$$

\hfill

\item
Assume that $Z_\rho(u)=\big(1\,,\,1-u^{-1}\big)$.
Then $\blambda=\bomega_2$,  $\dlambda=\dww=\delta_{\UD2}$, $m=1$,
and ${}^0\!L(\dww)=L(\dww)=L(\delta_{\UD2})$. 
The shift $\bmu$ and the weight  $\dmu$ are related as follows:
The weight spaces $L(\dww)_\dmu^\vee$ 
are 1-dimensional and are spanned by the classes $[L(Z_{\rho,\gamma})]$ 
with the following coweights $\gamma$:
$$\begin{tabular}{|c|c|c|c|c|c|c|}
\hline
 $\bmu$ & $\bomega_2$  &$2\bomega_1-\bomega_2$   &\multicolumn{2}{c|}{{\pmb 0}} & $-2\bomega_1+\bomega_2$&$-\bomega_2$\\
\hline
$\dmu$& $\delta_{\UD2}$  & $\delta_{\UD1}-\delta_{\UD2}+\delta_{\UD3}$  & $-\delta_{\UD1}+\delta_{\UD3}$  &$\delta_{\UD1}-\delta_{\UD3}$&
$-\delta_{\UD1}+\delta_{\UD2}-\delta_{\UD3}$&$-\delta_{\UD2}$\\
\hline
$\gamma_1;\gamma_2$& $\emptyset\,;\,\emptyset$  & $\emptyset\,;\,-1$  & $-1;-1$  &$-\frac{3}{2}\,;\,-1$&$(-1,-\frac{3}{2})\,;\,-1$
&$(-1,-\frac{3}{2})\,;\,(-1,-3)$\\
 \hline
\end{tabular}
$$

\hfill

\item
Assume that $Z_\rho(u)= \big((1-u^{-1})(1-\zeta^{-2}u^{-1})\,,\,1\big)$.
Then $\blambda=2\bomega_1$, $\dlambda=\delta_{\UD1}+\delta_{\UD3}$, $\dww=(\delta_{\UD3},\delta_{\UD1})$, $m=2$ and
\begin{align*}
L(\dww)&=L(\delta_{\UD1})\otimes L(\delta_{\UD3})= L(\delta_{\UD1}+\delta_{\UD3})\oplus L(0)\\
{}^0\!L(\dww)&=L(\delta_{\UD1}+\delta_{\UD3})
\end{align*}
The shift $\bmu$ and the weight  $\dmu$ are related as follows:
\begin{align*}
&\begin{tabular}{|c|c|c|c|c|c|c|c|c|c|c|}
\hline
 $\bmu$ & $2\bomega_1$     &\multicolumn{2}{c|}{$\bomega_2$} &$-2\bomega_1+2\bomega_2$\\
\hline
$\dmu$ & $\delta_{\UD1}+\delta_{\UD3}$  & $-\delta_{\UD1}+\delta_{\UD2}+\delta_{\UD3}$  &
$\delta_{\UD1}+\delta_{\UD2}-\delta_{\UD3}$&$-\delta_{\UD1}+2\delta_{\UD2}-\delta_{\UD3}$\\
 \hline
\end{tabular}\\[2mm]
&\begin{tabular}{|c|c|c|c|c|c|c|c|c|c|c|}
\hline
 $\bmu$ & \multicolumn{2}{c|}{ $2\bomega_1-\bomega_2$}  &{\pmb 0}&\multicolumn{2}{c|}{$-2\bomega_1+\bomega_2$}\\
\hline
$\dmu$&$-\delta_{\UD2}+2\delta_{\UD3}$& $2\delta_{\UD1}-\delta_{\UD2}$    & 0  &$-2\delta_{\UD1}+\delta_{\UD2}$&
$\delta_{\UD2}-2\delta_{\UD3}$\\
 \hline
\end{tabular}\\[2mm]
&\begin{tabular}{|c|c|c|c|c|c|c|c|c|c|c|}
\hline
 $\bmu$ &$2\bomega_1-2\bomega_2$    & \multicolumn{2}{c|}{$-\bomega_2$} &$-2\bomega_1$ \\
\hline
$\dmu$& $\delta_{\UD1}-2\delta_{\UD2}+\delta_{\UD3}$      & $-\delta_{\UD1}-\delta_{\UD2}+\delta_{\UD3}$ &$\delta_{\UD1}-\delta_{\UD2}-\delta_{\UD3}$&
$-\delta_{\UD1}-\delta_{\UD3}$\\
 \hline
\end{tabular}
\end{align*}
Let us concentrate on the case $\bmu=0$. 
Then, we have $\dmu=0$.
The space ${}^0\!L(\dww)_\dmu^\vee=L(\delta_{\UD1}+\delta_{\UD3})_0^\vee$ 
is spanned by the classes  $[L(Z_{\rho,\gamma})]$ with
$\gamma=(-1,-\frac{3}{2};-2)$, $(-1,-\frac{3}{2};-3)$ and $(-\frac{3}{2},-2;-3)$.

\hfill

\item
Assume that $Z_\rho(u)= \big(1\,,\,(1-u^{-1})(1-\zeta^{-2}u^{-1})\big)$.
Then $\blambda=2\bomega_2$, $\dlambda=2\delta_{\UD2}$, $\dww=(\delta_{\UD2},\delta_{\UD2})$, $m=2$ and
\begin{align*}
L(\dww)&=L(\delta_{\UD2})\otimes L(\delta_{\UD2})=L(2\delta_{\UD2})\oplus L(\delta_{\UD1}+\delta_{\UD3})\oplus L(0)\\
{}^0\!L(\dww)&=L(2\delta_{\UD2})\oplus L(\delta_{\UD1}+\delta_{\UD3})
\end{align*}
The shift $\bmu$ and the weight  $\dmu$ are related as follows:
\begin{align*}
&\begin{tabular}{|c|c|c|c|c|c|c|c|c|c|c|}
\hline
 $\bmu$ & $2\bomega_2$  &$2\bomega_1$   &\multicolumn{2}{c|}{$\bomega_2$} & $4\bomega_1-2\bomega_2$&$-2\bomega_1+2\bomega_2$\\
\hline
$\dmu$& $2\delta_{\UD2}$  & $\delta_{\UD1}+\delta_{\UD3}$  & $-\delta_{\UD1}+\delta_{\UD2}+\delta_{\UD3}$  &
$\delta_{\UD1}+\delta_{\UD2}-\delta_{\UD3}$&$2\delta_{\UD1}-2\delta_{\UD2}+2\delta_{\UD3}$&$-\delta_{\UD1}+2\delta_{\UD2}-\delta_{\UD3}$\\
 \hline
\end{tabular}\\[2mm]
&\begin{tabular}{|c|c|c|c|c|c|c|c|c|c|c|}
\hline
 $\bmu$ & \multicolumn{2}{c|}{ $2\bomega_1-\bomega_2$}  &\multicolumn{3}{c|}{{\pmb 0}}&\multicolumn{2}{c|}{$-2\bomega_1+\bomega_2$}\\
\hline
$\dmu$&$-\delta_{\UD2}+2\delta_{\UD3}$& $2\delta_{\UD1}-\delta_{\UD2}$  & $-2\delta_{\UD1}+2\delta_{\UD3}$  & 0&$2\delta_{\UD1}-2\delta_{\UD3}$  &$-2\delta_{\UD1}+\delta_{\UD2}$&
$\delta_{\UD2}-2\delta_{\UD3}$\\
 \hline
\end{tabular}\\[2mm]
&\begin{tabular}{|c|c|c|c|c|c|c|c|c|c|c|}
\hline
 $\bmu$ &$2\bomega_1-2\bomega_2$    &$-4\bomega_1+2\bomega_2$ & \multicolumn{2}{c|}{$-\bomega_2$} &$-2\bomega_1$ &$-2\bomega_2$\\
\hline
$\dmu$& $\delta_{\UD1}-2\delta_{\UD2}+\delta_{\UD3}$    & $-2\delta_{\UD1}+2\delta_{\UD2}-2\delta_{\UD3}$  & $-\delta_{\UD1}-\delta_{\UD2}+\delta_{\UD3}$ &$\delta_{\UD1}-\delta_{\UD2}-\delta_{\UD3}$&
$-\delta_{\UD1}-\delta_{\UD3}$&$-2\delta_{\UD2}$\\
 \hline
\end{tabular}
\end{align*}
\hfill\\
Let us concentrate on the case $\bmu=0$. 
Then, we have $\dmu=-2\delta_{\UD1}+2\delta_{\UD3}$, $ 0$ or $2\delta_{\UD1}-2\delta_{\UD3}$.
The space ${}^0\!L(\dww)_\dmu^\vee$ is spanned by the classes of simple modules $[L(Z_{\rho,\gamma})]$ 
as follows:
$$\begin{tabular}{|l|c|}
\hline
$L(2\delta_{\UD2})_{-2\delta_{\UD1}+2\delta_{\UD3}}^\vee$&$\gamma=(-1,-2;-1,-2)$\\
\hline
$L(2\delta_{\UD2})_0^\vee$&
$\gamma=(-\frac{3}{2},-2;-1,-2)$,  $(-\frac{3}{2},-2;-2,-4)$\\
\hline
$L(\delta_{\UD1}+\delta_{\UD3})_0^\vee$&
$\gamma=(-1,-\frac{3}{2};-1,-2)$, $(-1,-\frac{3}{2};-1,-3)$,  $(-\frac{3}{2},-2;-1,-3)$\\
\hline
$L(2\delta_{\UD2})_{2\delta_{\UD1}-2\delta_{\UD3}}^\vee$&
$\gamma=(-\frac{3}{2},-\frac{3}{2};-1,-2)$\\
\hline
\end{tabular}
$$

\hfill

\item
Assume that $Z_\rho(u)= \big(1\,,\,(1-u^{-1})^2\big)$.
Then $\blambda=2\bomega_2$, $\dlambda=2\delta_{\UD2}$, $\dww=2\delta_{\UD2}$, $m=1$ and
\begin{align*}
L(\dww)={}^0\!L(\dww)=L(2\delta_{\UD2})
\end{align*}
The shift $\bmu$ and the weight  $\dmu$ are related as in (d).
Let us concentrate on the case $\bmu=0$. 
Then, we have $\dmu=-2\delta_{\UD1}+2\delta_{\UD3}$, $ 0$ or $2\delta_{\UD1}-2\delta_{\UD3}$.
The space ${}^0\!L(\dww)_\dmu^\vee$ is spanned by the classes $[L(Z_{\rho,\gamma})]$ 
as follows:
$$\begin{tabular}{|l|c|}
\hline
$L(2\delta_{\UD2})_{-2\delta_{\UD1}+2\delta_{\UD3}}^\vee$&$\gamma=(-1,-1;-1,-1)$\\
\hline
$L(2\delta_{\UD2})_0^\vee$ &
$\gamma=(-1,-\frac{3}{2};-1,-1)$,  $(-1,-\frac{3}{2};-1,-3)$\\
\hline
$L(2\delta_{\UD2})_{2\delta_{\UD1}-2\delta_{\UD3}}^\vee$&
$\gamma=(-\frac{3}{2},-\frac{3}{2};-1,-1)$\\
\hline
\end{tabular}
$$

\end{enumerate}

\bigskip


\section{Appendix}

\subsection{Proof of Proposition \ref{prop:Phi}}

Let $\cw_{i,r}$ with $(i,r)\in I\times [1,a_i]$ be the basic characters of the torus $T$,
and $\cz_{i,s}$ with $(i,s)\in I\times [1,l_i]$ the basic characters of the torus $T_W$.
Let $\cbfw_{i,r}$ and $\cbfz_{i,s}$ be their images in representation rings $R_{\widetilde T}$ and $R_{\dot G}$.
We have
\begin{align}\label{RT2}
R_T=\bbC[\,\cbfw_{i,r}^{\pm 1}\,;\,(i,r)\in I\times [1,a_i\rrb
,\quad
R_{T_W}=\bbC[\,\cbfz_{i,s}^{\pm 1}\,;\,(i,s)\in I\times [1,l_i\rrb
\end{align}
Recall that $F$ is the fraction field of the ring $R$  in \eqref{zeta}.
Let $\widetilde\calA_R$ be the $R\otimes R_{T_W}$-algebra generated by
\begin{align*}
\cbfw_{i,r}^{\pm 1},\quad
\bfD_{i,r}^{\pm 1},\quad
(\cbfw_{i,r}-\bfzeta_i^n\cbfw_{i,s})^{-1}
,\quad
r\neq s
,\quad
n\in\bbZ
\end{align*}
modulo the relations
\begin{align}\label{DW}
\begin{split}
[\bfD_{i,r}\,,\,\bfD_{j,s}]&=[\cbfw_{i,r}^{\frac12}\,,\,\cbfw_{j,s}^{\frac12}]=0,\\[2mm]
\bfD_{i,r}^{\pm 1}\bfD_{i,s}^{\mp 1}&=\cbfw_{i,r}^{\pm \frac12}\cbfw_{i,r}^{\mp \frac12}=1,\\[2mm]
\bfD_{i,r}\cbfw_{j,s}^{\frac12}&=\bfzeta_i^{\delta_{i,j}\delta_{r,s}}\cbfw_{j,s}^{\frac12}\bfD_{i,r}
\end{split}
\end{align}
and set $\widetilde\calA_F=\widetilde\calA_R\otimes_RF.$
We define
\begin{align*}
\bfW_i(u)&=\prod_{s=1}^{a_i}\big(1-\cbfw_{i,s}u^{-1}\big),\\
\bfW_{i,r}&=\prod_{\substack{s=1\\s\neq r}}^{a_i}\big(1-\cbfw_{i,r}^{-1}\cbfw_{i,s}\big),\\[-1mm]
\bfW_i^+&=  \displaystyle\prod_{i\to j}
      \prod_{s=1}^{a_j}\prod_{n=1}^{-c_{ji}}
      (1-\bfzeta_i^{-2}\bfzeta_j^{2n+c_{ji}}\cbfw_{i,r}^{-1} \cbfw_{j,s}),\\
\bfW_i^-&= \displaystyle\prod_{j\rightarrow i}
      \prod_{s=1}^{a_j}\prod_{n=1}^{-c_{ji}}
      (1-\bfzeta_j^{2n+c_{ji}} \cbfw_{i,r}^{-1} \cbfw_{j,s}).
\end{align*}
Following \eqref{AZ}, \eqref{L} and \eqref{apm} we also set
\begin{gather*}
\bfZ_i(u)=\prod_{s=1}^{l_i}\big(1-\bfzeta_i^2\cbfz_{i,s}u^{-1}\big),\\
\bfD_i=\prod_{t=1}^{a_i}(\cbfw_{i,t})^{\frac12},\quad
\bfD_i^+=\prod_{i\to j} (\bfD_j)^{-c_{ji}},\quad
\bfD_i^-=\prod_{j\to i} (\bfD_j)^{-c_{ji}},\\
a_i^+=-\sum_{i\to j}a_jc_{ji},\quad a_i^-=-\sum_{i\to j}a_jc_{ij}
\end{gather*}
Let $f(u)^\pm$ be the expansion of a rational function $f(u)$ in $u^{\mp 1}$.
Proposition \ref{prop:Phi} is a consequence of \eqref{formula} and the following lemmas, the first of which is proved in \cite[thm.~7.1]{FT19}.
Note that the morphism $\Phi:\bfU_{\bmu,R}\otimes R_{T_W}\to\calA_{\mu,R}^\lambda$ there takes
$\cbfw_{i,r}$, $\cbfz_{i,s}$ and $\phi_i^+$ to $\zeta_i^{2\gamma_{i,r}}$, $\zeta_i^{2\rho_{i,s}}$ and $\zeta_i^{\|\gamma_i\|}$ respectively.

\begin{Lemma}\label{lem:A1}
There is an $F\otimes R_{T_W}$-algebra homomorphism
$\bfU_{\bmu,F}\otimes R_{T_W}\to \widetilde\calA_F$
such that
 \begin{align*}
 \bfE_i(u)&\mapsto 
   \bfD_i^2\cdot (\bfD_i^-)^{-1}\cdot \bfZ_i(u)\cdot
   \sum_{r=1}^{a_i} \delta\left(\cbfw_{i,r}u^{-1}\right)
  \cdot \bfW_i^-\cdot \bfW_{i,r}^{-1}\cdot \bfD_{i,r}^{-1},\\
 \bfF_i(u)&\mapsto -\bfzeta_i^{-1}\cdot(\bfD_i^+)^{-1}\cdot
   \sum_{r=1}^{a_i} \delta\left(\bfzeta_i^2\cbfw_{i,r}u^{-1}\right)\cdot
  \bfW_i^+\cdot \bfW_{i,r}^{-1}\cdot \bfD_{i,r},\\
 \psi^\pm_i(u)&\mapsto \bfD_i^2 \cdot(\bfD_i^+\bfD_i^-)^{-1}\cdot
   \Big(\bfZ_i(u)\cdot \bfW_i(u)^{-1}\cdot \bfW_i(\bfzeta_i^{-2}u)^{-1}\cdot
   \prod_{c_{ji}<0} \prod_{n=1}^{-c_{ji}} \bfW_j(\bfzeta_j^{-c_{ji}-2n}u)\Big)^\pm,\\[1mm]
\bfA_i^+(u)&\mapsto \bfD_i^{-1}\cdot \bfW_i(u),\\[3mm]
\bfA_i^-(u)&\mapsto \bfzeta_i^{a_i}\cdot u^{a_i}\cdot \bfD_i^{-1}\cdot \bfW_i(u),\\[3mm]
 \phi^+_i\,,\,\phi^-_i&\mapsto \bfD_i\,,\, (-\bfzeta_i)^{-a_i}\cdot \bfD_i^{-1}.
   \end{align*}
\qed
\end{Lemma}

\begin{Lemma}\label{lem:A2}
There is an $R\otimes R_{T_W}$-algebra embedding
$\calA_{\mu,R}^\lambda\to\widetilde\calA_R\otimes R_{T_W}$
such that, for each integer $N$,
\begin{align*}
(\scrS_i)^{\otimes N}
&\mapsto-(-\bfzeta_i^2)^{-a_i}\sum_{r=1}^{a_i}(\bfzeta_i^{-2}\cbfw_{i,r})^{N-a_i}\cdot \bfD_i^2\cdot \bfZ_i(\cbfw_{i,r})\cdot
\bfW_i^-\cdot \bfW_{i,r}^{-1}\cdot \bfD_{i,r}^{-1}  ,\\
(\scrQ_i)^{\otimes N}
&\mapsto(-\bfzeta_i^2)^{a_i^+}\bfzeta_i^{-a_i^-}\sum_{r=1}^{a_i}(\cbfw_{i,r})^{N+a_i^+}\cdot
(\bfD_i^+)^{-2} \cdot \bfW_i^+\cdot \bfW_{i,r}^{-1}\cdot \bfD_{i,r}.
\end{align*}
\end{Lemma}

\begin{proof}
By \eqref{iota*} and \eqref{z*}, there are algebra homomorphisms
\begin{align*}
&\iota_*:K^{\dot T_\calO}(\calR_T)^\W\to\calA_{\mu,R}^\lambda
,\\
&z^*:K^{\dot T_\calO}(\calR_T)\to K^{\dot T_\calO}(\Lambda).
\end{align*}
The map $\iota_*$ is invertible after some localization.
Composing $z^*$ with $(\iota_*)^{-1}$ we get an algebra embedding
$$\bfr:\calA_{\mu,R}^\lambda\to K^{\dot T_\calO}(\Lambda)_\loc=\widetilde\calA_R$$
where the subscript $\loc$ means that the elements $\cbfw_{i,r}-\bfzeta_i^n\cbfw_{i,s}$ are inverted.
Compare \cite[\S8.1]{FT19}.
Now, we compute the image of $(\scrQ_i)^{\otimes N}$. 
We have
\begin{align*}
\bfr([\scrO_{\calR_{\pm\delta_i}}])
&= \sum_{[\gamma]\in(\Gr_{\omega_{i,\pm1}})^T}
\Wedge\Big(\frac{z^\gamma  N_\calO}{ N_\calO\cap z^{\gamma} N_\calO}\Big)\cdot
\Wedge\Big(T_{[\gamma]}\Gr_{\omega_{i,\pm 1}}\Big)^{-1} 
\end{align*}
see, e.g.,  \cite[prop.~A.2]{BFNb}.
We have
\begin{align*}
\Gr_{\omega_{i,1}}&=\{tV_{i,\calO}\subset L\subset V_{i,\calO}\,;\,\dim(V_{i,\calO}/L)=1\},\\
(\Gr_{\omega_{i,1}})^T&=\{[w_{i,r}]\,;\,r=1,\dots,a_i\}.
\end{align*}
By \eqref{weightsN}, the quotient
$$\frac{z^{w_{i,r}}  N_\calO}{ N_\calO\cap z^{w_{i,r}} N_\calO}$$ 
decomposes as the sum of the following characters of $\dot T$
$$\big(\cw_{j,s}-\cw_{i,r},d_i(c_{ij}/2-1)+d_j(n+1),0\big)
,\quad
i\to j\in E
,\quad
n\in[0,-c_{ji}-1].$$
We deduce that
\begin{align*}
\bfr( (\scrQ_i)^{\otimes N})
&= \sum_{r=1}^{a_i} (\cbfw_{i,r})^N\cdot
    \displaystyle\prod_{i\to j}\prod_{s=1}^{a_j}\prod_{n=1}^{-c_{ji}}
      (1-\bfzeta_i^2\bfzeta_j^{-2n-c_{ji}} \cbfw_{i,r}\cbfw_{j,s}^{-1})\cdot
    \prod_{\substack{s=1\\s\neq r}}^{a_i} (1-\cbfw_{i,r}^{-1}\cbfw_{i,s})^{-1}
    \cdot \bfD_{i,r},\\
&=\sum_{r=1}^{a_i}(\cbfw_{i,r})^N\cdot\prod_{i\to j}\prod_{s=1}^{a_j}\prod_{n=1}^{-c_{ji}}(-\bfzeta_i^2\bfzeta_j^{-2n-c_{ji}} \cbfw_{i,r}\cbfw_{j,s}^{-1}) \cdot \bfW_i^+\cdot \bfW_{i,r}^{-1}\cdot \bfD_{i,r},\\
&=(-\bfzeta_i^2)^{a_i^+}\sum_{r=1}^{a_i}(\cbfw_{i,r})^{N+a_i^+}\cdot
\prod_{i\to j}\prod_{s=1}^{a_j}\prod_{n=1}^{-c_{ji}}(\bfzeta_j^{-2n-c_{ji}}\cbfw_{j,s}^{-1}) \cdot \bfW_i^+\cdot \bfW_{i,r}^{-1}\cdot \bfD_{i,r}.
\end{align*}
Compare with \cite[\S 8.1]{FT19} for the symmetric case, or with \cite[\S5(ii)]{NW23} for the cohomological version.
We compute the image of $(\scrS_i)^{\otimes N}$ in a similar way.
We have
$$(\Gr_{\omega_{i,-1}})^T=\{[-w_{i,r}]\,;\,r=1,\dots,a_i\}.$$
The  quotient
$$\frac{z^{-w_{i,r}}  N_\calO}{ N_\calO\cap z^{-w_{i,r}} N_\calO}$$ 
is the sum of the following characters of $\dot T$
\hfill
\begin{enumerate}[label=$\mathrm{(\alph*)}$,leftmargin=8mm,itemsep=1mm]
\item[$\bullet$]
$\big(\cw_{i,r}-\cw_{j,s},-d_j(n+1+c_{ji}/2),0\big)$ for  all $j\to i$, $(j,s)$ and $n\in[0,-c_{ji}-1]$,
\item[$\bullet$]
$\big(\cw_{i,r},-d_i,-\cz_{i,s}\big)$ for all $(i,s)$.
\end{enumerate}
We deduce that
\begin{align*}
\bfr( (\scrS_i)^{\otimes N})
&=\sum_{r=1}^{a_i} (\bfzeta_i^{-2}\cbfw_{i,r})^N\cdot
    \prod_{s=1}^{l_i} (1-\bfzeta_i^2\cbfz_{i,s} \cbfw_{i,r}^{-1})\cdot
    \displaystyle\prod_{j\to i}
      \prod_{s=1}^{a_j}\prod_{n=1}^{-c_{ji}}
      (1-\bfzeta_j^{2n+c_{ji}} \cbfw_{i,r}^{-1}\cbfw_{j,s})\cdot\\
   &\qquad \cdot{\displaystyle\prod_{\substack{s=1\\s\neq r}}^{a_i}(1-\cbfw_{i,r}\cbfw_{i,s}^{-1})^{-1}}
    \cdot \bfD_{i,r}^{-1},\\
&=(-1)^{a_i-1}\bfD_i^2\cdot(\cbfw_{i,r})^{-a_i}\cdot\sum_{r=1}^{a_i}(\bfzeta_i^{-2}\cbfw_{i,r})^N\cdot \bfZ_i(\cbfw_{i,r})\cdot
\bfW_i^-\cdot \bfW_{i,r}^{-1}\cdot \bfD_{i,r}^{-1}    
\end{align*}
\end{proof}

\subsection{$\infty$-Categories and stacks}\label{sec:CAT}

The paper uses coherent sheaves over $\infty$-stacks. 
Let us recall a few facts.
The terminology is the same as in \cite[\S 2]{CWb}.
Unless specified otherwise, all algebras are over $\bbC$ and all limit and colimit diagrams are small.

\subsubsection{}\label{sec:B.a}
Let $\CAlg$ be the $\infty$-category of nonpositively graded commutative dg algebras, and, given $A \in \CAlg$, 
let $A\-\CAlg$ be the $\infty$-category of nonpositively graded commutative dg $A$-algebras.
 We say that $B\in A\-\CAlg$ is $n$-truncated if $H^k(B) = 0$ for $k < -n$, and that it 
is truncated if it is $n$-truncated for some $n$. Let $A\-\Mod$ be the $\infty$-category of $A$-modules.
 An $A$-algebra $B$ is finitely $n$-presented if it is a compact object of the category of $n$-truncated $A$-algebras,
 i.e., $\Hom_A(B,-)$ commutes with arbitrary direct sums.
It  is almost finitely presented if its truncation $\tau_{<n} B$ is finitely $n$-presented for all $n$. 
It is strictly tamely  $n$-presented if it is a filtered colimit of finitely 
 $n$-presented $A$-algebras $B_n$ such that $B$ is flat over each $B_n$.
 It is \emph{strictly tamely presented} if $\tau_{<n} B$ is strictly tamely $n$-presented for all $n$.
 Strictly tamely presented algebras are coherent, i.e., $H^0(A)$ is a coherent ordinary ring and $H^n(A)$ is finitely 
presented over  $H^0(A)$ for all $n$. 

\begin{Example}
If $A$ is an ordinary algebra then $A\-\Mod$ is the enhanced unbounded derived category of ordinary $A$-modules.   
Assume $A$ is Noetherian. An ordinary $A$-algebra $B$ is almost finitely presented if and only if it is finitely presented in the 
usual sense. It is strictly tamely presented if and only if it is the union of finitely presented $A$-subalgebras $B_n$ 
such that $B$ is flat over each $B_n$.
\end{Example}

\subsubsection{}\label{sec:B.d}
Let $\scrS$ be the category of $\infty$-groupoids.
An $\infty$-stack is a functor $X\to\scrS$ which is a  sheaf for the fpqc topology.  
Let $X^\cl$ be the underlying classical stack. 
The stack $X$ is \emph{geometric} if its diagonal is affine and there is a faithfully flat morphism $\Spec A \to X$.
It is \emph{convergent} if $X(A) \cong \lim_n X(\tau_{<n}A)$ for any $A \in \CAlg$.
It is \emph{truncated} if $A$ is truncated for any flat $\Spec A \to X$.

\subsubsection{}\label{sec:B.e}
Let $f:X \to Y$ be a morphism of $\infty$-stacks.
The morphism $f$ is a closed immersion if the underlying morphism of classical stacks $f^\cl$ is a closed immersion.
It is geometric if $X \times_Y \Spec A$ is geometric for any $\Spec A \to Y$.
It is \emph{almost finitely presented} if \cite[def.~17.4.11]{L18} holds.
It is \emph{tamely presented} if it is geometric and there is a flat cover $\Spec B \to X\times_Y \Spec A $ 
such that $B$ is a strictly tamely presented $A$-algebra
for any $\Spec A \to Y$.
A geometric stack $X$ is tamely presented if the map $X\to\Spec \bbC$ is tamely presented.

\begin{Example}\hfill
\begin{enumerate}[label=$\mathrm{(\alph*)}$,leftmargin=8mm,itemsep=1mm]
\item
If $X$, $Y$ are geometric then $f$ is also geometric.
\item
If $f$ is representable, geometric, and almost finitely presented, then is tamely presented \cite[prop.~4.7]{CWb2}.
\item
If $G$ is a classical affine group scheme acting on an ind-scheme $X$, we write $X/G$ for the fpqc quotient. 
If $X$ is a tamely presented scheme, then $X/G$ is a tamely presented geometric stack \cite[prop.~4.11]{CWb2}.
\end{enumerate}
\end{Example}

\subsubsection{}\label{sec:B.g}
An ind-geometric stack is a convergent $\infty$-stack $X=\colim_\alpha X_\alpha$ which is a filtered colimit of 
truncated geometric stacks  along closed immersions in the category of convergent stacks.
It is reasonable if the structure maps are almost finitely presented.
It is ind-tamely presented if it is reasonable and the $X_\al$'s are tamely presented. 
A classical ind-scheme $X$ is ind-geometric if the $X_\al$ are quasi-compact and separated.
Then the quotient $X/G$ by a classical affine group scheme $G$ is also ind-geometric.

\subsubsection{}\label{sec:B.h}
A morphism of geometric stacks $f:X \to Y$ is proper if  the fiber product 
$X \times_Y \Spec A$ is proper over $\Spec A$ in the sense of \cite[def.~5.1.2.1]{L18} for any $\Spec A \to Y$.
Proper morphisms and almost finitely presented morphisms are stable under composition and base change by 
\cite[prop.~3.18, 3.19]{CWb1}.
A morphism $f: X \to Y$ of ind-geometric stacks is ind-proper (resp. an ind-closed immersion, almost ind-finitely presented)
if it is the colimit of morphisms $f_\alpha: X_\alpha \to Y_\alpha$ of ind-geometric presentations of $X$, $Y$ such that
each $f_\alpha$ is proper (resp. a closed immersion, almost finitely presented).

\subsubsection{}\label{sec:B.i}
If $X$ is an $\infty$-stack let $\QCoh(X)$ is the $\infty$-category of quasi-coherent sheaves on $X$. 
It is the limit of the categories $A$-$\Mod$ over all maps $\Spec A \to X$. 
 If $X$ is truncated and geometric then $\Coh(X) \subset \QCoh(X)$ is the subcategory of coherent sheaves, 
 i.e., bounded almost perfect objects. 
 If $X$ is tamely presented then the coherent sheaves are the bounded objects with coherent cohomology sheaves. 
 If $X=\colim_\alpha X_\alpha$ is a reasonnable ind-geometric stack 
 then $\Coh(X) \cong \colim \Coh(X_\al)$, see \cite[prop.~5.5]{CWb1}.
 
 \begin{Example} Let $G$ be a classical affine group scheme acting on an ind-scheme $X$.
We write 
$$\QCoh^G(X)=\QCoh(X/G)
,\quad
\Coh^G(X)=\Coh(X/G).$$
If $X$ is a scheme such that
 $X=\lim_\alpha X_\alpha$ is a filtered limit of Noetherian $G$-schemes along flat $G$-equivariant affine morphisms
then $\Coh^G(X)=\colim_\alpha\Coh^G(X_\alpha)$ by  \cite[lem.~4.14]{CWb}.
\end{Example}

\subsubsection{}\label{sec:B.j}
A morphism $f: X \to Y$ of truncated geometric stacks
has coherent pullback if $f^*(\Coh(Y\rrp\subset\Coh(X)$. 
It has stable coherent pullback if for any tamely presented morphism $Z \to Y$ with  truncated $Z$, 
the base change $X \times_Y Z \to Z$ has coherent pullback. 
If $Y$ is a tamely presented scheme, then $f$ having coherent pullback implies it has stable coherent pullback.
If $f:X\to Y$ is a geometric morphism of ind-geometric stacks and $Y=\colim_\alpha Y_\alpha$ is a reasonable presentation,
then $f$ has stable coherent pullback if the base change to every $Y_\alpha$ has stable coherent pullback
\cite[prop.~5.13]{CWb2}.

\subsubsection{}\label{sec:B.k}
Let $f: X \to Y$ be a geometric morphism of ind-tamely presented ind-geometric stacks.
It  has stable coherent pullback if its base change to each $Y_\al$ does. 
If $f$ is proper and almost ind-finitely presented then the pushforward $f_*$ preserves coherence by
\cite[prop.~3.19]{CWb1}.
Pushforward along a proper morphism is 
continuous, see \cite[prop.~3.14]{CWb1},
and satisfies base change with respect to $*$-pullback, see \cite[prop.~6.5]{CWb2}.



\begin{thebibliography}{XX}


\bibitem{BFNa} A. Braverman, M. Finkelberg, H. Nakajima, Towards a mathematical definition of Coulomb branches of 
3-dimensional  N=4  gauge theories, II, Adv. Theor. Math. Phys. 22 (2019), 1071-1147

\bibitem{BFNb} A. Braverman, M. Finkelberg, H. Nakajima, 
Coulomb branches of quiver gauge theories and slices in the affine Grassmannian 
(with appendices by A. Braverman, M. Finkelberg, J. Kamnitzer, R. Kodera, H. Nakajima, B. Webster, A. Weekes), 
Adv. Theor. Math. Phys. 23 (2019), 75-166.

\bibitem{BFNc} A. Braverman, M. Finkelberg, H. Nakajima, Ring objects in the equivariant derived Satake category arising from
Coulomb branches, Adv. Theor. Math. Phys. 23 (2019), 253-344.

\bibitem{BW16} H. Bao, W. Wang, Canonical bases in tensor products revisited, Amer. J. Math. 138 (2016), 1731-1738.

\bibitem{BM11} A. Buch, L. Mihalcea, Quantum K-theory of Grassmannians, Duke Math. J. 156 (2011), 501-538.


\bibitem{CWa}
S. Cautis, H. Williams, Cluster theory of the coherent Satake category, {J. Amer. Math. Soc.} 32 (2019), 709-778.

\bibitem{CWb}
S. Cautis, H. Williams, Canonical bases for Coulomb branches of 4d $\cN=2$ gauge theories, arxiv:2306.03023.

\bibitem{CWb1}
S. Cautis, H. Williams, Ind-geometric stacks, arXiv:2306:03043.

\bibitem{CWb2}
S. Cautis, H. Williams, Tamely presented morphisms and coherent pullback, arxiv:2306.03119.

\bibitem{CG} N. Chriss, V. Ginzburg, Representation theory and complex geometry, Birkh\"auser, 1997.

\bibitem{DFO94} Y. Drozd, V. Futorny, S. Ovsienko, Harish-Chandra subalgebras and Gelfand-Zetlin modules,
Finite-dimensional algebras and related topics (Ottawa, ON, 1992), NATO adv. Sci. Inst. Ser. C math. Phys. Sci. 424, Kluwer Acad. Publ., Dordrecht, 79-93 (1994).

\bibitem{F03} G. Faltings, Algebraic loop groups and moduli spaces of bundles, J. Eur. Math. Soc. 5 (2003), 41-68.


\bibitem{FT19}
M. Finkelberg, A. Tsymbaliuk, Multiplicative Slices, Relativistic Toda and Shifted Quantum Affine Algebras, Representations and Nilpotent Orbits of Lie Algebraic Systems, Prog. in Math. 330 (2019), 133--304.


\bibitem{GHL24} C. Geiss, D. Hernandez, B. Leclerc, 
Representations of shifted quantum affine algebras and cluster algebras I. The simply-laced case, arXiv:2401.04616.




\bibitem{H23} D. Hernandez, Representations of shifted quantum affine algebras, Int. Math. Res. Not. IMRN 13 (2023) 11035-11126.

\bibitem{HL16} D. Hernandez, B. Leclerc,  
A cluster algebra approach to q-characters of Kirillov-Reshetikhin modules, J. Eur. Math. Soc. 18 (2016), 1113-1159.

\bibitem{KTWWY19a}
J. Kamnitzer, P. Tingley, B. Webster, A. Weekes, O. Yacobi,
Highest weights for truncated shifted Yangians and product monomial crystals, J. Comb. Algebra 3 (2019), 237-303.

\bibitem{KTWWY19b}
J. Kamnitzer, P. Tingley, B. Webster, A. Weekes, O. Yacobi,
On category $\calO$ for affine Grassmannian slices, Proc. London Math. Soc. 119 (2019), 1179-1233.

\bibitem{K95} M. Kashiwara, On crystal bases, Representations of Groups, Proc. of the 1994 Annual Seminar of the Canadian 
Math. Soc. Ban 16 (1995) 155-197, Amer. Math. Soc., Providence, RI. 

\bibitem{K03} M. Kashiwara, Realizations of crystals,
in : Combinatorial and geometric representation theory (Seoul, 2001), 
Contemp. Math., 141-160 (2003).

\bibitem{K06}
S. Kato, On the combinatorics of Unramified Admissible Modules, Publ. RIMS, Kyoto Univ. 42 (2006), 589--603.

\bibitem{KL09}
M. Khovanov, A. Lauda, A diagrammatic approach to categorification of quantum
groups I, Represent. Theory 13, 309-347 (2009)

\bibitem{KP18} T. Kimura, V. Pestun, Fractional quiver W-algebras, Lett. Math. Phys. 108, 2425-2451(2018).


\bibitem{L18}
J. Lurie, Spectral algebraic geometry,
  https://www.math.ias.edu/~lurie/papers/SAG-rootfile.pdf (2018).
  
\bibitem{L93}  G. Lusztig, Introduction to Quantum Groups, Birkh\"auser, Boston, 1993

\bibitem{M22} L.  C. Mihalcea, H. Naruse,  C. Su, Left Demazure?Lusztig Operators on Equivariant (Quantum) Cohomology and
K-Theory, International Mathematics Research Notices, Vol. 2022, 12096-12147.
  
\bibitem{N01} H. Nakajima, Quiver varieties and finite-dimensional representations of quantum affine algebras, 
J. Amer. Math. Soc. 14 (2001), 145-238.

\bibitem{N03} H. Nakajima, $t$-analogs of $q$-characters of quantum affine algebras of type $A_n$, $D_n$,
in : Combinatorial and geometric representation theory (Seoul, 2001), 
Contemp. Math., 141-160 (2003).


\bibitem{NW23} 
H. Nakajima, A. Weekes, 
Coulomb branches of quiver gauge theories with symmetrizers, J. of Euro. Math. Soc. (to appear).





\bibitem{VV10} M. Varagnolo, E. Vasserot, Double affine Hecke algebras and affine flag manifolds, I, in :
Affine Flag Manifolds and Pricipal Bundles, Trends Math. Birkh\"auser-Springer, 233-289 (2010).

\bibitem{VV11} M. Varagnolo, E. Vasserot, Canonical bases and KLR-algebras, J. Reine Angew. Math. 659 (2011), 
67-100.

\bibitem{VV23} M. Varagnolo, E. Vasserot, Non symmetric quantum loop groups and K-theory, Ann. Sci. \'Ec. Norm. Sup\'er. (to appear), arXiv:2308.01809.

\bibitem{VV26} M. Varagnolo, E. Vasserot, Shifted affine quantum groups and crystals, in preparation.

\bibitem{W12} B. Webster, Canonical bases and higher representation theory, Compos. Math. 151 (2015), 121-166.

\bibitem{W17} B. Webster, Knot invariants and higher representation theory, Mem. Amer. Math. Soc. 250 (2017).

\bibitem{W19} B. Webster, Weighted Khovanov-Lauda-Rouquier algebras, Documenta Mathematica, 24 (2019) 209-250.

\bibitem{W19b} B. Webster, Gelfand-Tsetlin modules in the Coulomb context, arXiv:1904.05415.

\bibitem{We19} A. Weekes, Generators for Coulomb branches of quiver Gauge theories, arXiv:1903.07734.

\end{thebibliography}
\end{document}